\documentclass[11pt]{article}

\usepackage{a4wide}
\usepackage{tikz,overpic,graphicx,caption,subcaption}
\usetikzlibrary{decorations.markings}

\usepackage{amsmath,amsthm,amssymb}
\usepackage{mathtools}
\usepackage{bbm}

\DeclareMathOperator{\Tr}{\mathrm {Tr}}

\theoremstyle{plain}
\newtheorem{theorem}{Theorem}[section] 
\newtheorem{lemma}[theorem]{Lemma}
\newtheorem{proposition}[theorem]{Proposition}
\newtheorem{definition}[theorem]{Definition}
\newtheorem{corollary}[theorem]{Corollary}
\newtheorem{conjecture}[theorem]{Conjecture}

\theoremstyle{definition}
\newtheorem{remark}[theorem]{Remark} 

\newcommand{\eastdimer} {[very thick] ++(0,0.1)--++(0,0.8)}
\newcommand{\westdimer} {[very thick]++(0,0.1)--++(0,0.8)}
\newcommand{\southdimer} {[very thick]++(0.1,0)--++(0.8,0)}
\newcommand{\northdimer} {[very thick] ++(0.1,0)--++(0.8,0)}

\newcommand{\eastdomino} {++(-.5,-.5)--++(0,2)--++(1,0)--++(0,-2)--++(-1,0)}
\newcommand{\westdomino} {++(-.5,-.5)--++(0,2)--++(1,0)--++(0,-2)--++(-1,0)}
\newcommand{\southdomino}  {++(-.5,-.5)--++(0,1)--++(2,0)--++(0,-1)--++(-2,0)}
\newcommand{\northdomino}  {++(-.5,-.5)--++(0,1)--++(2,0)--++(0,-1)--++(-2,0)}

\newcommand{\eastdominoplus} {[very thick,red]++(-.5,0)--++(1,1)}
\newcommand{\westdominoplus} {[very thick,red]++(-.5,1)--++(1,-1)}
\newcommand{\southdominoplus}  {[very thick,red]++(-.5,0)--++(2,0)}

\newcommand{\eps}{\varepsilon}

\renewcommand{\Re}{\mathop{\mathrm{Re}}}
\renewcommand{\Im}{\mathop{\mathrm{Im}}}

\title{Biased $2 \times 2$ periodic Aztec diamond and an elliptic curve}
\author{Alexei Borodin\footnote{Department of Mathematics, Massachusetts Institute of Technology, 77 Massachusetts Ave., Cambridge, MA 02139, USA. E-mail: borodin@math.mit.edu}  \and Maurice Duits\footnote{Department of Mathematics, Royal Institute of Technology, Lindstedtsvägen 25, SE 10044, Stockholm Sweden. E-mail: duits@kth.se}}
\date{}

\begin{document}

\maketitle
\begin{abstract}
    We study random domino tilings of the Aztec diamond with a biased $2 \times 2$ periodic weight function and associate a  linear flow on an elliptic curve  to this model. Our main result is a double integral formula for the correlation kernel, in which the integrand is expressed in terms of this flow. For special choices of parameters the flow is periodic, and this allows us to perform a saddle point analysis for the correlation kernel.  In these cases we compute the local correlations in the smooth disordered (or gaseous) region. The special example in which the flow has period six is worked out in more detail, and we show that in that case the boundary of the rough disordered region is an algebraic curve of degree eight.
\end{abstract}

\setcounter{tocdepth}{1}
\tableofcontents

\section{Introduction}

Domino tilings of the Aztec diamond, originally introduced in \cite{EKL},  form a popular arena for various interesting phenomena of  integrable probability. A domino tiling of the Aztec diamond can be viewed as a perfect matching, also called dimer configuration,  on the Aztec diamond graph. This is a particular bipartite subgraph of the square lattice (cf. Figure \ref{fig:dimers}). By putting weights on the edges of the Aztec diamond graph, one defines a probability measure on the set of all perfect matchings, and hence all domino tilings, by saying that the probability of having a particular matching  is proportional to the product of the weights of the edges in that matching. In recent years, several works have appeared on domino tilings of the Aztec diamond where the weights are \emph{doubly periodic}. That is, they are periodic in two independent directions, and we will use the notation $k \times \ell$ to indicate that they are $k$-periodic in one direction and $\ell$-periodic in the other. In this paper, we will study a particular example of a $2 \times 2$ doubly periodic weighting that is a generalization  of the model  studied in \cite{BCJ1,BCJ2,CJ,CY,DK,JM}. The difference is that we introduce an extra parameter that induces a bias towards horizontal dominos, and we refer to this model as the \emph{biased $2 \times 2$ periodic Aztec diamond.}  The model considered in  \cite{BCJ1,BCJ2,CJ,CY,DK,JM} will be referred to as the unbiased $2 \times 2$ periodic Aztec diamond.

Doubly periodic weightings lead to  rich behavior when the size of the Aztec diamond becomes large. The Aztec diamond can be  partitioned into three regions:  frozen, rough disordered (or liquid) and  smooth disordered (or gaseous). They are characterized by the different local limiting Gibbs measures  that one expects in these regions \cite{KOS}. The difference between the smooth and disordered regions is that the dimer-dimer correlations decay exponentially with their distance in the smooth disordered region and polynomially in the rough region. The three regions are clearly visible in Figure \ref{fig:sample} where we have plotted a sample of our model for a large Aztec diamond.

 From general arguments, that go back to \cite{Kast}, we know that the correlation functions in our model are determinantal. In order to perform a rigorous asymptotic study, one aims to find an expression for the correlation kernel that is amenable for an asymptotic analysis. For the unbiased $2\times 2$ periodic Aztec diamond, a double integral representation  was  first found in \cite{CJ} (more precisely, they were able to find the inverse Kasteleyn matrix \cite{Kast}).   Based on this expression, the boundary between the smooth and rough disordered region has been studied extensively in  \cite{BCJ1,BCJ2,JM}. Unfortunately,  it is not obvious how the expression in \cite{CJ}  extends to the biased generalization  that we  consider in this paper.  Instead, we  follow the  approach of \cite{BD}. 

In \cite{BD} the authors studied  probability measures on particle configurations given by products of minors of block Toeplitz matrices. The biased $ 2\times 2$ periodic  Aztec diamond can be viewed as a special case of such a probability measure. The main result of \cite{BD} is an explicit  double integral formula for the correlation  kernel, provided one can find a Wiener-Hopf factorization for the product of the matrix-valued symbols for the block Toeplitz matrices.  That Wiener-Hopf factorization can in principle be found by carrying out an iterative procedure, in which the total number of iterations is of the same order as the size of the Aztec diamond.  In certain special cases, such as the unbiased $2 \times 2$ periodic  Aztec diamond \cite{BD} and a family of $2 \times k$ periodic weights  \cite{B}, the procedure is periodic, and after  a few iterations one ends up with the same parameters that one started with. This means that the Wiener-Hopf factorization has a rather simple form, and after inserting that expression in the double integral formula one obtains a suitable starting point for a saddle point analysis \cite{BD,DK}. However, generically,  the iteration in \cite{BD} is too complicated to find simple expressions for the Wiener-Hopf factorization, and  other ideas are needed.

The biased  $2\times 2$ periodic Aztec diamond is the simplest doubly periodic case in which it is difficult to trace the flow in \cite{BD}. Our first main result is that the Wiener-Hopf factorization can alternatively be computed by following a linear flow  on an explicit elliptic curve. This flow is rather simple and consists of repeatedly adding a particular point on the elliptic curve. For generic parameters, one expects the flow to be ergodic, but for special choices the flow will be periodic. We will identify a few explicit examples of these periodic cases, and perform an asymptotic study in the smooth disordered region for the general periodic situation. 

The reason why the iterative procedure in our case is linearizable on an elliptic curve can be traced back to \cite{MV}. 
In that work it was shown how an isospectral flow on certain quadratic matrix polynomials, obtained by repeatedly moving the right divisor of the polynomial with a given spectrum to the left side, is linearizable on the Jacobian (or the Prym variety) of the corresponding spectral curve. The main goal of \cite{MV} was to describe the dynamics of certain discrete analogs of classical integrable systems in terms of Abelian functions. Some of the key ideas used in that work had previously originated in constructing the so-called \emph{finite gap solutions} of integrable PDEs, see their book-length exposition \cite{BBEIM} with historic notes and references therein. The matrix case, which was most relevant for \cite{MV}, had been originally developed in \cite{D1,D2,I,K1,K2}. 

While our situation does not exactly fit into the formalism of \cite{MV}, similar ideas do apply, and they led us to the linearization. We hope that they will also help with studying more general tiling models.   

To conclude, let us mention that in \cite{DK} it was shown that the double periodicity leads to matrix-valued orthogonal polynomials. For the unbiased $2 \times 2$ periodic Aztec diamond, these matrix-valued orthogonal polynomials have a particularly simple structure. Somewhat surprisingly, they even have explicit integral expressions that lead to an explicit double integral representation for the correlation kernel. The expression in \cite{DK} was re-derived in \cite{BD}.  For the biased model it is interesting to see what the flow on the elliptic curve implies for the matrix-valued orthogonal polynomials, and if explicit expressions can be given in general and/or for the periodic case. Furthermore, it is interesting to compare our results with  \cite{Bert}, in which matrix orthogonal polynomials were studied using an abelianization based on the spectral curve for the orthogonality weight. 

\begin{figure}[t]
\begin{center}
    \includegraphics[scale=.4,angle=45]{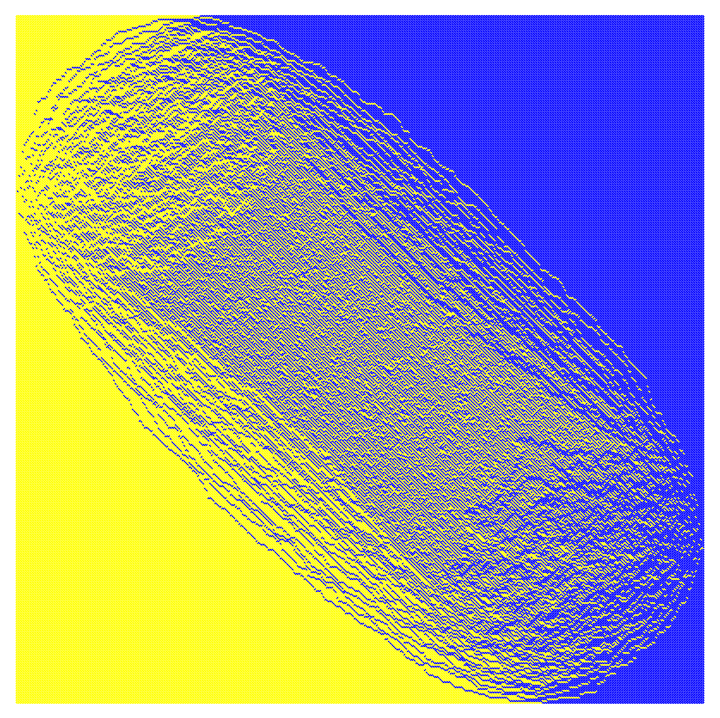}
\end{center}
\caption{A sampling of the biased doubly periodic for a large Aztec diamond. The West and South dominos are colored yellow, and the North and East dominos are colored blue. The three different regions are clearly visible, with the smooth disordered region in the middle, surrounded by the rough disordered region and  frozen regions in the corners. }
\label{fig:sample}
\end{figure}

\section{Preliminaries}

In this section we will introduce the dimer model that we are interested in, discuss several standard different representations from the literature and recall the determinantal structure of the correlation functions for a corresponding point processes. In our discussion we repeat necessary definitions from earlier works, in particular of \cite{BD,DK,Jdet,J}, and we will make specific references to those works at several places to refer the reader for more details. We refer to \cite{G} for a general introduction to random tilings. 

\subsection{A doubly periodic dimer model} \label{sec:dimer}

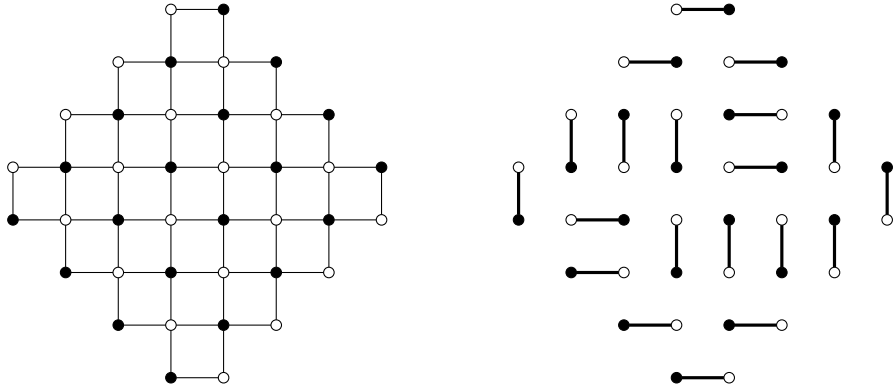
\begin{figure}[t]
    \begin{center}
\begin{tikzpicture}[scale=0.7]
    
    \foreach \l in {0,1,2,3} \foreach \k in {0,1,2,3,4} \filldraw (\k+\l,\k-\l) circle(.1);
    \foreach \l in {0,1,2,3,4} \foreach \k in {0,1,2,3} \draw (\k+\l,\k-\l+1) circle(.1);
    \foreach \l in {0,1,2,3} \foreach \k in {0,1,2,3} {
        \draw (\k+\l+.1,\k-\l+1)-\draw (\k+\l+.9,\k-\l+1); 
        \draw (\k+\l,\k-\l+.9)-\draw (\k+\l,\k-\l+.1);
        \draw (\k+\l+.1,\k-\l)-\draw (\k+\l+.9,\k-\l); 
        \draw (\k+\l+1,\k-\l+.9)-\draw (\k+\l+1,\k-\l+.1);}
\end{tikzpicture} \qquad \qquad 
\begin{tikzpicture}[scale=0.7]
    \foreach \l in {0,1,2,3} \foreach \k in {0,1,2,3,4} \filldraw (\k+\l,\k-\l) circle(.1);
    \foreach \l in {0,1,2,3,4} \foreach \k in {0,1,2,3} \draw (\k+\l,\k-\l+1) circle(.1);
   \draw (0,0) \eastdimer;
   \draw (1,1)  \eastdimer;
   \draw (3,1)  \eastdimer;
   \draw (3,-1) \eastdimer;
   \draw (5,-1)  \eastdimer;

   \draw (1,0)  \southdimer;
   \draw (1,-1)  \southdimer;
   \draw (2,-2)  \southdimer;
   \draw (3,-3) \southdimer;
   \draw (4,-2) \southdimer;
   \draw (4,2)  \southdimer;
   
   \draw (6,-1)  \westdimer;
   \draw (6,1) \westdimer;
   \draw (7,0)  \westdimer;
   \draw (2,1)  \westdimer;
   \draw (4,-1)  \westdimer;
   \draw (6,1)  \westdimer;

   \draw (3,4)  \northdimer;
   \draw (2,3)  \northdimer;
   \draw (4,3)  \northdimer;
   \draw (4,1) \northdimer;

\end{tikzpicture}
\end{center}
\caption{The left picture is the bipartite graph $\mathcal G_N$, with $N=4$, and the right picture is  a perfect mathching of $\mathcal G_N$. }\label{fig:dimers}
\end{figure}

For $N \in \mathbb N$ define a bipartite graph $\mathcal G_N=(\mathcal B_N \cup \mathcal W_N,\mathcal E_N)$,  with black vertices 
$$
    \mathcal B_N=\left\{\left(\tfrac12-N +j+k, -\tfrac 12-j +k\right)  \mid j=0,\ldots,N-1, \quad k=0,\ldots,N\right\},
$$ 
and white vertices 
$$
    \mathcal W_N=\left\{\left(\tfrac12-N +j+k, \tfrac 12-j +k\right)  \mid j=0,\ldots,N, \quad k=0,\ldots,N-1\right\},
$$
and with edges $\mathcal E_N$  between  black and white vertices that are neighbors in the lattice graph (i.e., that have a difference of $ (\pm 1,0)$ or $ (0,\pm 1)$). This gives the graph on the left of Figure \ref{fig:dimers}. The picture on the right of Figure \ref{fig:dimers} is a perfect matching of this bipartite graph, also   called a dimer configuration. A dimer model is a probability distribution on the space of all perfect matchings $\mathcal M$ of this graph $\mathcal G_N$ such that the probability of a particular matching $M$ is proportional to   
$$
 \mathbb P(M)\sim \prod_{e \in M} w(e),
$$
where $w:\mathcal E \to (0,\infty)$ is a weight function. 
\begin{figure}[t]
\begin{center}
    \begin{tikzpicture}[scale=1]
        \filldraw (0,0) circle(.1);
        \draw (1,0) circle(.1);
        \draw (-1,0) circle(.1);
        \draw (0,1) circle(.1);
        \draw (0,-1) circle(.1); 
        \draw (0,.1)--(0,.9);
        \draw (0,-.1)--(0,-.9);
        \draw (0.1,0)--(0.9,0);
        \draw (-.1,0)--(-.9,0);
        \node at (0.4,0.5)  {$ a \alpha$};
        \node  at (-.2,-0.5) {$a$};
        \node  at (-0.5,0.2) {$1 $};
        \node  at (0.5,-0.2){$ \alpha$};
        \node[text width=7cm]  at (0.3,-2.4){If the coordinates of  the center black vertex  are $(-\tfrac 12-N+m,\tfrac 12 +m-2j )$ with $m$ even};
    \end{tikzpicture}
    \begin{tikzpicture}[scale=1]
        \filldraw (0,0) circle(.1);
        \draw (1,0) circle(.1); 
        \draw (-1,0) circle(.1);
        \draw (0,1) circle(.1);
        \draw (0,-1) circle(.1);
        \draw (0,.1)--(0,.9);
        \draw (0,-.1)--(0,-.9);
        \draw (0.1,0)--(0.9,0);
        \draw (-.1,0)--(-.9,0);
        \node at (0.2,0.5)  {$\frac a \alpha$};
        \node  at (-.2,-0.5) {$a$};
        \node  at (-0.5,0.2) {$1 $};
        \node  at (0.5,-0.3){$\tfrac 1 \alpha$};
        \node[text width=7cm]  at (0.3,-2.4) {If the coordinates of  the center black vertex  are $(-\tfrac 12-N+m,\tfrac 12 +m-2j )$ with $m$ odd};
        \draw (2,0);
    \end{tikzpicture}
   
\end{center}
\caption{The weights on the edges of $\mathcal G_N$.} \label{fig:dimer_weights}
\end{figure}

In this paper, we will consider the weight functions defined as is shown in Figure \ref{fig:dimer_weights}. There are two parameters $\alpha,a \in (0,1]$. The vertical and horizontal edges with a black vertex on top or on the right all have weights $a$ and $1$, respectively. For vertical edges with a black vertex on the bottom and horizontal edges with a black vertex on the left, the weight depends on the coordinates of that black vertex. These weights are given by $a \alpha$ and $\alpha$, or by $a/\alpha$ and $1/\alpha$, depending on the coordinates of the black vertex in that edge.  If the vertical coordinate of that vertex is $\frac{1}{2}+k$  for an even $k$, then the weights are $a \alpha$ and $\alpha$. If the vertical coordinate of that vertex is  $\frac{1}{2}+k$ for an odd $k$,  then we have the same weight, but with $\alpha$ replaced by $1/\alpha$.  The distribution of the weights is thus two periodic in two different directions;  edges whose coordinates differ by a multiple of $(2,2)$ or $(2,-2)$ have the same weight. 

Note that only the parameter $\alpha$ is responsible for the double periodicity.  Indeed, for $\alpha=1$ the weights no longer depend on the position of the black vertex in the center in Figure \ref{fig:dimer_weights}.  We will be particularly interested in the doubly periodic situation and thus in the case $0<\alpha<1$.   The effect of the extra parameter $a$ is that all the vertical edges are given an extra factor $a$. If $0<a<1$, this makes them less likely, and the model is biased towards horizontal edges.    As we will see, adding this parameter has a profound effect on the integrable structure of this model. Moreover, we will see  that the special case $a=1$, studied by several authors \cite{BD,CY,DK}, is a very particular point.

An alternative way of representing matchings is by drawing dominos. Indeed, each matching is equivalent to a domino tiling by drawing rectangles around the matched vertices as is shown in Figure \ref{fig:from_dimer_to_domino}. The dominos tile a planar region known as the Aztec diamond. We distinguish between four different types of dominos called the West, East, North and South dominos. The West dominos are the vertical dominos with a black vertex on the bottom, the East dominos are the vertical dominos with a black vertex on the top, the North dominos are the horizontal dominos with a black vertex on the right  and, finally, the South dominos are the horizontal dominos with a black vertex on the left. In Figure \ref{fig:from_dimer_to_domino} these four types of dominos are the furthermost ones in the corresponding corners.

\begin{figure}
\begin{center}
    \begin{tikzpicture}[scale=0.7]
        \foreach \l in {0,1,2,3} \foreach \k in {0,1,2,3,4} \filldraw (\k+\l,\k-\l) circle(.1);
        \foreach \l in {0,1,2,3,4} \foreach \k in {0,1,2,3} \draw (\k+\l,\k-\l+1) circle(.1);
       \draw (0,0) \eastdimer;
       \draw (1,1)  \eastdimer;
       \draw (3,1)  \eastdimer;
       \draw (3,-1) \eastdimer;
       \draw (5,-1)  \eastdimer;

       \draw (1,0)  \southdimer;
       \draw (1,-1)  \southdimer;
       \draw (2,-2)  \southdimer;
       \draw (3,-3) \southdimer;
       \draw (4,-2) \southdimer;
       \draw (4,2)  \southdimer;
       
       \draw (6,-1)  \westdimer;
       \draw (6,1) \westdimer;
       \draw (7,0)  \westdimer;
       \draw (2,1)  \westdimer;
       \draw (4,-1)  \westdimer;
       \draw (6,1)  \westdimer;

       \draw (3,4)  \northdimer;
       \draw (2,3)  \northdimer;
       \draw (4,3)  \northdimer;
       \draw (4,1) \northdimer;

    \end{tikzpicture}
    \qquad \quad
    \begin{tikzpicture}[scale=0.7]
        
        \foreach \l in {0,1,2,3} \foreach \k in {0,1,2,3,4} \filldraw (\k+\l,\k-\l) circle(.1);
        \foreach \l in {0,1,2,3,4} \foreach \k in {0,1,2,3} \draw (\k+\l,\k-\l+1) circle(.1);
       \draw (0,0) \eastdomino;
       \draw (1,1)  \eastdomino;
       \draw (5,-1)  \eastdomino;
       \draw (3,1)  \eastdomino;
       \draw (3,-1) \eastdomino;
      
       \draw (1,-1)  \southdomino;
       \draw (2,-2)  \southdomino;
       \draw (3,-3) \southdomino;
       \draw (4,-2) \southdomino;
       \draw (4,2)  \southdomino;      
       
       \draw (6,-1)  \westdomino;
       \draw (7,0)  \westdomino;
       \draw (2,1)  \westdomino;
       \draw (4,-1)  \westdomino;
       \draw (6,1)  \westdomino;

       \draw (3,4)  \northdomino;
       \draw (2,3)  \northdomino;
       \draw (4,3)  \northdomino;
       \draw (4,1) \northdomino;
       \draw (1,0)  \southdomino;

    \end{tikzpicture}
\end{center}
\caption{The right picture is the domino representation of the dimer configuration on the left}
\label{fig:from_dimer_to_domino}
\end{figure}
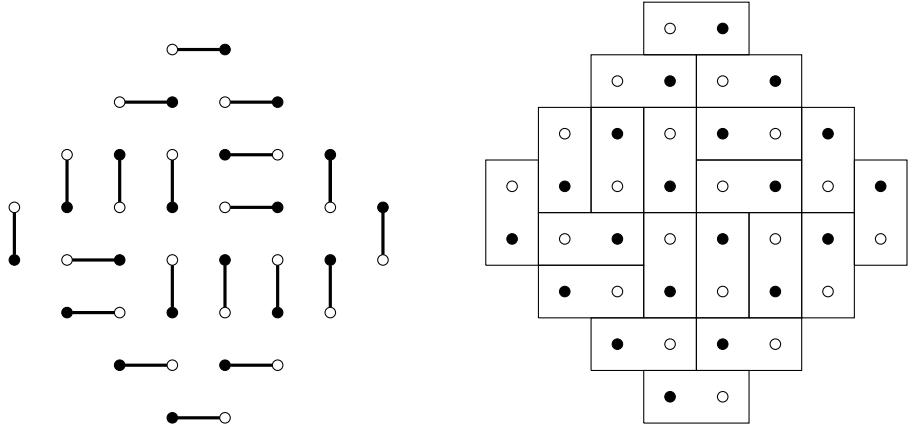
Note that the weighting that we will consider is such that all North dominos have weight $1$ and all East dominos have weight $a$. The weight of a West domino is either $a \alpha$ if the vertical coordinate of the lower left corner is even, or $a/\alpha$ if that coordinate is odd. Similarly, the weight of a South domino is either $\alpha$ if the vertical coordinate of the lower left corner is even, and $1/\alpha$ if that coordinate is odd. For small $a>0$ we expect to see more South and North dominos, as the West and East domino have small weight.

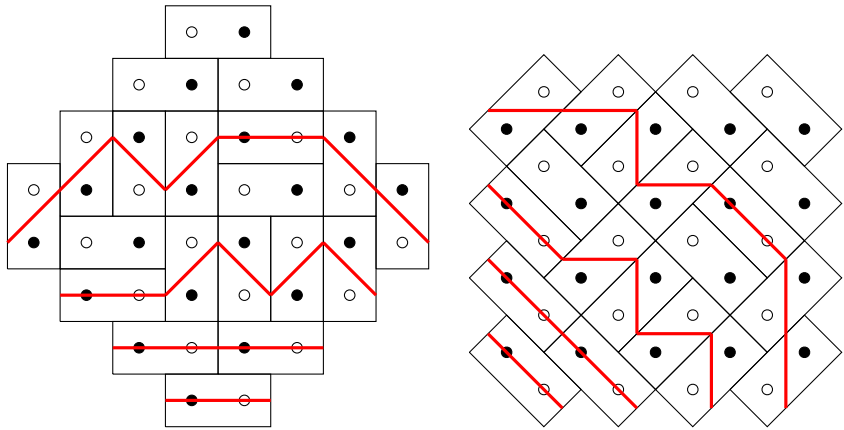
\begin{figure}[t]
    \begin{center}
        \begin{tikzpicture}[scale=0.7]
            
            \foreach \l in {0,1,2,3} \foreach \k in {0,1,2,3,4} \filldraw (\k+\l,\k-\l) circle(.1);
            \foreach \l in {0,1,2,3,4} \foreach \k in {0,1,2,3} \draw (\k+\l,\k-\l+1) circle(.1);
           \draw (0,0) \eastdomino;
           \draw (1,1)  \eastdomino;
           \draw (5,-1)  \eastdomino;
           \draw (3,1)  \eastdomino;
           \draw (3,-1) \eastdomino;
          
           \draw (1,-1)  \southdomino;
           \draw (2,-2)  \southdomino;
           \draw (3,-3) \southdomino;
           \draw (4,-2) \southdomino;
           \draw (4,2)  \southdomino;      
           
           \draw (6,-1)  \westdomino;
           \draw (7,0)  \westdomino;
           \draw (2,1)  \westdomino;
           \draw (4,-1)  \westdomino;
           \draw (6,1)  \westdomino;
    
           \draw (3,4)  \northdomino;
           \draw (2,3)  \northdomino;
           \draw (4,3)  \northdomino;
           \draw (4,1) \northdomino;
           \draw (1,0)  \southdomino;
    
           \draw[red,very thick] (0,0) \eastdominoplus;
           \draw[red,very thick] (1,1)  \eastdominoplus;
           \draw[red,very thick] (5,-1)  \eastdominoplus;
           \draw[red,very thick] (3,1)  \eastdominoplus;
           \draw[red,very thick] (3,-1) \eastdominoplus;
    
           \draw[red,very thick] (1,-1)  \southdominoplus;
           \draw[red,very thick] (2,-2)  \southdominoplus;
           \draw[red,very thick] (3,-3) \southdominoplus;
           \draw[red,very thick] (4,-2) \southdominoplus;
           \draw[red,very thick] (4,2)  \southdominoplus;
             
           \draw[red,very thick] (6,-1)  \westdominoplus;
           \draw[red,very thick] (7,0)  \westdominoplus;
           \draw[red,very thick]  (2,1)  \westdominoplus;
           \draw[red,very thick]  (4,-1)  \westdominoplus;
           \draw[red,very thick]  (6,1)  \westdominoplus;
        \end{tikzpicture}
        \quad \begin{tikzpicture}[rotate=-45,scale=0.7]
            
            \foreach \l in {0,1,2,3} \foreach \k in {0,1,2,3,4} \filldraw (\k+\l,\k-\l) circle(.1);
            \foreach \l in {0,1,2,3,4} \foreach \k in {0,1,2,3} \draw (\k+\l,\k-\l+1) circle(.1);
           \draw (0,0) \eastdomino;
           \draw (1,1)  \eastdomino;
           \draw (5,-1)  \eastdomino;
           \draw (3,1)  \eastdomino;
           \draw (3,-1) \eastdomino;
          
           \draw (1,-1)  \southdomino;
           \draw (2,-2)  \southdomino;
           \draw (3,-3) \southdomino;
           \draw (4,-2) \southdomino;
           \draw (4,2)  \southdomino;      
           
           \draw (6,-1)  \westdomino;
           \draw (7,0)  \westdomino;
           \draw (2,1)  \westdomino;
           \draw (4,-1)  \westdomino;
           \draw (6,1)  \westdomino;
    
           \draw (3,4)  \northdomino;
           \draw (2,3)  \northdomino;
           \draw (4,3)  \northdomino;
           \draw (4,1) \northdomino;
           \draw (1,0)  \southdomino;
    
           \draw[red,very thick] (0,0) \eastdominoplus;
           \draw[red,very thick] (1,1)  \eastdominoplus;
           \draw[red,very thick] (5,-1)  \eastdominoplus;
           \draw[red,very thick] (3,1)  \eastdominoplus;
           \draw[red,very thick] (3,-1) \eastdominoplus;
    
           \draw[red,very thick] (1,-1)  \southdominoplus;
           \draw[red,very thick] (2,-2)  \southdominoplus;
           \draw[red,very thick] (3,-3) \southdominoplus;
           \draw[red,very thick] (4,-2) \southdominoplus;
           \draw[red,very thick] (4,2)  \southdominoplus;
             
           \draw[red,very thick] (6,-1)  \westdominoplus;
           \draw[red,very thick] (7,0)  \westdominoplus;
           \draw[red,very thick]  (2,1)  \westdominoplus;
           \draw[red,very thick]  (4,-1)  \westdominoplus;
           \draw[red,very thick]  (6,1)  \westdominoplus;
        \end{tikzpicture}
    \end{center}
    \caption{The DR paths on a domino tiling.}\label{fig:DNR}
    \end{figure}

\subsection{Non-intersecting  paths}

A useful alternative representation, that is easily obtained from the dominos, is the representation by DR-paths \cite{J,Jnon,Stan}. By drawing an upright path across each West domino, a down-right path across each East domino, a horizontal across each South domino and nothing on a North domino, we obtain the picture given in Figure \ref{fig:DNR}. There are four paths leaving from the lower left side of the Aztec diamond and ending at the lower right side. The paths also cannot intersect. Clearly, the paths determine the location of the East, West and South dominos, and therewith the entire tiling. One can therefore represent each dimer configuration with a collection of non-intersecting paths. 

Instead of looking directly at the DR paths, however, we will consider closely related interpretation in terms of non-intersecting paths on a different graph. The reason for this is two-fold. First, the DR paths are rather uneven in length. The bottom path is much shorter than the top path. The second reason is that it turns out to be useful  to add paths so that we have an infinite number of them. The auxiliary paths will have no effect on the model, but will give a very convenient integrable structure. 

We start with a directed graph $\mathcal G_p=(\{0,1, \ldots, 2N\} \times \mathbb Z, \mathcal E_p)$ where we draw edges between the following vertices (we use the index $p$  in $\mathcal G_p$ and $\mathcal E_p$ to distinguish this graph from the bipartite graph in the dimer representation):
$$(2j,k)\to (2j+1,k), \qquad (2j,k)\to (2j+1,k+1),$$
$$(2j+1,k)\to (2j+2,k), \qquad (2j+2,k+1)\to (2j+2,k).$$
A part of the graph is shown in  Figure \ref{fig:paths}. 
\begin{figure}[t]
\begin{center}
    \begin{tikzpicture}[scale=0.6
        ,decoration={markings, mark= at position 0.5 with {\arrow{stealth}}}] 
        \foreach \k in {0,1,2,3} \foreach \ell in {0,1,2,...,8} {
            \draw [postaction={decorate}] (2*\k,\ell)--(2*\k+1,\ell+1);
            \draw [postaction={decorate}] (2*\k,\ell)--(2*\k+1,\ell);
            \filldraw (2*\k,\ell) circle(.05);
            \filldraw (2*\k+1,\ell) circle(.05);
            \filldraw (2*\k+1,\ell+1) circle(.05); }
            \foreach \k in {0,1,2} \foreach \ell in {0,1,2,...,8} {
                \draw [postaction={decorate}] (2*\k+2,\ell+1)--(2*\k+2,\ell);
                \draw [postaction={decorate}] (2*\k+1,\ell)--(2*\k+2,\ell);
                \filldraw (2*\k+2,\ell+1) circle(.05);
                \filldraw (2*\k+2,\ell) circle(.05);
                \filldraw (2*\k+1,\ell+1) circle(.05);
                };
                \foreach \k in {0,1,2,3} \foreach \ell in {0,1,2,...,8} {
                    \draw [postaction={decorate}] (2*\k+2,\ell+1)--(2*\k+2,\ell);
                    \draw [postaction={decorate}] (2*\k+1,\ell)--(2*\k+2,\ell);
                    \filldraw (2*\k+2,\ell+1) circle(.05);
                    \filldraw (2*\k+2,\ell) circle(.05);
                    \filldraw (2*\k+1,\ell+1) circle(.05);
                    }
                    \node  at (-1,6) {$(0,0)$};
                    \node  at (-1.5,1) {$(0,-M)$};
                    \node  at (9.5,6) {$(2N,0)$};
                    \node  at (10,1) {$(2N,-M)$};
            \foreach \j in {0,1,2,3,4,5} \filldraw (0,6-\j) circle(.15);
            \foreach \j in {0,1,2,3,4,5} \filldraw (8,6-\j) circle(.15);
    \end{tikzpicture}
 \qquad 
\begin{tikzpicture}[scale=0.6
    ,decoration={markings, mark= at position 0.5 with {\arrow{stealth}}}] 
    \foreach \k in {0,1,2,3} \foreach \ell in {0,1,2,...,8} {
        \draw [postaction={decorate}] (2*\k,\ell)--(2*\k+1,\ell+1);
        \draw [postaction={decorate}] (2*\k,\ell)--(2*\k+1,\ell);
        \filldraw (2*\k,\ell) circle(.05);
        \filldraw (2*\k+1,\ell) circle(.05);
        \filldraw (2*\k+1,\ell+1) circle(.05); }
        \foreach \k in {0,1,2} \foreach \ell in {0,1,2,...,8} {
            \draw [postaction={decorate}] (2*\k+2,\ell+1)--(2*\k+2,\ell);
            \draw [postaction={decorate}] (2*\k+1,\ell)--(2*\k+2,\ell);
            \filldraw (2*\k+2,\ell+1) circle(.05);
            \filldraw (2*\k+2,\ell) circle(.05);
            \filldraw (2*\k+1,\ell+1) circle(.05);
            };
            \foreach \k in {0,1,2,3} \foreach \ell in {0,1,2,...,8} {
                \draw [postaction={decorate}] (2*\k+2,\ell+1)--(2*\k+2,\ell);
                \draw [postaction={decorate}] (2*\k+1,\ell)--(2*\k+2,\ell);
                \filldraw (2*\k+2,\ell+1) circle(.05);
                \filldraw (2*\k+2,\ell) circle(.05);
                \filldraw (2*\k+1,\ell+1) circle(.05);
                }
              
        \foreach \j in {0,1,2,3,4,5} \filldraw (0,6-\j) circle(.15);
        \foreach \j in {0,1,2,3,4,5} \filldraw (8,6-\j) circle(.15);
        \draw[line width=.1cm](0,6)--(1,7)--(2,7)--(3,8)--(4,8)--(5,8)--(6,8)--(6,7)--(7,7)--(8,7)--(8,6);
        \filldraw (1,7) circle (.15);
        \filldraw (2,7) circle (.15);
        \filldraw (3,8) circle (.15);
        \filldraw (4,8) circle (.15);
        \filldraw (5,8) circle (.15);
        \filldraw (6,7) circle (.15);
        \filldraw (7,7) circle (.15);
        \draw[line width=.1cm](0,5)--(1,6)--(2,6)--(3,7)--(4,7)--(4,6)--(6,6)--(6,5)--(7,5)--(8,5)--(8,5);
        \filldraw (1,6) circle (.15);
        \filldraw (2,6) circle (.15);
        \filldraw (3,7) circle (.15);
        \filldraw (4,6) circle (.15);
        \filldraw (5,6) circle (.15);
        \filldraw (6,5) circle (.15);
        \filldraw (7,5) circle (.15);
        \draw[line width=.1cm](0,4)--(1,5)--(2,5)--(3,5)--(4,5)--(4,4)--(5,4)--(6,4)--(7,4)--(8,4)--(8,4);
        \filldraw (1,5) circle (.15);
        \filldraw (2,5) circle (.15);
        \filldraw (3,5) circle (.15);
        \filldraw (4,4) circle (.15);
        \filldraw (5,4) circle (.15);
        \filldraw (6,4) circle (.15);
        \filldraw (7,4) circle (.15);
        \draw[line width=.1cm] (0,3)--(1,4)--(2,4)--(2,3)--(8,3);
        \filldraw (1,4) circle (.15);
        \filldraw (2,3) circle (.15);
        \filldraw (3,3) circle (.15);
        \filldraw  (4,3) circle (.15);
        \filldraw (5,3) circle (.15);
        \filldraw (6,3) circle (.15);
        \filldraw (7,3) circle (.15);
        \draw[line width=.1cm] (0,2)--(8,2);
        \filldraw (1,2) circle (.15);
        \filldraw (2,2) circle (.15);
        \filldraw (3,2) circle (.15);
        \filldraw (4,2) circle (.15);
        \filldraw (5,2) circle (.15);
        \filldraw (6,2) circle (.15);
        \filldraw(7,2) circle (.15);
        \draw[line width=.1cm] (0,1)--(6,1)--(6,0)--(7,1)--(8,1);
        \filldraw (1,1) circle (.15);
        \filldraw (2,1) circle (.15);
        \filldraw (3,1) circle (.15);
        \filldraw (4,1) circle (.15);
        \filldraw (5,1) circle (.15);
        \filldraw (6,0) circle (.15);
        \filldraw(7,1) circle (.15);
        
\end{tikzpicture}
\caption{The left figure shows the underlying graph $\mathcal G_p$. The right figure shows the graph $\mathcal G_p$ and a collection of non-intersecting paths starting in $(0,-j)$ and ending in $(2N,-j)$ for $j=0,\ldots M$, with $N=4$ and $M=5$.}
\label{fig:paths}
\end{center}
\end{figure}
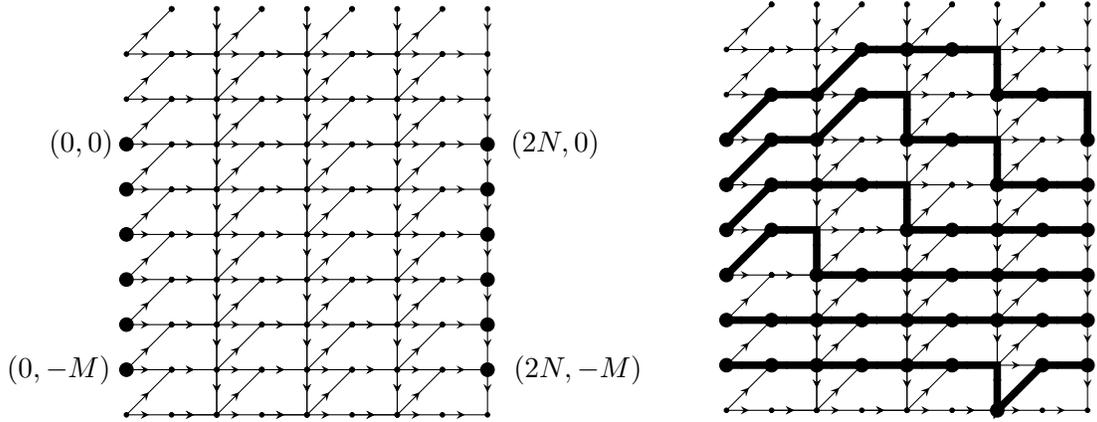
We then fix starting points $(0,-j)$ for $j=0,\ldots, M$, and endpoints $(2N,-j)$ for $j=0,\ldots, M$ and consider collections of paths in the directed graph that connect the starting points with the endpoints, such that no paths have a vertex in common  (i.e., they never intersect).  

Note that if $M\geq 2N-2$  only the $N$ top paths and the $N-1$ bottom paths are non-trivial, but any path in between is, due to the non-intersecting condition, necessarily a straight line.  In fact, even the top $N$ and bottom $N-1$ paths have parts where they are necessarily straight lines. Indeed, in the region between the lines $(m,-N+m/2)$ and $(m,-M+N+m/2)$ for $m=0,\ldots,2N$, all the paths are necessarily horizontal. 

The connection with the dimer models is the following: If we remove all paths below the line $(m,-N+m/2)$ then the configuration that remains is equivalent to the DR paths for the domino tilings of Aztec diamond. Indeed, by further removing all horizontal parts $(m,u)\to (m+1,u)$ for odd $m$ and concatenating the result, we obtain the picture in the  middle of Figure \ref{fig:LGV_to_DNR}. The coordinate transform $(m,u)\mapsto (m,u-m)$ maps the middle picture to  the DR-paths shown on the right of Figure \ref{fig:LGV_to_DNR}.

\begin{figure}[t]
    \begin{center}
        \begin{tikzpicture}[scale=0.6
            ,decoration={markings, mark= at position 0.5 with {\arrow{stealth}}}] 
            \foreach \k in {0,1,2,3} \foreach \ell in {0,1,2,...,8} {
                \draw [postaction={decorate}] (2*\k,\ell)--(2*\k+1,\ell+1);
                \draw [postaction={decorate}] (2*\k,\ell)--(2*\k+1,\ell);
                \filldraw (2*\k,\ell) circle(.05);
                \filldraw (2*\k+1,\ell) circle(.05);
                \filldraw (2*\k+1,\ell+1) circle(.05); }
                \foreach \k in {0,1,2} \foreach \ell in {0,1,2,...,8} {
                    \draw [postaction={decorate}] (2*\k+2,\ell+1)--(2*\k+2,\ell);
                    \draw [postaction={decorate}] (2*\k+1,\ell)--(2*\k+2,\ell);
                    \filldraw (2*\k+2,\ell+1) circle(.05);
                    \filldraw (2*\k+2,\ell) circle(.05);
                    \filldraw (2*\k+1,\ell+1) circle(.05);
                    };
                    \foreach \k in {0,1,2,3} \foreach \ell in {0,1,2,...,8} {
                        \draw [postaction={decorate}] (2*\k+2,\ell+1)--(2*\k+2,\ell);
                        \draw [postaction={decorate}] (2*\k+1,\ell)--(2*\k+2,\ell);
                        \filldraw (2*\k+2,\ell+1) circle(.05);
                        \filldraw (2*\k+2,\ell) circle(.05);
                        \filldraw (2*\k+1,\ell+1) circle(.05);
                        }
                       
                \foreach \j in {0,1,2,3,4,5} \filldraw (0,6-\j) circle(.15);
                \foreach \j in {0,1,2,3,4,5} \filldraw (8,6-\j) circle(.15);
                \draw[line width=.1cm](0,6)--(1,7)--(2,7)--(3,8)--(4,8)--(5,8)--(6,8)--(6,7)--(7,7)--(8,7)--(8,6);
                \filldraw (1,7) circle (.15);
                \filldraw (2,7) circle (.15);
                \filldraw (3,8) circle (.15);
                \filldraw (4,8) circle (.15);
                \filldraw (5,8) circle (.15);
                \filldraw (6,7) circle (.15);
                \filldraw (7,7) circle (.15);
                \draw[line width=.1cm](0,5)--(1,6)--(2,6)--(3,7)--(4,7)--(4,6)--(6,6)--(6,5)--(7,5)--(8,5)--(8,5);
                \filldraw (1,6) circle (.15);
                \filldraw (2,6) circle (.15);
                \filldraw (3,7) circle (.15);
                \filldraw (4,6) circle (.15);
                \filldraw (5,6) circle (.15);
                \filldraw (6,5) circle (.15);
                \filldraw (7,5) circle (.15);
                \draw[line width=.1cm](0,4)--(1,5)--(2,5)--(3,5)--(4,5)--(4,4)--(5,4)--(6,4)--(7,4)--(8,4)--(8,4);
                \filldraw (1,5) circle (.15);
                \filldraw (2,5) circle (.15);
                \filldraw (3,5) circle (.15);
                \filldraw (4,4) circle (.15);
                \filldraw (5,4) circle (.15);
                \filldraw (6,4) circle (.15);
                \filldraw (7,4) circle (.15);
                \draw[line width=.1cm] (0,3)--(1,4)--(2,4)--(2,3)--(8,3);
                \filldraw (1,4) circle (.15);
                \filldraw (2,3) circle (.15);
                \filldraw (3,3) circle (.15);
                \filldraw  (4,3) circle (.15);
                \filldraw (5,3) circle (.15);
                \filldraw (6,3) circle (.15);
                \filldraw (7,3) circle (.15);
                \draw[line width=.1cm] (0,2)--(8,2);
                \filldraw (1,2) circle (.15);
                \filldraw (2,2) circle (.15);
                \filldraw (3,2) circle (.15);
                \filldraw (4,2) circle (.15);
                \filldraw (5,2) circle (.15);
                \filldraw (6,2) circle (.15);
                \filldraw(7,2) circle (.15);
                \draw[line width=.1cm] (0,1)--(6,1)--(6,0)--(7,1)--(8,1);
                \filldraw (1,1) circle (.15);
                \filldraw (2,1) circle (.15);
                \filldraw (3,1) circle (.15);
                \filldraw (4,1) circle (.15);
                \filldraw (5,1) circle (.15);
                \filldraw (6,0) circle (.15);
                \filldraw(7,1) circle (.15);
                \filldraw[white,opacity=.7] (0,2.9)--(2,2.9)--(8,6)--(9,6)--(9,-0.2)--(-1,-0.2)--(-1,3)--(0,2.9);
                \foreach \j in {1,2,3,4} \filldraw (2*\j,\j+2) circle(.15);
                \filldraw (0,3) circle(.15);
        \end{tikzpicture}
        \qquad  \qquad
        \begin{tikzpicture}[scale=0.7
            ,decoration={markings, mark= at position 0.5 with {\arrow{stealth}}}] 

            \foreach \k in {0,1,2,3} \foreach \ell in {1,2,...,8} {
                \draw [postaction={decorate}] (\k,\ell)--(\k+1,\ell+1);
                \draw [postaction={decorate}] (\k,\ell)--(\k+1,\ell);
                \filldraw (\k,\ell) circle(.05);
                \filldraw (\k+1,\ell) circle(.05);
                \filldraw (\k+1,\ell+1) circle(.05); }
                \foreach \k in {0,1,2} \foreach \ell in {2,3,...,8} {
                    \draw [postaction={decorate}] (\k+2,\ell+1)--(\k+2,\ell);
                    
                    \filldraw (\k+2,\ell+1) circle(.05);
                    \filldraw (\k+2,\ell) circle(.05);
                    \filldraw (\k+1,\ell+1) circle(.05);
                    };
                    \foreach \k in {-1,0,1,2} \foreach \ell in {1,2,...,7} {
                        \draw [postaction={decorate}] (\k+2,\ell+1)--(\k+2,\ell);
                        \draw [postaction={decorate}] (\k+1,\ell)--(\k+2,\ell);
                        \filldraw (\k+2,\ell+1) circle(.05);
                        \filldraw (\k+2,\ell) circle(.05);
                        \filldraw (\k+1,\ell+1) circle(.05);
                        }
           \draw[line width=.1cm](0,6)--(1,7)--(2,8)--(3,8)--(3,7)--(4,7)--(4,6);
           \draw[line width=.1cm](0,5)--(1,6)--(2,7)--(2,6)--(3,6)--(3,5);         
                \draw[line width=.1cm](0,4)--(1,5)--(2,5)--(2,4);             
                \draw[line width=.1cm] (0,3)--(1,4)--(1,3);                          
        \end{tikzpicture}
        \qquad \qquad 
        \begin{tikzpicture}[scale=0.7
            ,decoration={markings, mark= at position 0.5 with {\arrow{stealth}}}] 

            \foreach \k in {0,1,2,3} \foreach \ell in {1,2,...,7} {
                \draw [postaction={decorate}] (\k,\ell)--(\k+1,\ell);
                \draw [postaction={decorate}] (\k,\ell)--(\k+1,\ell-1);
                \filldraw (\k,\ell-1) circle(.05);
                \filldraw (\k+1,\ell-1) circle(.05);
                \filldraw (\k+1,\ell) circle(.05); }
                \foreach \k in {-1,0,1,2} \foreach \ell in {2,3,...,8} {
                    \draw [postaction={decorate}] (\k+2,\ell-1)--(\k+2,\ell-2);
                    \filldraw (\k+2,\ell-1) circle(.05);
                    \filldraw (\k+2,\ell-2) circle(.05);
                    \filldraw (\k+1,\ell-1) circle(.05);
                   
                    };

                    \foreach \ell in {0,1,2,3} {
                \draw [postaction={decorate}] (\ell,0)--(\ell+1,0);}
                   
           \draw[line width=.1cm](0,6)--(1,7-1)--(2,8-2)--(3,8-3)--(3,7-3)--(4,7-4)--(4,6-4);
           \draw[line width=.1cm](0,5)--(1,6-1)--(2,7-2)--(2,6-2)--(3,6-3)--(3,5-3);         
                \draw[line width=.1cm](0,4)--(1,5-1)--(2,5-2)--(2,4-2);             
                \draw[line width=.1cm] (0,3)--(1,4-1)--(1,3-1);                          
        \end{tikzpicture}
        \caption{From the non-intersecting paths on the graph $\mathcal G_p$ to the DR paths. The middle picture is obtained by removing the horizontal steps  from the paths and the graph $\mathcal G_p$. In the  second transformation $(m, u) \mapsto (m, u-m)$ we obtain the rotated DR paths.} \label{fig:LGV_to_DNR}
    \end{center}
\end{figure}
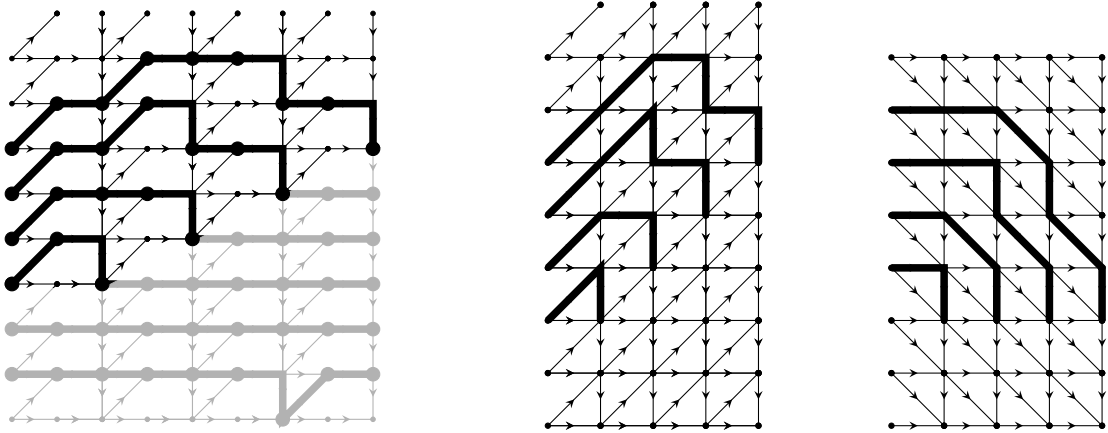

The next step is to put a probability measure on the collection of non-intersecting paths that is consistent with the dimer model from Section \ref{sec:dimer}. To make the correspondence, we note that each up-right diagonal edge in the graph $\mathcal G_p$ corresponds to a West domino, each vertical edge to an East domino, and each horizontal edge (after removing the auxiliary horizontal edges at the odd steps) corresponds to a South domino.  A careful comparison with the weights for the dimer models leads us to assigning weights to the underlying directed graph as follows: the horizontal edges $(m,u)\to (m+1,u)$ for odd $m$ are auxiliary and have weight 1, the vertical edges correspond to  East dominos and have weight $a$, the horizontal edges $(m,u)\to (m+1,u)$ for even $m$ correspond to  South dominos and have weight $\alpha$ if $u$ is even and weight $1/\alpha$ if $u$ is odd, and, finally, the up-right edges $(m,u) \to (m+1,u+1)$ for $m$ even have weight $a\alpha$ if $u$ is even and weight $a/\alpha$ if $u$ is odd. This is also represented in the following  finite weighted graph that is the building block for the rest of $\mathcal G_p$: 

\begin{center}
    \begin{tikzpicture}
        [decoration={markings, 
        mark= at position 0.5 with {\arrow{stealth}}}] 
        \draw [postaction={decorate}] (0,0)--(1,0);
        \draw [postaction={decorate}] (0,0)--(1,1);
        \draw [postaction={decorate}] (0,1)--(1,1);
        \draw [postaction={decorate}] (0,1)--(1,2);
        \draw [postaction={decorate}] (1,0)--(2,0);
        \draw [postaction={decorate}] (1,1)--(2,1);
        \draw [postaction={decorate}] (1,2)--(2,2);
        \draw [postaction={decorate}] (2,1)--(2,0);
        \draw [postaction={decorate}] (2,2)--(2,1);
        \filldraw (0,0) circle(.05);
        \filldraw (0,1) circle(.05);
        \filldraw (1,0) circle(.05);
        \filldraw (1,1) circle(.05);
        \filldraw (1,2) circle(.05);
        \filldraw (2,0) circle(.05);
        \filldraw (2,1) circle(.05);
        \filldraw (2,2) circle(.05);

        \node at (0.8,1.3) {$\alpha $};
        \node  at (0.2,1.7) {$a \alpha$};
        \node  at (0.5,-.3)  {$ \tfrac 1 \alpha$};
        \node  at (0,.4)  {$ \tfrac a \alpha$};
        \node  at (1.5,-.3){$1$};
        \node  at (1.5,.7){$1$};
        \node  at (1.5,1.7){$1$};
        \node  at (2.2,.5){$a$};
        \node  at (2.2,1.5){$a$};
        \node  at (-1,-.25){$(2j,2k-1)$};
        \node  at (-.85,1){$(2j,2k)$};
        \node  at (3.5,-.25){$(2j+2,2k-1)$};
        \node  at (3.2,1){$(2j+2,2k)$};
    \end{tikzpicture}
\end{center}

Then the probability of having  a particular configuration of non-intersecting paths is proportional to the product of the  weights of all the edges  in the corresponding dimer/domino configuration.

\subsection{A determinantal point process}
Let us now  assign a point process to the above collections of paths.  We place points on these paths by taking the lowest possible vertex on each vertical section (including those of length $0$), as indicated in the right panel of Figure \ref{fig:paths},
$$
    (m,u^{j}_m) \quad \text{ for } j=1,\ldots,M, \quad m=0,\ldots, 2N,
$$
where $u_0^{j}=u_{2N}^{j}=-j+1$ and $M \geq N$.  Since the top $N$ paths uniquely determine the dimer configuration, so do the points $(m,u_{m}^j)$. Further, our probability measure also turns the set of points with coordinates $(m,u_m^j )$ into a point process on $\{0,1,\ldots,2 N\}\times \mathbb Z$. 

We stress that we are  only interested in the points $(m,u^j_m)$ with $j \leq N-m/2+1$, as it is those  that determine the tiling. The other points are auxiliary and only added for convenience. Indeed, by a theorem of Lindström-Gessel-Viennot (see, e.g., \cite{Lind,GV}) the probability of a given point configuration is proportional to
$$
\prod_{m=1}^{2N} \det T_m(u_{m-1}^j,u^k_{m})_{j,k=1}^M,
$$
where $T_m$ are the transition matrices defined by
$$
    \left[T_{m}(2 k_1-\ell_1,2 k_2-\ell_2)\right]_{\ell_1,\ell_2=0}^1= 
    \frac{1}{2\pi i}\oint A_{m}(z) \frac{dz}{z^{k_2-k_1+1}},
$$
for $k_1,k_2 \in \mathbb Z$, 
and  $A_m(z)$ given by 
$$
 A_m(z)=
    \begin{dcases}
        A_e(z),& \text{if } m \text{ is even},\\
        A_o(z),& \text{if } m \text{ is odd},
    \end{dcases}
$$
with
\begin{equation}\label{eq:symbolsoddeven}
A_o(z)=\begin{pmatrix}
    \alpha & a \alpha z\\
    \frac{a}{\alpha}  & \frac{1}{\alpha}
\end{pmatrix}, \qquad 
A_e(z)=\frac{1}{1-a^2/z}
        \begin{pmatrix}
            1 &  a\\
            \frac{a}{z} & 1
       \end{pmatrix}.
    \end{equation}
We will also use the notation
$$
    A(z)=\prod_{m=1}^{2N}A_m(z).
$$
By the Eynard-Mehta theorem (see, e.g., \cite{EM}), the point process is determinantal, meaning that there exists a kernel 
\begin{equation}
    K_{N,M}: \left(\{0,1,\ldots,2 N\}\times \mathbb Z\right) \times\left( \{0,1,\ldots,2 N\}\times\mathbb Z\right)  \to \mathbb C,
\end{equation} 
such that, for any $(m_k,u_k ) \in \{0,\ldots, 2N\} \times \mathbb Z$ and $k=1,\ldots,n$,
$$  
   \mathbb P(\textrm{there are points at } (m_k,u_k ), \quad k=1,\ldots,n) = \det \left[K_{N,M}((m_j,u_j),(m_k,u_k))\right]_{j,k=1}^n.
$$
Now we recall that we are only interested in the top $N$ paths, and thus we will restrict $u_j$ to be in $\{-N+1,\ldots,0\}$. Then  the marginal densities are independent of $M$ as long as $M$ is sufficiently large and 
$$
 K_{N,M}((m_1,u_1),(m_2,u_2)) = \lim_{M\to \infty}  K_{N,M}((m_1,u_1),(m_2,u_2))= K_{N}((m_1,u_1),(m_2,u_2)).
$$
In \cite{BD} a double integral formula for the correlation kernel $K_N$ was given. That formula involves a solution to a Wiener-Hopf factorization. 
\begin{proposition}  \label{prop:BD}  \cite[Theorem 3.1]{BD} Suppose that we can find a factorization
$$
    A(z)=A_-(z)A_+(z) 
$$
with $2\times 2$ matrices $A_\pm(z)$ such that 
\begin{enumerate}
    \item $A_+^{\pm 1} (z)$ are analytic in $|z|<1$ and continuous in $|z|\leq 1$,
    \item $A^{\pm 1 } _-(z)$ are analytic in $|z|>1$ and continuous in $|z|\geq 1$,
    \item $A_-(z)\sim \begin{pmatrix} 1 & 0 \\ 0 & 1\end{pmatrix} $ as $ z\to \infty$.
\end{enumerate}
Then the kernel $K_{N,M}$ has the pointwise limit $K_N$ as $M\to \infty$ given by 
\begin{multline} \label{eq:correlationkernelBD_Intro}
    \left[K_N((m,2 x-j),(m',2x'-j'))\right]_{j,j=0}^1= -\frac{\mathbbm{1}_{m'<m}}{2 \pi i} \int_{|z|=1} \prod_{j=m'+1}^{m} A_j(z) \frac{dz}{z^{x-x'+1}}\\
    +\frac{1}{(2 \pi i)^2}
    \oint_{|w|=\rho_1} \oint_{|z|=\rho_2}\left(\prod_{j=m'+1}^{2 N} A_j(w)\right) A_+(w)^{-1} A_-(z)^{-1}  \left(\prod_{j=1}^{m} A_j(z)\right) \frac{w^{x'}}{z^{x+1}}\frac{dz dw}{z-w},
\end{multline}
where $|a|^2<\rho_1<\rho_2<1/|a|^2$, $\mathbbm{1}_{m'<m}=1$ if $m'<m$ and $0$ otherwise,  and the integration contours are positively oriented. 
\end{proposition}
\begin{remark}
    Proposition \ref{prop:BD} is only part of Theorem 3.1 in \cite{BD}. Indeed, the kernel in \eqref{eq:correlationkernelBD_Intro} is called $K_{top}$ in \cite{BD}. We note that here we already shifted coordinates compared to \cite{BD}. Also, in the formulation of Theorem 3.1 in \cite{BD} one needs a second factorization $A(z)= \tilde A_-(z) \tilde A_+(z)$. However, all that is needed for Proposition \ref{prop:BD} is the existence of such a factorization, and that is guaranteed by Theorem 4.8  in \cite{BD}. 
\end{remark}
\subsection{The Wiener-Hopf factorization} \label{sec:intro_wh}
The  question remains how to find a Wiener-Hopf factorization that is explicit enough to be able to use \eqref{eq:correlationkernelBD_Intro} as a starting point for asymptotic analysis.  The idea for finding a Wiener-Hopf factorization is simple (see also \cite[Section 4.4]{BD}). 
Write 
$$
    A(z)= \frac{1}{(1-a^2/z)^N}(P(z))^{N},
$$
where 
$$
  P(z)= \begin{pmatrix}
    \alpha & a \alpha z\\
    \frac{a}{\alpha}   & \frac{1}{\alpha}
\end{pmatrix} 
\begin{pmatrix}
    1 &  a\\
    \frac{a}{z} & 1
\end{pmatrix}.
$$
Then in the first step we look for a Wiener-Hopf factorization of the form
$$
        P(z)=P_{0,-}(z)P_{0,+}(z),
$$
and then write 
$$
    (P(z))^N=P_{0,-}(z) (P_1(z))^{N-1} P_{0,+}(z),
$$
where 
$$
    P_1(z)=P_{0,+}(z)P_{0,-}(z).
$$
Next, we compute a factorization for $P_1(z)=P_{1,-}(z)P_{1,+}(z)$ and set $P_2(z)=P_{1,+}(z)P_{1,-}(z)$. At each step in the procedure we thus construct a new matrix valued function $P_{k+1}(z)=P_{k,+}(z)P_{k,-}(z)$ constructed by  switching the order of the  Wiener-Hopf factorization 
\begin{equation}
    \label{eq:WH_intro}    
    P_{k}(z)=P_{k,-}(z)P_{k,+}(z).
\end{equation}
The result is that we find a Wiener-Hopf factorization for $A(z)$ of the form 
$$
  A(z)=
  \frac{1}{(1-a^2/z)^N}\left(P_{0,-}(z) \cdots P_{N-1,-}(z)\right)
    \left(P_{N-1,+}(z) \cdots P_{0,+}(z)\right).
$$
An important point is that this procedure defines a flow 
$$ P_0(z) \mapsto P_1(z) \mapsto P_2(z) \mapsto \ldots$$
and to obtain explicit representations for the correlation kernel in \eqref{eq:correlationkernelBD_Intro} we need to have a sufficiently detailed description of this flow.  As was pointed out in \cite[Section 4]{BD}, there is a general procedure to capture this flow. Generically, the description in \cite{BD} of the flow is rather difficult to control, but for specific values it can be written explicitly.  Indeed, for $a=1$, the double integral formula of \cite{DK} could be reproduced. See also \cite{B} for other cases where it was tractable. It is important to note that in the cases of both \cite{DK} and \cite{B}, the flow was periodic, which is  of great help, in particular for asymptotic analysis. For the model that we consider in this paper, however, it appears difficult to control this flow for $a<1$, and the point of the present paper is to give an alternative more tangible description. We will show that the flow is equivalent to translations on an explicit elliptic curve. This will also help us to track other choices of parameters for which the flow is periodic. 

\section{Main results}
We now present our main results. All proofs will be postponed to Sections 5. 

\subsection{An elliptic curve} \label{sec:examplestorsion}
Consider an elliptic curve $\mathcal E$ (over $\mathbb R)$ defined by the equation 
\begin{equation} \label{eq:elliptic_aztec}
y^2=x^2+\frac{4 x(x-a^2)(x-1/a^2)}{(a+1/a)^2(\alpha+1/\alpha)^2},
\end{equation}
where $\alpha$ and $a$ are the parameters from the dimer model in Section \ref{sec:dimer}. One easily verifies that the curve crosses the $x$-axis  precisely three times, once at the origin and at two further  intersection points in  $(-\infty,0)$. The elliptic curve has therefore two connected components, and one of those, denoted by $\mathcal E_-$, lies entirely in the left half plane. Note also that $(0,0)$, $(a^2,a^2)$ and $(a^{-2},a^{-2})$ are the intersection points of the curve with the line $y=x$. The point $(a^{-2},a^{-2})$ will be of particular interest to us. 

\begin{figure}[t] 
    \begin{center}
    \begin{overpic}[scale=.6]{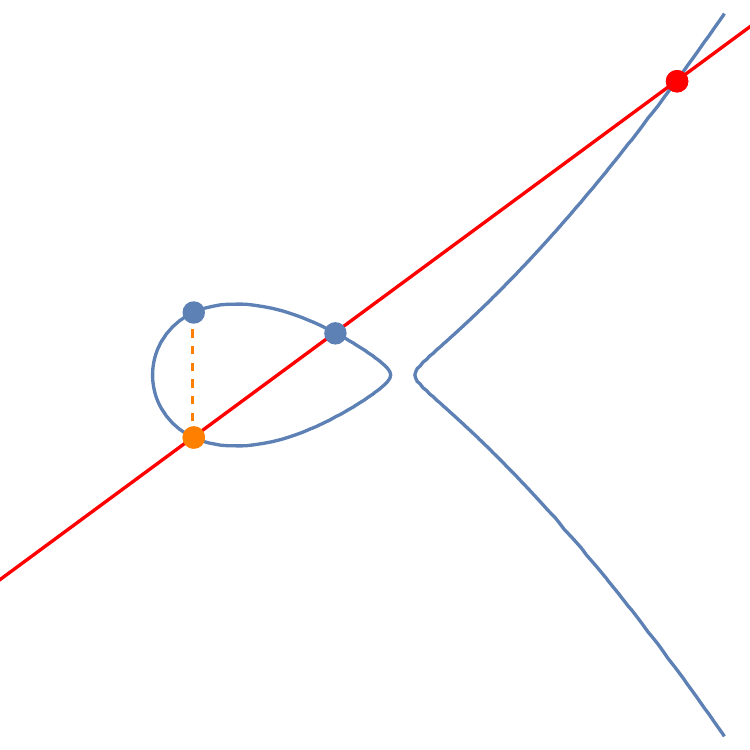}
        \put (12,65)  {\small{$(x_{j+1},y_{j+1})$}}
        \put (90,80)  {\small{$(a^{-2},a^{-2})$}}
        \put (37,62)  {\small{$(x_j,y_j)$}}
        \put (42,37)  {\small{$\mathcal E_-$}}
    \end{overpic}
\end{center}
\caption{The flow on the elliptic curve. At each step we add the point $(a^{-2},a^{-2})$. This can be geometrically represented by drawing a straight line through  $(a^{-2},a^{-2})$ and $(x_j,y_j)$. This line intersects the curve at a unique third point in $\mathcal E_-$. The point $(x_{j+1},y_{j+1})$ is then obtained from the intersection point by flipping the sign of the second coordinate.}
\label{fig:addition_on_curve}
\end{figure}
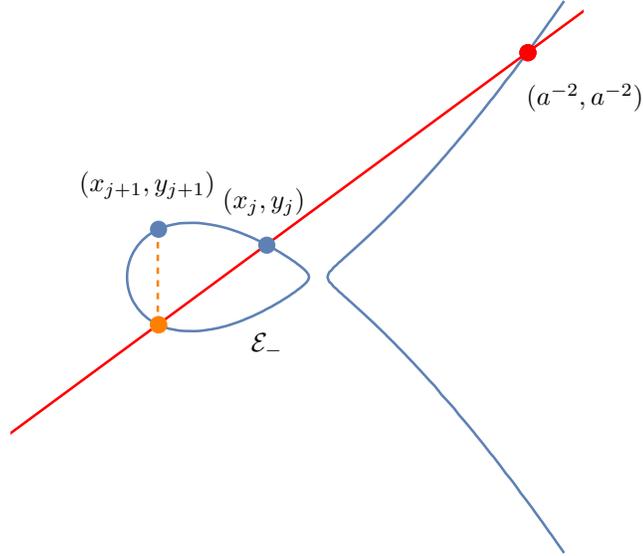 

It is well known that an elliptic curve carries an Abelian group structure, and we can add points on the curve. The point at infinity serves as the identity.  We will be interested in a linear flow on the curve that is constructed by repeatedly adding the point $(a^{-2},a^{-2})$  starting from the initial parameters $(x_0,y_0)=(-1,-\frac{1-\alpha^2}{1+\alpha^2})$. That is, we consider the flow
$$
    \begin{dcases}
        (x_{j+1},y_{j+1})= \sigma(x_j,y_j),\\
        (x_0,y_0)=\left(-1,-\frac{1-\alpha^2}{1+\alpha^2}\right),
    \end{dcases}
$$
where 
$$
    \sigma(x,y)=(x,y)+(a^{-2},a^{-2}),
$$
and $+$ represents addition on the elliptic curve.  The flow can be nicely illustrated  by the geometric description of the group addition on the curve. Starting from $(x_j,y_j)$ we compute $(x_{j+1},y_{j+1})$ as follows: the straight line passing though $(x_j,y_j)$ and $(a^{-2},a^{-2})$ passes through a third point and $(x_{j+1},y_{j+1})$ is the reflection of that point with respect to the $x$ axis (in other words, we flip the sign of the $y$-coordinate). See also Figure \ref{fig:addition_on_curve}. It can happen that the line through $(x_j,y_j)$ and $(a^{-2},a^{-2})$  is tangent to $\mathcal E_-$ at point $(x_j,y_j)$. In that case, $(x_{j+1},y_{j+1})$ is just the reflection of the $(x_j,y_j)$  with respect to the $x$ axis.  Note that the initial point $(x_0,y_0)$ lies on the oval $\mathcal E_-$, and from the geometric interpretation it is easy to see that every point $(x_{j},y_j)$  is on the oval $\mathcal E_-$. 

Our first main result  is that this flow uniquely determines the correlations for the biased Aztec diamond as described in Sections 2.1--2.4 above. But before we explain that, we first discuss properties of the flow that will be of interest to us. For generic choices of the parameters one can expect the flow to be ergodic on $\mathcal E_-$,  but for certain special parameters  $(a^{-2},a^{-2})$ will be a torsion point. In those cases the flow is periodic. This distinction  has important implications  for our asymptotic analysis of the tiling model. We will therefore discuss a few examples in which  $(a^{-2},a^{-2})$ is a torsion point.

First, if we assume that $\alpha=1,$ then our dimer model is an example of a Schur process \cite{OR}, and we know that simpler double integral formulas for its correlation kernel can be given. This should mean that our flow has a particularly simple structure. Indeed, for $\alpha=1$, the oval $\mathcal E_-$ reduces to a singleton $\mathcal E_-=\{(-1,0)\}$, and the flow is constant. This can also be seen directly, from the fact that the two factors in the definition of $P(z)$ commute. 

The second case of interest is the unbiased case where $a=1.$ In that case, $(a^{-2},a^{-2})=(a^2,a^2)$, and the elliptic curve  is tangent to the line $y=x$ at that point. For general $a>0$ we have the relation $(a^2,a^2)+(a^{-2},a^{-2})=(0,0)$ and thus, for $a=1$, we have $2(a^{-2},a^{-2})=(0,0)$. It is also clear that $(0,0)$ is a point of order $2$, and thus $(a^{-2},a^{-2})$ is of order 4. This implies that our flow is periodic and returns to its initial point after $4$ steps. For completeness, we compute the flow explicitly:
\begin{equation}
    \left(-1,-\frac{1-\alpha^2}{1+\alpha^2}\right)\mapsto \left(-\alpha^2,0\right)\mapsto  \left(-1,\frac{1-\alpha^2}{1+\alpha^2}\right) \mapsto \left(-\frac{1}{\alpha^2},
    0\right) \mapsto \left(-1,-\frac{1-\alpha^2}{1+\alpha^2}\right).
\end{equation}
See the left panel of Figure \ref{fig:curveOrder6} for an illustration.

The next example we would like to discuss is that of an order six torsion point. This happens when 
\begin{equation} \label{eq:condsix}
  a^2=\frac{\alpha}{\alpha^2+\alpha+1}.
\end{equation}
The flow on the elliptic curve is given by:
\begin{multline} \label{eq:flowsix}
    \left(
        -1, - \frac{1-\alpha^2}{1+\alpha^2}
    \right)
    \mapsto
    \left(
        -\alpha^2,  \frac{-\alpha^2+\alpha^3}{1+\alpha}
    \right)
    \mapsto
    \left(
        -\alpha^2,  \frac{\alpha^2-\alpha^3}{1+\alpha}
    \right)
    \mapsto
    \left(
        -1,  \frac{1-\alpha^2}{1+\alpha^2}
    \right)\\
    \mapsto
    \left(
        -\frac{1}{\alpha^2},  \frac{1-\alpha}{\alpha^2+\alpha^3}
    \right)
    \mapsto
    \left(
        -\frac{1}{\alpha^2},  -\frac{1-\alpha}{\alpha^2+\alpha^3}
    \right)
    \mapsto
    \left(
        -1, - \frac{1-\alpha^2}{1+\alpha^2}
    \right).
\end{multline}
Indeed, after six steps we have returned to our initial point. This case is illustrated on the right panel of Figure \ref{fig:curveOrder6}.

We found the relation \eqref{eq:condsix} by computing the division polynomial of order $6$ and requiring that $(a^{-2},a^{-2})$ is a zero of this polynomial. In fact, this provides a recipe for deriving relations between $a$ and $\alpha$ such that   $(a^{-2},a^{-2})$ is a torsion point of order $m$. We recall the notion of division polynomials in Appendix \ref{sec:divisionpolynomials} and provide such relations for $m=4,5,6,7,8$. 

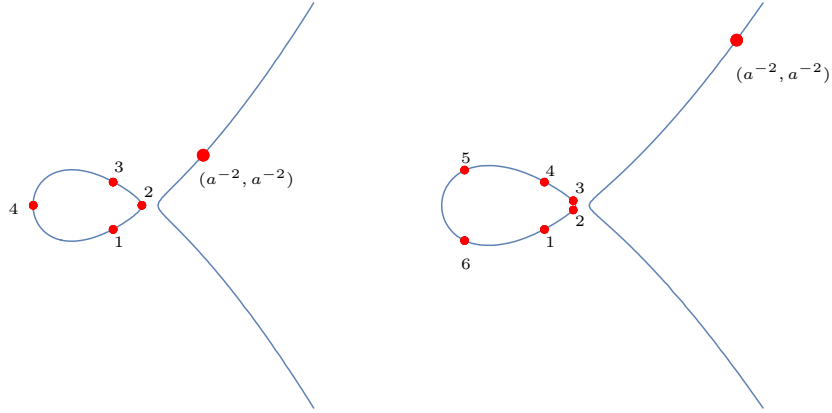
\begin{figure}[t]
    \begin{center}
        \begin{overpic}[scale=0.35]{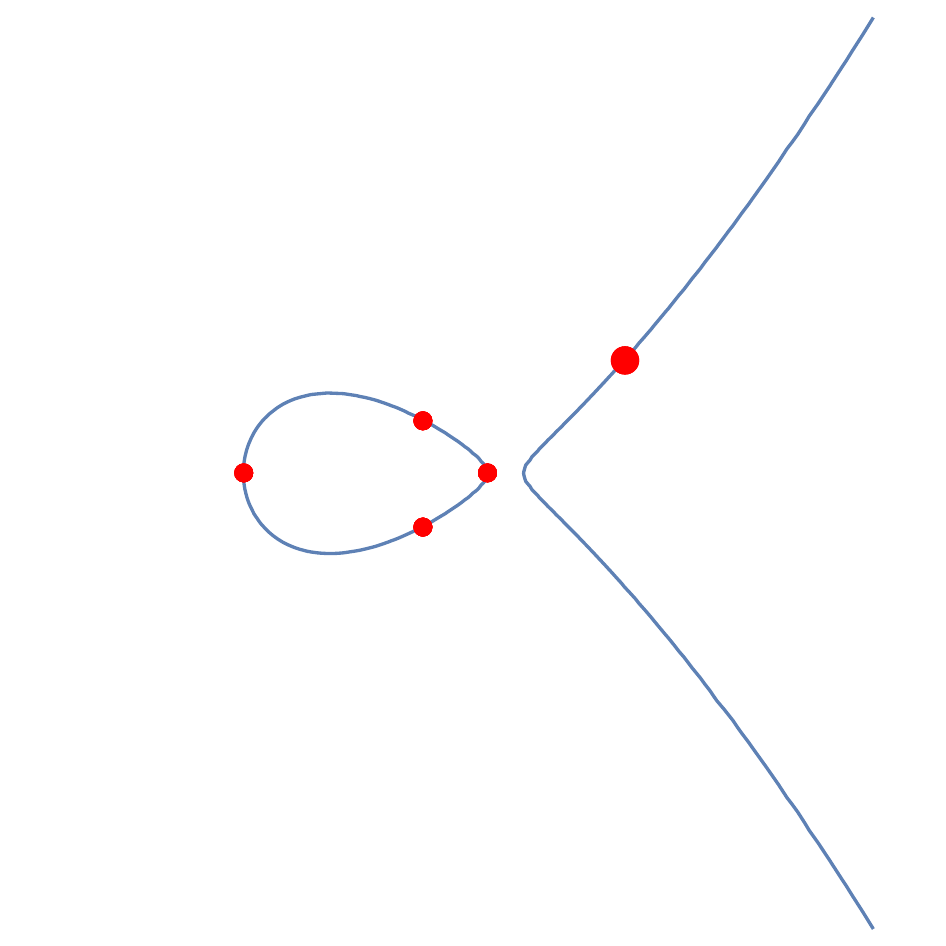}
            \put (45,40) {\tiny{1}}
            \put (52,52) {\tiny{2}}
            \put (45,58) {\tiny{3}}
            \put (20,48) {\tiny{4}}
            \put (65,55) {\tiny{$(a^{-2},a^{-2})$}}
        \end{overpic}
        \begin{overpic}[scale=0.35]{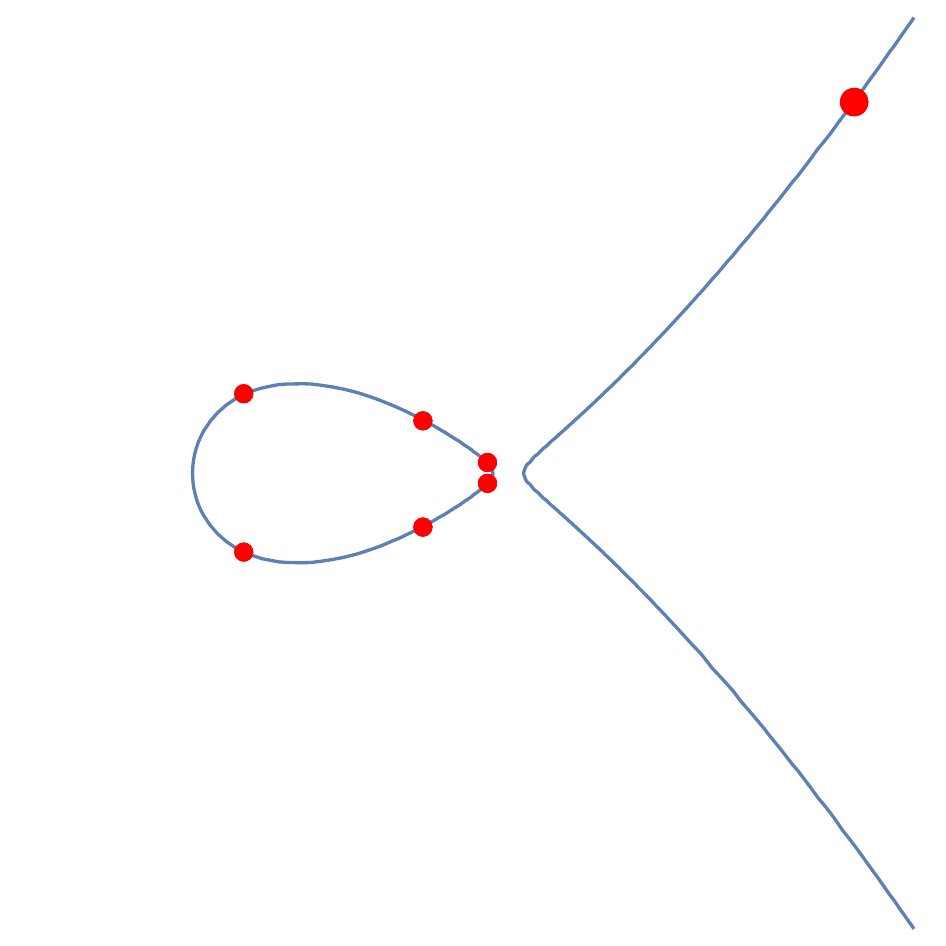}
            \put (45,40) {\tiny{1}}
            \put (52,53) {\tiny{3}}
            \put (45,57) {\tiny{4}}
            \put (52,45) {\tiny{2}}
            \put (25,60) {\tiny{5}}
            \put (25,35) {\tiny{6}}
            \put (90,80) {\tiny{$(a^{-2},a^{-2})$}}
        \end{overpic}
       
        \caption{The picture on the left illustrates the flow in case $(a^{-2},a^{-2})$ is a torsion point of order four. The picture on the right shows the flow in case that point has order six.}
        \label{fig:curveOrder6}
    \end{center}
\end{figure}
\subsection{Correlation kernel}
To explain the connection between the flow on the elliptic curve and the Wiener-Hopf factorization in Proposition \ref{prop:BD} we define functions $a,b,d: \mathcal E_- \to (0,\infty)$ by 
$$
  \begin{dcases}
    a(x,y)=\frac{a(a^2+1)(\alpha^2+1)}{2}\frac{ y-x}{1-a^2 x},\\
    b(x,y)=-\frac{1}{\alpha a x},\\
    d(x,y)=\frac{2a \alpha x(x-1/a^2)}{(a^2+1)(\alpha+1/\alpha)(y-x)}.
  \end{dcases}
$$
Since $x<0$ for $(x,y)\in \mathcal E_-$, these functions are well-defined with no poles and take strictly positive values.  Consider the maps
$$
    \mathcal P_-:(x,y)\mapsto
     b(x,y)
    \begin{pmatrix} 
         a(x,y)& 0 \\
         0 & 1
    \end{pmatrix} 
    \begin{pmatrix} 
        1& 1 \\
        \frac{a^2}{z} & 1
    \end{pmatrix}
    \begin{pmatrix} 
        1& 0 \\
        0 & a(x,y)
    \end{pmatrix},
$$
and 
$$
\mathcal P_+:(x,y)\mapsto
\begin{pmatrix} 
    1 & 0\\
    0&  {\frac{a^2}{\alpha^2}}d(x,y)
   \end{pmatrix} 
   \begin{pmatrix} 
    1& a^2 z \\
    1 & 1
   \end{pmatrix}
   \begin{pmatrix} 
    1& 0 \\
    0 & d(x,y)
   \end{pmatrix}.
$$
The first main result of this paper is  that the factorization \eqref{eq:WH_intro} is given by $$P_{k,\pm}(z)=\mathcal P_\pm(\sigma^k(x,y)).$$
We will discuss this claim at length in Section 4 in a slightly more general  setup and we refer to that section for more details.  The claim is then a special case of Theorem \ref{thm:equivalence}. Of important to us now is that it, together with Proposition \ref{prop:BD}, implies the following.

\begin{theorem}  \label{thm:main_result}
    The correlation kernel $K_N$ from Proposition \ref{prop:BD} can be written as 
\begin{multline} \label{eq:correlationkernelmain}
    \left[K_N((2m + \eps,2 x-j),(2m'+ \eps',2x'-j'))\right]_{j,j'=0}^1\\
    = -\frac{\mathbbm{1}_{2m'+\eps'<2m+\eps}}{2 \pi i} \int_{|z|=1} A_{e}(z)^{-\eps'} (P(z))^{m-m'}A_{o}(z)^{\eps} \frac{ z^{m-x-m'+x'}dz}{(z-a^2)^{m-m'} z}\\
    +\frac{1}{(2 \pi i)^2}
    \oint_{|w|=\rho_1} \oint_{|z|=\rho_2} A_{e}(w)^{-\eps'} P(w)^{N-m'}P_+(w)^{-1}P_-(z)^{-1} P(z)^{m}A_{o}(z)^{\eps} \\
    \times \frac{w^{x'+N-m'}(z-a^2)^{N-m}}{z^{x+N-m}(w-a^2)^{N-m'}}\frac{dz dw}{z(z-w)},
\end{multline}
where 
\begin{equation}\label{eq:pmin}
    P_-(z)  =\prod_{j=0}^{N-1}  b(\sigma^j(x,y))
     \begin{pmatrix} 
      a(\sigma^j(x,y))& 0 \\
      0 & 1
     \end{pmatrix} 
     \begin{pmatrix} 
      1& 1 \\
      \frac{a^2}{z} & 1
     \end{pmatrix}
     \begin{pmatrix} 
      1& 0 \\
      0 & a(\sigma^j(x,y))
     \end{pmatrix}
    \end{equation}
    and
    \begin{equation} \label{eq:pplus}
        P_+(z)=
       \prod_{j=0}^{N-1} 
        \begin{pmatrix} 
         1 & 0\\
         0&  \frac{a^2}{\alpha^2}d(\sigma^{N-1-j}(x,y))
        \end{pmatrix} 
        \begin{pmatrix} 
         1& a^2 z \\
         1 & 1
        \end{pmatrix}
        \begin{pmatrix} 
         1& 0 \\
         0 & d(\sigma^{N-1-j}(x,y))
        \end{pmatrix},
     \end{equation}
    and the contours of integration are counterclockwise oriented circles with radii $\rho_1$ and $\rho_2$ such that  $|a|^2<\rho_1<\rho_2<1/|a|^2$.  
\end{theorem}

A  proof of this theorem is given in Section \ref{sec:proofflow}.

If $(a^{-2},a^{-2})$ is a torsion point of order $d$, the flow $(x,y) \mapsto \sigma(x,y)$ is periodic, and the double integral formula can be rewritten in a useful way. 

\begin{corollary}
    Assume that $(a^{-2},a^{-2})$ is a torsion point of order $d$. Define 
    $$
     P_-^{(d)}(z)=P_{0,-}(z)\cdots P_{d-1,-}(z),
    $$
    and  
    $$
        P_+^{(d)}(z)=P_{d-1,+}(z)\cdots P_{0,+}(z).
    $$
Then we can rewrite \eqref{eq:correlationkernelmain} as
\begin{multline} \label{eq:correlationkernelmainperiodic}
    \left[K_{dN}((2m+\eps,2 x-j),(2m'+\eps',2x'-j'))\right]_{j,j'=0}^1\\=-\frac{\mathbbm{1}_{2m'+\eps'<2m+\eps}}{2 \pi i} \oint_{|z|=1} A_{e}(z)^{-\eps'} (P(z))^{m-m'}A_{o}(z)^{\eps} \frac{ z^{m-x-m'+x'}dz}{(z-a^2)^{m-m'} z}\\
    +\frac{1}{(2 \pi i)^2}
    \oint_{|w|=\rho_1} \oint_{|z|=\rho_2}  A_{e}(w)^{-\eps'} P(w)^{dN-m'}(P_+^{(d)}(w))^{-N}(P_-^{(d)}(z))^{-N} P(z)^{m}A_{o}(z)^{\eps}\\
    \times  \frac{w^{x'+dN-m'}(z-a^2)^{dN-m}}{z^{x+dN-m}(w-a^2)^{dN-m'}}\frac{dz dw}{z(z-w)},
\end{multline}
where $|a|^2<\rho_1<\rho_2<1/|a|^2$.
\end{corollary}
Note that in \eqref{eq:correlationkernelmain} we have replaced the size of the Aztec diamond $N$ by $dN$. This is not necessary and the upcoming analysis can  also be performed for the general case. Since the difference will only involve non-essential  cumbersome bookkeeping, we feel that working with $dN$ instead of $N$ makes for a cleaner presentation.
\subsection{Asymptotics}

The representation of the correlation kernel in \eqref{eq:correlationkernelmainperiodic} is a good starting point for an asymptotic study.  We will compute the microscopic process  in the limit $N\to \infty$ near the point 
\begin{equation}\label{eq:point_zoom_in}
    (2dT,2X)=(2 d \lfloor N\tau\rfloor, 2 \lfloor d N \xi\rfloor), \qquad 0<\tau<1, \ \  -\tfrac12 < \xi-\tau/2<0.
\end{equation}
  That is, we consider the limiting behavior of the correlation kernel 
\begin{equation} \label{eq:local_scaling_limit}
        \left[K_{ d N}\left(\left(2 d T +2m+ \eps, 2X+2x-j\right),\left( 2d T+2m'+\eps',2 X+2x'-j'\right)\right)\right]_{j,j'=0}^1\
\end{equation}
as $N\to \infty$, with $m,m'\in \mathbb Z$ fixed. Note that the first coordinate of the point \eqref{eq:point_zoom_in} is a multiple of $2d$ and the second coordinate is a multiple of $2$. This restriction is made for clarity purposes and is not necessary. Note also that any finite shift from \eqref{eq:point_zoom_in} can be absorbed into the variables $2m+\eps$, $2m'+\eps'$, $2x+j$ +$2x'+j'$ in \eqref{eq:local_scaling_limit}.

\subsubsection{The spectral curve}

To perform the asymptotic analysis it is convenient to diagonalize the matrices $P(w)$, $P(z)$, $P_+^{(d)}(w)$ and $P_-^{(d)}(z)$. 

The spectral curve  $\det (P(z)-\lambda)=0$ can be easily computed: 
\begin{equation} \label{eq:spectral_curve_introduction}
    \lambda^2- \left(\alpha+ \frac{1}{\alpha}\right)(1+a^2) \lambda+(1-a^2z)\left(1-\frac{a^2}{z}\right)=0.
\end{equation}
The curve has branch points at $z=0$, $z=\infty$, and at the zeros of  the discriminant:
\begin{equation} \label{eq:defR}
R(z):= \left(\alpha+ \frac{1}{\alpha}\right)^2(1+a^2)^2-4(1-a^2z)\left(1-\frac{a^2}{z}\right)=0.
\end{equation}
These zeros are negative and will be denoted by $x_1$ and $x_2$, ordered as $x_1<x_2<0$. With these points, we define a Riemann surface $\mathcal R$ consisting of two sheets $\mathcal R_j=\mathbb C \setminus \left((-\infty,x_1) \cup (x_2,0)\right)$, that we connect in the usual crosswise manner along the cuts $(-\infty,x_1)$ and $(x_2,0)$. The sheets have $0$ and $\infty$ as  common points.  See also  Figure \ref{fig:rieman}. We will write $z^{(j)}$ to indicate the point $z$ on the sheet $\mathcal R^{(j)}$. Then we define the square root $(R(z))^{1/2}$ on $\mathcal R$ such that $(R(z^{(1)}))^{1/2}>0$ for $z^{(1)}>0$.  The spectral curve \eqref{eq:spectral_curve_introduction} then defines a meromorphic function on $\mathcal R$ given by 
\begin{equation} \label{eq:lambda_curve}
\lambda(z)=\frac12 \left(\alpha+ \frac{1}{\alpha}\right)(1+a^2)+ \frac12 (R(z))^{1/2},
\end{equation}
with poles at $0$ and $\infty$, and zeros at $(a^{\pm 2})^{(2)}$.  The restrictions of $\lambda$ to $\mathcal R^{(j)}$ will be denoted by $\lambda_j$, i.e., $\lambda_j(z)=\lambda(z^{(j)})$.

Next, consider the spectral curves for $P_-^{(d)}$ and $P_+^{(d)}$,
\begin{align}
    \det (P_-^{(d)}(z)-\mu)=\mu^2&-\mu\Tr P_-^{(d)}(z) +\det P_-^{(d)}(z)=0,\label{eq:spectral_curve_mu_intro}\\
    \det (P_+^{(d)}(z)-\nu)= \nu^2&-\nu\Tr P_+^{(d)}(z) +\det P_+^{(d)}(z)=0, \label{eq:spectral_curve_nu_intro}
\end{align}
These spectral curves factorize \eqref{eq:spectral_curve_introduction} in the following way.

\begin{lemma} \label{lem:spectral_factorization_intro}
   The equations \eqref{eq:spectral_curve_mu_intro}, \eqref{eq:spectral_curve_nu_intro} for $\mu$ and $\nu$ define meromorphic functions on $\mathcal R$ such that 
   \begin{equation} \label{eq:spectralfactorization_intro}
       (\lambda (z))^d=\mu (z) \nu(z),
   \end{equation} 
   for $z\in \mathcal R$.  Then $\mu$ has a zero at $(a^{2})^{(2)}$ and a pole at $0$, both of the same order $d$, and  $\nu$ has  a zero at $(a^{-2})^{(2)}$ and a pole at $\infty$, both of the same order $d$. 

    With $E(z)$ defined by 
    \begin{equation} \label{eq:general_eigenvectors_intro}
     E(z)=
     \begin{pmatrix}
       a\alpha(1+z)& a\alpha(1+z)\\
       \lambda_1(z)- \alpha (a^2+1) & \lambda_ 2(z)- \alpha (a^2+1)
     \end{pmatrix},
   \end{equation}
    we have 
   \begin{equation} \label{eq:spectral_decomposition_Pz_intro}
P(z)=E(z)
\begin{pmatrix}
        \lambda_1(z) & 0 \\
        0 & \lambda_2(z)
\end{pmatrix}E(z)^{-1},
\end{equation}
\begin{equation}  \label{eq:spectral_decomposition_Pminus_intro}
    P^{(d)}_-(z)= E(z) 
    \begin{pmatrix}
        \mu_1(z) & 0 \\
     0 & \mu_2(z)
    \end{pmatrix} 
    E(z)^{-1}, 
\end{equation}
and 
\begin{equation} \label{eq:spectral_decomposition_Pplus_intro} 
    P^{(d)}_+(z)= E(z) 
        \begin{pmatrix}
            \nu_1(z) & 0 \\
            0 & \nu_2(z)
\end{pmatrix} E(z)^{-1}.
\end{equation}
Here $\mu_j(z)=\mu(z^{(j)})$ and $\nu_j(z)=\nu(z^{(j)})$ for $z \in \mathbb C \setminus \left((-\infty,x_1) \cup (x_2,0)\right)$.
\end{lemma}

The proof of this lemma will be given in Section \ref{sec:prooflemmasp}
\begin{figure}
    \begin{center}
    \begin{tikzpicture}[scale=.7]
        \draw (0,0)--(2,2)--(12,2)--(10,0)--(0,0);
        \draw (0,-3)--(2,-1)--(12,-1)--(10,-3)--(0,-3);
        \draw[very thick] (1,1)--(4,1);
        \draw[very thick] (6,1)--(7,1);
        \draw[very thick] (6,-2)--(7,-2);
        \draw[very thick] (1,-2)--(4,-2);
        \draw[help lines,dashed] (4,1)--(4,-2);
        \draw[help lines,dashed] (6,1)--(6,-2);
        \draw[help lines,dashed] (7,1)--(7,-2);
        \draw[help lines,dashed] (11,1)--(11,-2);

        \filldraw (8,1) circle(.05);
        \filldraw (10,-2) circle(.05);
        \filldraw (7,1) circle(.05);
        \filldraw (7,-2) circle(.05);
        \filldraw (11,1) circle(.05);
        \filldraw (11,-2) circle(.05);
        \draw (4,1.3) node {\small{$x_1$}};
        \draw (6,1.3) node {\small{$x_2$}};
        \draw (7,1.3) node {\small{$0$}};
        \draw (8,1.3) node {\small{$a^2$}};
        \draw (4,-2.5) node {\small{$x_1$}};
        \draw (6,-2.5) node {\small{$x_2$}};
        \draw (7,-2.5) node {\small{$0$}};
        \draw (10,-2.5) node {\small{$1/a^2$}};
        \draw (11,-2.5) node {\small{$\infty$}};
        \draw (11,1.3) node {\small{$\infty$}};

        \draw[help lines,orange,dashed] (4,1)--(6,1)--(6,-2)--(4,-2)--(4,1);
        \draw[help lines,orange,dashed] (7,1)--(11,1)--(11,-2)--(7,-2)--(7,1);
      
    \end{tikzpicture}

    \caption{The two sheeted Riemann surface $\mathcal R$. The  dashed lines represent the cycles $\mathcal C_1$ and $\mathcal C_2$.}
    \label{fig:rieman}
\end{center}
\end{figure}
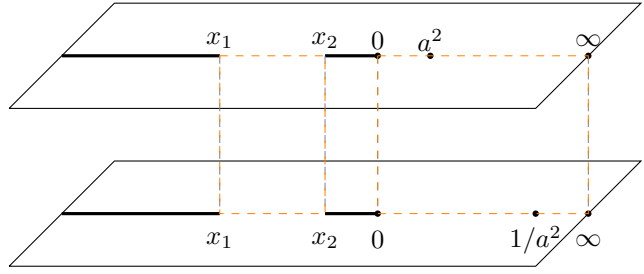

One particular consequence of this lemma is that we can simultaneously diagonalize $P(z)$ and $P_\pm^{(d)}(z)$. In the following theorem we use this to rewrite the correlation kernel in ~\eqref{eq:correlationkernelmainperiodic}.

\begin{theorem}\label{thm:correlationkernel_spectralcurve}
    Assume $(a^{-2},a^{-2})$ is a torsion point of order $d$. Set, with $E(z)$ as in \eqref{eq:spectral_curve_introduction},  
    \begin{equation} \label{eq:defF}
        F(z)= \begin{dcases}  E(z)
            \begin{pmatrix}
                1 & 0 \\
                0 & 0 
            \end{pmatrix}
            E(z)^{-1},& z \in \mathcal R_1,\\
             E(z)
                \begin{pmatrix}
                    0 & 0 \\
                    0 & 1
                \end{pmatrix}
                E(z)^{-1},& z \in \mathcal R_2.
        \end{dcases}
    \end{equation}
    Then,
\begin{multline} \label{eq:correlationkernelmainperiodiceigenvalues}
    \left[K_{   d N}((2dT+2m+\eps, 2X +2x-j),(2dT +2m'+ \eps',2X +2x'-j'))\right]_{j,j'=0}^1\\
    = -\frac{\mathbbm{1}_{2m'+\eps'<2m+\eps}}{2 \pi i} 
    \int_{\gamma_2^{(1)}\cup \gamma_2^{(2)}} A_e(z)^{-\eps'} F(z)A_o(z)^{\eps} \lambda(z) ^{m-m'}\frac{z^{m-x-m'+x'}}{(z-a^2)^{m-m'}} \frac{dz}{z}\\
    +\frac{1}{(2 \pi i)^2}
    \oint_{\gamma_1^{(1)}\cup \gamma_1^{(2)}} \oint_{\gamma_2^{(1)} \cup \gamma_2^{(2)}} A_e(w)^{-\eps'}  F(w) F(z) A_{o}(z)^{\eps} \frac{\lambda(z)^{m} }{\lambda(w)^{m'}}\frac{w^{x'-m'} }{z^{x-m}}\frac{(w-a^2)^{m'}}{(z-a^2)^{m}} \\
    \times 
    \frac{\mu(w)^{N-T}}{ \mu(z)^{N-T} }\frac{\nu(z)^{T}}{\nu(w)^{T}}
     \frac{w^{d(N-T)+X}}{z^{d(N-T)+X}}\frac{(z-a^2)^{d(N-T)}}{(w-a^2)^{d(N-T)}}\frac{dw dz}{z(z-w)},
\end{multline}
where $\gamma_2^{(1,2)}$ are the  unit circles with counterclockwise orientation on the sheets $\mathcal R_{1,2}$,  $\gamma_1^{(1)}$ is a  counterclockwise oriented contour inside the contour $\gamma_2^{(1)}$ on the sheet $\mathcal R_1$ that goes around $(a^2)^{(1)}$ and the cut $[x_2,0]$,  and $\gamma_1^{(2)}$ is a  counterclockwise oriented contour on the sheet $\mathcal R_2$ inside the contour $\gamma_2^{(2)}$  that goes around  the cut $[x_2,0]$.  See also Figure \ref{fig:contours_1}.
\end{theorem}

The proof of this Theorem will be given in 
\ref{sec:prooftheoremp}.

Note that $A_e(w)^{-1}$ is analytic at $w=a^2$ (even though $A_e(w)$ is not). Moreover,  $\lambda(w)^{-m'}\mu(w)^{N-T}$  has a zero at $w=(a^2)^{(2)}$ of order $d(N-T)-m'$,  and this zero cancels the pole at $w=(a^2)^{(2)}$ in the double integral in \eqref{eq:correlationkernelmainperiodiceigenvalues}. The contour $\gamma_1^{(2)}$ therefore does not have to go around $(a^2)^{(2)}$.

By passing to the eigenvalues and spectral curves we in fact are essentially looking at a scalar problem, instead of a matrix-valued one. 

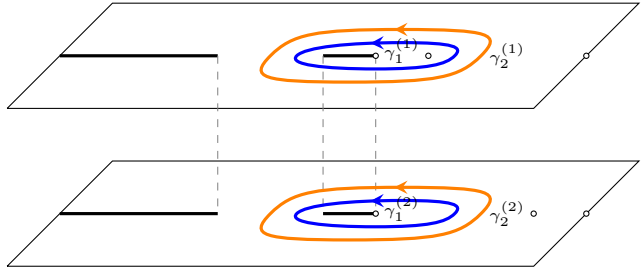
\begin{figure}[t]
    \begin{center}
    \begin{tikzpicture}[scale=.7,path/.append style={
        decoration={
            markings,
            mark=at position 0.5 with {\arrow[xshift=2.5\pgflinewidth,>=stealth]{>}}
        },
        postaction=decorate,
        thick,
      }]
        \draw (0,0)--(2,2)--(12,2)--(10,0)--(0,0);
        \draw (0,-3)--(2,-1)--(12,-1)--(10,-3)--(0,-3);
        \draw[very thick] (1,1)--(4,1);
        \draw[very thick] (6,1)--(7,1);
        \draw[very thick] (6,-2)--(7,-2);
        \draw[very thick] (1,-2)--(4,-2);
        \draw[help lines,dashed] (4,1)--(4,-2);
        \draw[help lines,dashed] (6,1)--(6,-2);
        \draw[help lines,dashed] (7,1)--(7,-2);

        \draw[path,orange, very thick] plot [smooth cycle, tension=1.5] coordinates {(5,1)  (6.5,.5)(9,1) (7.5,1.5) };
        \draw[path,very thick, blue] plot [smooth cycle, tension=1.2] coordinates {(5.5,1) (7,0.75)  (8.5,1) (7.5,1.25)};
        \draw[path,very thick, blue] plot [smooth cycle, tension=1.2] coordinates {(5.5,-2) (7,-2.25)  (8.5,-2) (7.5,-1.75)};
    
        \draw[path,orange, very thick] plot [smooth cycle, tension=1.5] coordinates {(5,-2) (6.5,-2.5)  (9,-2) (7.5,-1.5)};
        \draw (9.5,1) node {\tiny{$\gamma_2^{(1)}$}};
        \draw (9.5,-2) node {\tiny{$\gamma_2^{(2)}$}};
        \draw (7.5,1.1) node {\tiny{$\gamma_1^{(1)}$}};
        \draw (7.5,-1.9) node {\tiny{$\gamma_1^{(2)}$}};
        \filldraw[fill=white] (8,1) circle(.05);
        \filldraw[fill=white](10,-2) circle(.05);
        \filldraw[fill=white](7,1) circle(.05);
        \filldraw[fill=white](7,-2) circle(.05);
        \filldraw[fill=white](11,1) circle(.05);
        \filldraw[fill=white] (11,-2) circle(.05);
    \end{tikzpicture}
    \caption{The contours of integration in \eqref{eq:correlationkernelmainperiodiceigenvalues}. The blue contour represents $\gamma_1$ and the orange  contours are the unit circles on the two different sheets.}
    \label{fig:contours_1}
\end{center}
\end{figure}

\begin{remark}
    We note that the spectral curve $\det \left(P(z)-\lambda\right)=0$ and the elliptic curve  $\mathcal E$ in \eqref{eq:elliptic_aztec} are related. Indeed, \eqref{eq:elliptic_aztec} can be written as 
    $$
    \det \left(P(x)-\tfrac12 (a^2+1)(\alpha+1/\alpha)\left(1+y/x\right)\right)=0.
    $$ 
    In other words, the elliptic curve  $\mathcal E$ equals the spectral curve after changing the spectral variable. 
\end{remark}
\subsubsection{Saddle point equation and classification of different regions} \label{sec:char}

The representation \eqref{eq:correlationkernelmainperiodiceigenvalues} is a very good starting point for asymptotic analysis. To illustrate this we will perform a partial asymptotic study, based on a saddle point analysis.  We note that a similar analysis has been given in \cite{B,DK}.  An interesting feature is that our analysis will depend on the torsion $d$, but in such a way that we can treat all values of $d$ simultaneously. 

To perform a saddle point analysis of \eqref{eq:correlationkernelmainperiodiceigenvalues} we need to find the saddle points and the contours of steepest descent/ascent for the action defined by 
\begin{equation} \label{eq:defPhi}
 \Phi(z;\tau,\xi)=(1-\tau)\log \mu(z) -\tau \log \nu(z) +d(1-\tau+\xi) \log z -d(1-\tau)\log(z-a^2).
\end{equation}
This is a multi-valued function, but the differential 
$$
\Phi'(z) dz
$$
is single valued on $\mathcal R$. Its zeros are the saddle points for $\Re \Phi$, and we will be especially interested in them. Let $\mathcal C_1$ be the cycle on $\mathcal R$ defined by connecting the segments $(x_1,x_2)$ on $\mathcal R_1$ and $\mathcal R_2$ at the end points $x_1$ and $x_2$. Similarly, let $\mathcal C_2$ be the cycle  that combines the copies of $(0,\infty)$ on both sheets.

\begin{proposition}\label{prop:four_saddles}
       The differential $\Phi'(z)dz$ has simple poles at $0$, $(a^2)^{(1)}$, $(1/a^2)^{(2)}$ and $\infty$. There  are four saddle points (i.e., the critical points where $\Phi'(z)dz=0$) counted according to multiplicity. There are at least two distinct saddle points on the cycle $\mathcal C_1$.
\end{proposition}

\begin{figure}[t]
    \begin{center}
    \begin{overpic}[scale=.6]{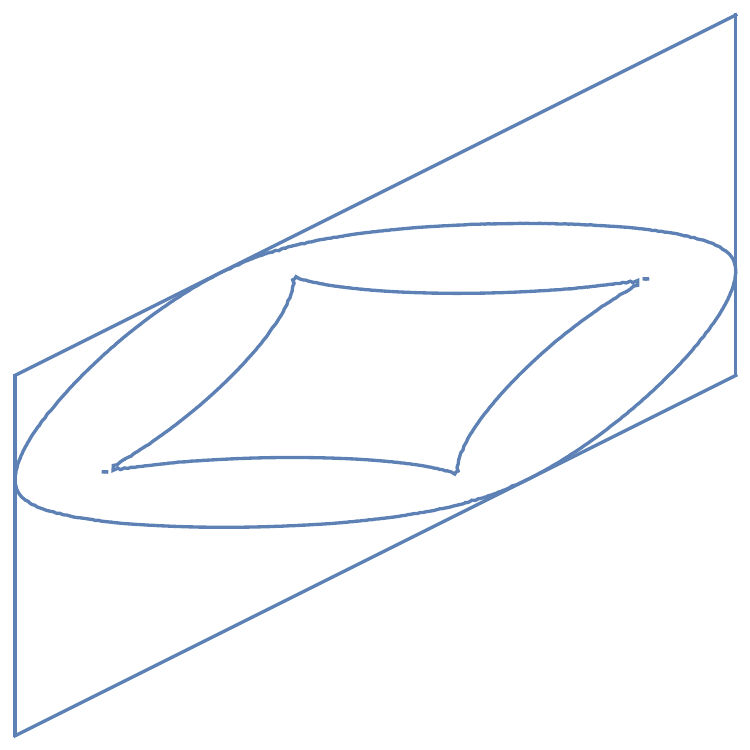}
        \put (45,50) {\textrm{Smooth}}
        \put (60,64) {\textrm{Rough}}
        \put (75,75) {\textrm{Frozen}}
    \end{overpic}
    \begin{overpic}[scale=.4,angle=45]{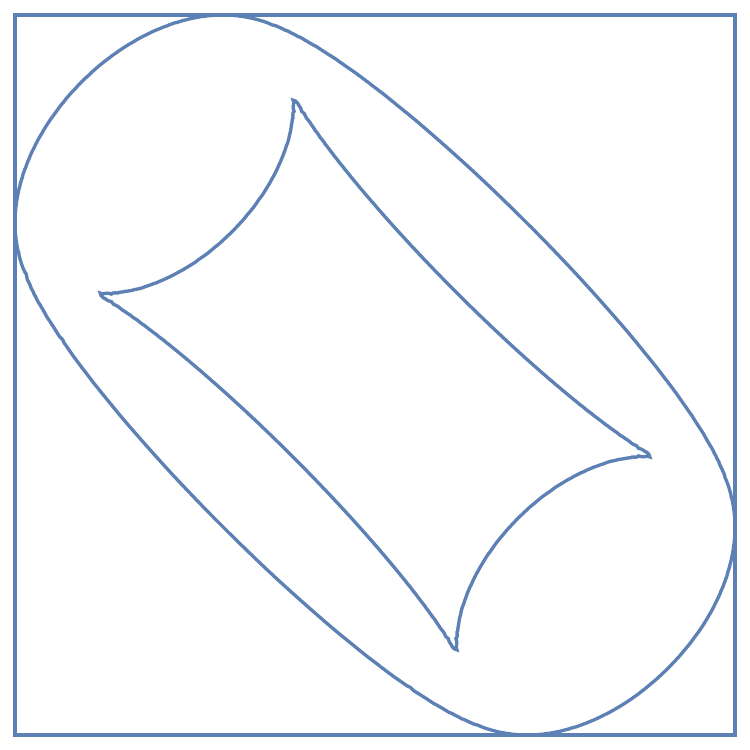}
    \end{overpic}
    \caption{Both pictures represent a partitioning of the region into the frozen region, the rough disordered region and the smooth disordered region. The picture on the left uses the natural coordinates $(\tau, \xi)$ corresponding to the point process associated  with the non-intersecting paths. The picture on the right corresponds to the coordinates for the original dimer model. In both pictures we have $a^2=\alpha/(1+\alpha+\alpha^2)$ and $\alpha=\frac12$.}
    \label{fig:regions}
\end{center}
\end{figure}
There are always two saddle points on the cycle $\mathcal C_1$, but it is the location of the two other saddle points that determines the phase at the point $(\tau,\xi)$.  We say that  $(\tau, \xi)$ is 
\begin{itemize}
    \item  in the \textbf{frozen} region, if we have two distinct saddle points on the cycle $\mathcal C_2$;
    \item in the \textbf{smooth disordered} region, if we have four distinct saddle points on the cycle $\mathcal C_1$;
    \item in the \textbf{rough disordered} region,   if there is a saddle point in the upper half plane of $\mathcal R_1$ or $\mathcal R_2$;
    \item on the boundary between the \textbf{rough and smooth disorderd} regions, when this saddle point from the upper half plane coalesces with its complex conjugate on the cycle $\mathcal C_1$;
    \item on the boundary between the \textbf{rough and frozen} regions, when the saddle point from the upper half plane coalesces with its complex conjugate on the cycle $\mathcal C_2$.
\end{itemize}
We note that the terminology rough, smooth and frozen goes  back to at least \cite{KOS}. In that work  also the alternatives \emph{gaseous} for smooth disordered and \emph{liquid} for rough disordered were mentioned. In the subsequent literature both these terms have been used. We chose to use terminology frozen, rough and smooth disordered. The difference between these regions is in the decay of the local correlations for the local Gibbs measure. In the frozen region, the randomness disappears. In the rough  disordered region,   the correlations between two points decay polynomially in their distance, whereas in the smooth disorder regions these correlations decay exponentially. Our list above suggests that these different behavior can be characterized in terms of the location of the two remaining saddle points. The following theorem justifies this characterization for the smooth disordered (or gaseous) region.

\begin{theorem} \label{thm:gas}
    Let $(\tau,\xi)$ be in the smooth disordered region. Then 
    \begin{multline} \label{eq:gaskernel}
       \lim_{N \to \infty} \left[K_{ d N}((2dT+2m+\eps, 2X+2x-j),(2d T+2m'+\eps',2X+2x'-j')\right]_{j,j'=0}^1\\
    = -\frac{\mathbbm{1}_{2m'+\eps'<2m+\eps}}{2 \pi i} 
    \int_{ \gamma_2^{(2)}} A_e(z)^{-\eps'} F(z) A_o(z)^{\eps} \lambda(z) ^{d(m-m')}\frac{z^{m-x-m'+x'}}{(z-a^2)^{m-m'}} \frac{dz}{z}\\
    +\frac{\mathbbm{1}_{2m'+ \eps'\geq 2m+ \eps}}{2 \pi i} 
    \int_{ \gamma_2^{(1)}} A_e(z)^{-\eps'} F(z) A_o(z)^{\eps} \lambda(z) ^{d(m-m')}\frac{z^{m-x-m'+x'}}{(z-a^2)^{m-m'}} \frac{dz}{z}.
    \end{multline}
\end{theorem}

Note that from \eqref{eq:gaskernel} we see that the limiting mean density in the smooth disordered region is given by
\begin{multline*}
    \lim_{N \to \infty} \left[K_{ d N}((2dT+2m+\eps, 2X+2x-j),(2d T+2m+\eps,2X+2x-j')\right]_{j,j'=0}^1\\
    =    \frac{1}{2 \pi i}  \int_{ \gamma_2^{(1)}}  A_e(z)^{-\eps} F(z) A_o(z)^{\eps}  \frac{dz}{z},
\end{multline*}
and the right-hand side is independent of $(T,X)$ (as long as it is in the smooth disordered region). 

It is also not difficult to see that the right-hand side of  \eqref{eq:gaskernel} decays exponentially with the distance between $(m,x)$ and $(m',x')$. Indeed, for $m$ and $m'$ fixed, the right-hand side is the $(x-x')$-th Fourier coefficient of a function that is analytic in an annulus. Such coefficients decay exponentially with a rate that is determined by the width of the annulus. More generally, the exponential decay follows from a steepest descent analysis for the right-hand side of \eqref{eq:gaskernel}.

The proof of Theorem \ref{thm:gas} will be given in Section \ref{sec:proofsaddle} and it is based on a saddle point analysis of the integral representation \eqref{eq:correlationkernelmainperiodiceigenvalues}. We are confident that such a saddle point analysis can be carried out similarly for the rough disordered and frozen regions. Since it  requires  non-trivial  effort and since a full asymptotic study is not the main focus of this paper, we do  not perform such an analysis here.

\subsection{The boundary of the rough disordered region} \label{sec:boundary}
We will now show that the boundary of the rough disordered region is an algebraic curve and discuss how this curve can be found explicitly in particular cases.

We start with the following proposition. 
\begin{proposition} \label{prop:phiprime}
    With $\Phi$ as in \eqref{eq:defPhi} and $R(z)=a^2(z-x_1)(z-x_2)/z$  as in \eqref{eq:defR} we have 
    \begin{equation}\label{eq:phiingam}
        \Phi'(z)=d(1-\tau )a^2\frac{z\gamma_1+\gamma_2+  \gamma_3 R(z)^{1/2}}{(z-a^2)z R(z)^{1/2}} -d\tau \frac{\gamma_1+\gamma_2z+  \gamma_3 zR(z)^{1/2}}{(z-a^{-2})z R(z)^{1/2}}+\frac{d(1-\tau+\xi)}{z} -\frac{d(1-\tau)}{z-a^2},
    \end{equation}  
    where $\gamma_1,\gamma_2$ and $\gamma_3$ are real constants determined by  
       \begin{equation} \label{eq:constants_gamma}
        \begin{dcases}
            \gamma_1= -\frac{1}{2}\left( 
                \frac{1}{d}\sum_{j=0}^{d-1} a(\sigma^j(x,y)) \right)^{1/2} \left(\frac{1}{d}\sum_{k=0}^{d-1} \frac{1}{a(\sigma^k(x,y))}
            \right)^{1/2},\\
            \gamma_2+a^2 \gamma_1=-\frac12 (a^2+1)(\alpha+ 1/\alpha),\\
            \gamma_3=\frac{1}{2},
        \end{dcases}
       \end{equation}
       and the square root is taken such that $R(z)^{1/2}$ is meromorphic on $\mathcal R$ and $R(z^{(1)})^{1/2}>0$ for $z>0$. 
    \end{proposition}

    The proof of this proposition will be given in Section \ref{sec:proofphi}.

By inserting the constants \eqref{eq:phiingam} into $\Phi'(z)$, multiplying by $(z-a^2)(z-a^{-2})R(z)^{1/2}$, and re-organizing the equation so that all terms with $R(z)^{1/2}$ are on the right, we see that $\Phi'(z)=0$ can be written as
\begin{multline}  \label{eq:saddlepointequations2}
  (1-\tau)a^2(\gamma_1 z+ \gamma_2)(z-a^{-2})-\tau  (\gamma_1+\gamma_2 z)(z-a^2)\\
  =-R(z)^{1/2}\left((1-\tau) a^2\gamma_3 (z-a^{-2}) -\tau \gamma_3(z-a^2) z\right. \\ \left. +(1-\tau+\xi)(z-a^2)(z-a^{-2})- (1-\tau)z(z-a^{-2})\right).
\end{multline}
Before we proceed,  note  that $z=a^{- 2}$ and  $z=a^{2}$  are two solutions that we just introduced by multiplying by $(z-a^{2})(z-a^{-2})$ and are not saddle points. 

By squaring both sides of \eqref{eq:saddlepointequations2} and multiplying by $z$ we  find a polynomial equation of degree $6$ in $z$ with coefficients that are quadratic functions of $\tau$ and $\xi$. Since $z=a^{\pm 2}$ are solutions that we are not interested in, we are left with an equation of  degree four. There are four solutions to this equation, and each of them corresponds to exactly one point on the surface. This confirms that we indeed have four saddle points, which was part of the statement in Proposition \ref{prop:four_saddles}.

This also allows to write an equation for the rough disordered boundary. Indeed, the coefficients of this fourth degree equation will be quadratic expressions in $\tau$ and $\xi$. We have a third order saddle in case the discriminant vanishes. The discriminant of a polynomial of degree four is a  polynomial in its coefficients of degree six. Thus, the discriminant is a polynomial in $\tau $ and $\xi$ of degree twelve. In the explicit cases that we tried, we found,  with the help of computer software, that this degree twelve curve can be factorized into a curve of degree eight and remaining factors that are not relevant. This also matches with the findings of \cite{CJ} and \cite{BD} for the special case $a=1$. We have, however, only been able to verify that this holds numerically in special cases  (one of them we will discuss in Appendix \ref{sec:order_six}) and do not have a proof that it holds generally. We leave this as an interesting open problem and post the following conjecture:
\begin{conjecture}
    The boundary of the rough disordered region is an algebraic curve in $\tau$ and $\xi$ of degree eight.
\end{conjecture}

\begin{remark}
    There is another  way of parametrizing the boundary. Indeed, on the two  components of the boundary of the rough region we have  a coalescence of saddle points on the cycles $\mathcal C_1$  or $\mathcal C_2$. This means that we have a double zero of the differential $\Phi'(z)dz$.  This gives a way of parametrizing these curves. Indeed, $\Phi'(z)=\Phi''(z)=0$ for $z \in \mathcal C_1$ or $\mathcal C_2$ gives a linear system of equations for $\mu$ and $\xi$ that can be easily solved.  
\end{remark}

Another interesting consequence of \eqref{eq:phiingam} is that the saddle point equation $\Phi'(z)=0$ only depends  on the order $d$ of the torsion  via the constant $\gamma_1$ in \eqref{eq:constants_gamma}. However, it is even possible to replace this with another expression that does not involve $d$: 

\begin{lemma}\label{lem:relationsgamma}
    The constants $\gamma_1$ and $\gamma_2$ from \eqref{eq:constants_gamma} are related via
    \begin{equation}\label{eq:g1g2}
     \gamma_1  \int_{x_1}^{x_2} \frac{x dx}{(x-a^{-2} )\sqrt{R(x)}}=-\gamma_2 \ {\int_{x_1}^{x_2} \frac{dx}{(x-a^{-2}) \sqrt{R(x)}}},
    \end{equation}
    where $\sqrt{R(x)}>0$ for $x\in (x_1,x_2)$.
\end{lemma}
The proof of this lemma will ve given in Section \ref{sec:proofgamma}.

By replacing the equation for $\gamma_1$ in \eqref{eq:constants_gamma} by \eqref{eq:g1g2} we see that we have eliminated the dependence on $d$ from the saddle point equation, and the saddle point equation makes sense for general parameters $a$ and $\alpha$. Although the arguments that we provide in this paper use the torsion at several places, it is natural to conjecture that the saddle point analysis and its consequences  can be extended in this way. In particular, we conjecture the characterization of the different phases in Section \ref{sec:char} and Theorem \ref{thm:gas}  to hold under this extension.  We leave this as an open problem. 

\subsection{Overview of the rest of the paper and the proofs}

In the remaining part of this paper we will prove the main results. In Section \ref{sec:flow} we will show that the linear flow on the elliptic curve can be used to find a Wiener-Hopf factorization in Proposition \ref{prop:BD}. We will do this in a more general setup than only for the biased Aztec diamond. In Section \ref{sec:proofs} we will return to the biased Aztec diamond and prove Theorem \ref{thm:main_result} in Section \ref{sec:proofflow}, which is by then just an identification of the parameters in the discussion of Section \ref{sec:flow}. Then Lemma \ref{lem:spectral_factorization_intro} and Theorem \ref{thm:correlationkernel_spectralcurve} are proved in Sections \ref{sec:prooflemmasp} and \ref{sec:prooftheoremp}, respectively. The saddle point analysis starts with proving Lemma \ref{prop:four_saddles} in Section \ref{sec:prooffoursaddle}. After that, we perform a saddle point analysis in  Section \ref{sec:proofsaddle} and prove Theorem \ref{thm:gas}. Proposition \ref{prop:phiprime} is proved in Section \ref{sec:proofphi} and Lemma \ref{lem:relationsgamma} in Section \ref{sec:proofgamma}. In Appendix \ref{sec:order_six}  we work out the example where $(a^{-2},a^{-2})$ is a torsion point of order six. We compute the boundary of the rough disordered region, and we provide numerical results supporting the saddle point analysis of Section \ref{sec:proofsaddle}. Finally, in Appendix \ref{sec:divisionpolynomials} we will show how the notion of division polynomials can be used to find algebraic relations between $a$ and $\alpha$ so that $(a^{-2},a^{-2})$ is a torsion point of order $d$.

\section{The flow} \label{sec:flow}
In this section we introduce a flow on a space of matrices that will give a Wiener-Hopf factorization in the correlation kernel. We prove that this flow is equivalent to a linear flow on an elliptic curve using  translations by a fixed point on that curve. 
\subsection{The space}

First we have to define the space of matrices that we work on. To this end, we first introduce
\begin{equation} \label{eq:defCalS}
		\mathcal S=\left \{
				\begin{pmatrix}
					a_{11} & a_{12} + b_{12} z\\
					a_{21} +b_{21}/z  & a_{22 }
				\end{pmatrix}
				\mid  a_{11},a_{22},a_{12},a_{21}, b_{12}, b_{21} >0
			\right	\}.
\end{equation}
Clearly, the determinant $\det P(z)$ of any $P \in \mathcal S$ is a rational function in $z$ with poles at $z=0$ and $z= \infty$ and no other. Also, $\det P(z)$ will have two zeros $z_1$ and $z_2$, and we will assume that $$0<z_1<1<z_2.$$ 
Then the winding number of $\det P(z)$ with respect to the unit circle equals zero. 

The flow that we will define on $\mathcal S$ will be such that $\det P(z)$ and $ \Tr P(z)$ will be invariant under it. We therefore introduce the sets 
$$
\mathcal S(z_1,z_2,c_1,c_2)
= 
\left \{ P(z) \in \mathcal S 
\mid  \  \Tr P(z)=2 c_1, \quad \det P(z)= -c_2(z-z_1)(z-z_2)/z 
\right\}
$$
for $c_1,c_2>0$ and $0< z_1<1<z_2.$ 

Naturally, $c_1,c_2$ and $z_1,z_2$ be expressed in terms of $a_{ij}$ and $b_{ij}$. Indeed, 
\begin{equation} \label{eq:equationsforc1c2}
    \begin{dcases}
    c_1= \frac{a_{11}+a_{22}}{2},\\
    c_2= a_{21} b_{12},
    \end{dcases}
\end{equation}
and $z_1,z_2$ are the solutions to $\det P(z)=0$. Equivalently, $z_1$ and $z_2$ can be obtained from the following equations:
\begin{equation} \label{eq:equationsforz1z2}
    \begin{dcases}
        z_1z_2=\frac{a_{12}b_{21}}{a_{21}b_{12}},\\
        c_2(z_1+z_2)=a_{11}a_{22}-(a_{21}a_{12}+b_{12}b_{21}),
    \end{dcases}
\end{equation} 
which, combined with the condition $0<z_1<1<z_2$, determine $z_1$ and $z_2$ uniquely. We also note that the condition $0<z_1<1<z_2$ is equivalent to requiring $\det P(1)>0$, because $\det P(z)\to -\infty$ for $z\downarrow 0$ and $z  \to + \infty$. In terms of $a_{ij}$ and $b_{ij}$ this means that the condition is equivalent to 
$$a_{11}a_{22}> (a_{12}+b_{12})(a_{21}+b_{21}).$$
Note that this also shows that right-hand side of the second equation in \eqref{eq:equationsforz1z2} is positive, as it should be. 

It should also be noted that $c_1,c_2,z_1$ and $z_2$ cannot take arbitrary values. For instance, we have the following result. 
\begin{lemma} \label{lem:relationc1c2z1z2}
    We have 
        \begin{equation} \label{eq:relationc1c2z1z2}
            c_1^2\geq c_2 (\sqrt {z_1}+\sqrt{z_2})^2.
        \end{equation}
\end{lemma}
\begin{proof}
    The proof follows after inserting \eqref{eq:equationsforc1c2} and \eqref{eq:equationsforz1z2} into 
    \begin{align*}
        c_1^2-c_2 (\sqrt {z_1}+\sqrt{z_2})^2&=c_1^2-c_2 (z_1+z_2+2 \sqrt{z_1 z_2})\\
        &=\frac{(a_{11}+a_{22})^2}{4}-a_{11}a_{22}+a_{12}a_{21}+b_{12}b_{21}-2\sqrt{a_{12}a_{21}b_{12}b_{21}}\\
        &=\frac{(a_{11}-a_{22})^2}{4}+ (\sqrt{a_{12}a_{21}}-\sqrt{b_{12}b_{21}})^2\geq 0,
    \end{align*}
   giving  the result.
\end{proof}
As we will see later, the inequality \eqref{eq:relationc1c2z1z2} is sufficient to ensure that $\mathcal S(z_1,z_2,c_1,c_2) \neq \emptyset$. We will give an explicit parametrization of $\mathcal S(z_1,z_2,c_1,c_2)$ in terms of part of an elliptic curve. But first, let us  define a flow on $\mathcal S(z_1,z_2,c_1,c_2)$.

\subsection{Definition of the flow}

We will be interested in factorization of the matrices in $\mathcal S$ of a particular form. Start by introducing the sets 
	$$
		\mathcal S_- =\left  \{ 
			\begin{pmatrix}
				a & 0 \\
				0& 1
			\end{pmatrix}
			\begin{pmatrix}
			1 & 1 \\
			\frac{z_1}{z} & 1 
			\end{pmatrix}
				\begin{pmatrix}
			b & 0 \\
			0& 1 
			\end{pmatrix}
			\mid a>0, b>0, 0<z_1<1 \
			\right\},
	$$
and 
	$$
		\mathcal S_+ = \left\{ 
			\begin{pmatrix}
				1 & 0 \\
				0& c 
			\end{pmatrix}
			\begin{pmatrix}
				1 & \frac{z}{z_2} \\
				1 & 1 
			\end{pmatrix}
			\begin{pmatrix}
				1 & 0 \\
				0& d 
			\end{pmatrix}
			\mid c> 0, d>0, z_2 >1
		 \right \}.
	$$
It is straightforward to verify that if $Q_+ \in \mathcal S_+$ and $Q_-\in \mathcal S_-$ then $Q_+ Q_- \in \mathcal S$ and $Q_-Q_+\in \mathcal S$.
\begin{proposition} \label{prop:uniquefact}
		Let $P \in \mathcal S(z_1,z_2,c_1,c_2)$. Then there exist unique $Q_\pm\in \mathcal S_{\pm } $ such that 
			$
				P=Q_-Q_+.
			$ 
\end{proposition}
\begin{proof}
Note that 
	$$
    \begin{pmatrix}
        a & 0 \\
        0& 1 
    \end{pmatrix}
    \begin{pmatrix}
        1 & 1 \\
        \frac{z_1}{z} & 1 
    \end{pmatrix}
    \begin{pmatrix}
        b & 0 \\
        0& 1 
    \end{pmatrix}
	\begin{pmatrix}
		1 & 0 \\
		0& c 
	\end{pmatrix}
	\begin{pmatrix}
		1 & \frac{z}{z_2} \\
		1 & 1 
	\end{pmatrix}
	\begin{pmatrix}
		1 & 0 \\
		0& d 
	\end{pmatrix}	
	=
	\begin{pmatrix}
		ab+ac& acd+ \frac{abdz}{z_2}\\
		c+\frac{bz_1}{z} & cd+\frac{bd z_1}{z_2}
	\end{pmatrix}.
	$$
	To find  $Q_\pm$ we have to solve 
	$$
    \begin{pmatrix}
		ab+ac& acd+ \frac{abdz}{z_2}\\
		c+\frac{bz_1}{z} & cd+\frac{bd z_1}{z_2}
	\end{pmatrix}=
	\begin{pmatrix}
		a_{11} & a_{12} + b_{12} z\\
		a_{21} +b_{21}/z  & a_{22 }
	\end{pmatrix}.
	$$
	By comparing the coefficients on both sides  we obtain six equations for  the six unkowns $a,b,c,d,z_1$ and $z_2$. The parameters $z_1,z_2$ can be found from the condition $\det P(z_1)= \det P(z_2)=0$. Then  finding the remaining equation gives
    \begin{equation} \label{eq:abcd_intermsof_aijbij}
        \begin{dcases} 
            a=\frac{a_{11}z_1}{a_{21}z_1+b_{21}},\\
            b= \frac{b_{21}}{z_1},\\
            c= a_{21},\\
            d= \frac{a_{12}(a_{21}z_1+b_{21})}{a_{11}a_{21}z_1}.
        \end{dcases}
     \end{equation}
  	 This determines the factorization $P=Q_-Q_+$ uniquely.
\end{proof}
Because of the special structure of $\mathcal S_\pm$ we have uniqueness of the factorization. However, for our purposes we need an additional degree of freedom by adding a diagonal factor.  Indeed, if $P=Q_-Q_+$ then 
$P_-=Q_- D$ and $P_+= D^{-1} Q_+$  
for any diagonal matrix $D$ also provides a factorization of $P$ such that $P_+P_- = D^{-1} Q_+Q_- D \in \mathcal S(z_1,z_2,c_1,c_2)$. We will use this additional degree of freedom by requiring that 
$$
    P_+ P_-=P_-P_+ + \mathcal O(1), \qquad z \to \infty.
$$
In other words, we require that the leading term in the asymptotic behavior fo $P_-P_+$ and $P_+P_-$ match. In order to achieve this, we define
\begin{equation}\label{eq:definitionD}
   D=
   \begin{pmatrix}
    1&0\\
    0& ab
    \end{pmatrix},
\end{equation}
where $a,b$ are the parameters in $Q_-$.

\begin{definition} \label{def:flow}
    Define the map $s: \mathcal S(z_1,z_2,c_1,c_2)\to \mathcal S(z_1,z_2,c_1,c_2)$ as follows: for $P\in \mathcal S(z_1,z_2,c_1,c_2)$ let $P=Q_-Q_+$ be the unique factorization from Proposition \ref{prop:uniquefact} and take $P_+=D^{-1} Q_+ $ and $P_-=  Q_-D$  where $D$ is defined by \eqref{eq:definitionD}. Then set $s(P)=P_+P_-$. 
\end{definition}
    The flow on $\mathcal S(z_1,z_2,c_1,c_2)$ that we wish to study is then defined by iterating the map $s$, i.e., the flow is defined by the recurrence
   $$
       \begin{dcases}
           P_{k+1}=s(P_k), & k\geq 0,\\
           P_0=P \in \mathcal S(z_1,z_2,c_1,c_2).
       \end{dcases}
   $$
It turns out it is rather complicated  to keep track of this dynamics, and our  goal is to describe this dynamics in a way that it  is easier to grasp. 

\subsection{Translations on an elliptic curve}

Consider the elliptic curve $\mathcal E$ over $\mathbb R$ defined by (with $c_1,c_2>0$ and $0<z_1<1<z_2$ as before)
$$\mathcal E= 
		\left\{(x,y)  \in \mathbb R^2
			\mid c_1^2 (y^2-x^2)= c_2 x(x-z_1)(x-z_2) 
		\right\}.
$$
We also assume, cf. Lemma \ref{lem:relationc1c2z1z2}, that 
\begin{equation} \label{ineq:ellipticcurve}
\frac{c_1^2}{c_2}\geq (\sqrt{ z_1}+\sqrt{ z_2})^2.
\end{equation}
This inequality  implies that we have three points on the curve whose $y$ coordinate is zero, $(0,0)$, $(-t_1,0)$ and $(-t_2,0)$, with $t_1,t_2>0$. Moreover, the curve $\mathcal E$ has two connected components 
$$
    \mathcal E_\pm =\left \{(x,y) \in \mathcal E \mid \pm x\geq 0\right\},
$$ 
one in the left half plane and the other in the right half plane. It will also be important for us that the lines $y= \pm x$ lie above and below $\mathcal E_-$, meaning that $y^2-x^2<0$ and thus $|y|/|x|<1$. Indeed, the lines $y=\pm x$ intersect $\mathcal E$ at most at three points, and we already established that these points are on $\mathcal E_+$. This implies that $\mathcal E_-$ has to lie fully below or above each of these lines and  since  $(-t_1,0)$ and $(-t_2,0)$ lie below  the line $y=-x$ and above the line $y=x$, so does $\mathcal E_-$.

	\begin{figure}[t]
		\begin{center}
	\includegraphics[scale=.6]{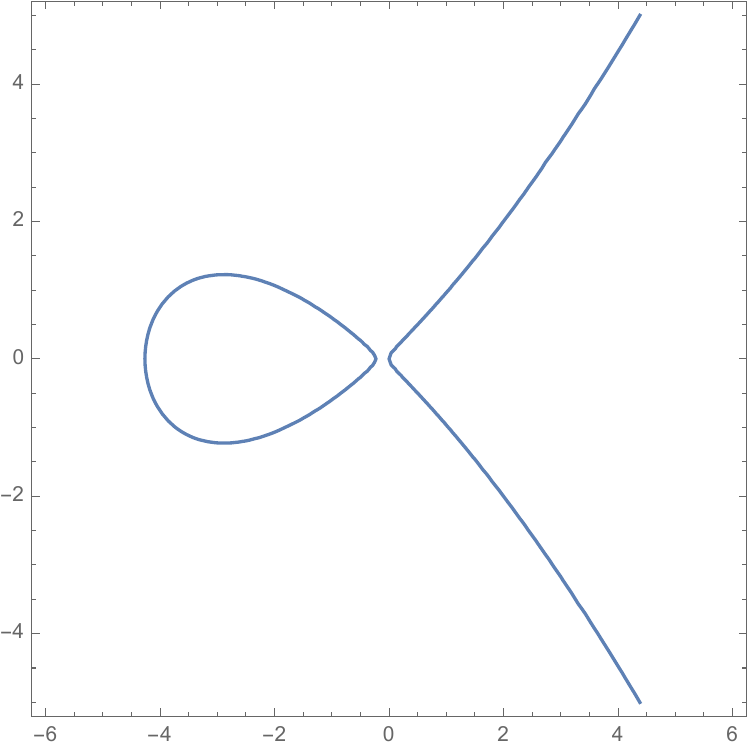}
	\caption{An example with parameter $z_1=\frac12$, $z_2=2$ and $c_1^2/c_2=7$. Under strict inequality in \eqref{ineq:ellipticcurve} we always have an oval in the left half plane. In case we have equality, the oval has shrunk to a point. }
\end{center}
	\end{figure}

There is a classical construction of addition on an elliptic curve which we will use. We can add  two points $(x_1,x_1), (x_2,y_2) \in \mathcal E$ as follows: generically, the line through $(x_1,y_1)$ and $(x_2,y_2)$ intersects the elliptic curve at exactly one point $(x_3,-y_3)$.  Then we define $(x_1,y_1)+(x_2,y_2)= (x_3,y_3)$. One exception is when  $(x_2,y_2)= (x_1,y_1)$ (in which case the addition becomes a doubling of the point), but this can be defined by continuity. The other exception  is $(x_1,y_1)+(x_1,-y_1)$ which we define to be the point at infinity.  The addition turns $\mathcal E$ into group with  the point at infinity as  zero. 

We will be mostly interested in translation by $(z_2,z_2)$   on $\mathcal E$.
Observe that if $(x,y) \in \mathcal E_-$ then $(x,y) +(z_2,z_2) \in \mathcal E_-$. We will  define the translation operator
$$\sigma : \mathcal E_- \to \mathcal E_- : (x,y)\mapsto (x,y)+(z_2,z_2).$$
It is not hard to put this into a concrete formula. Since it will be useful to have this formula at hand, and in order to simplify arguments later, we include it in the following lemma.
\begin{lemma} \label{lem:additionellipticcurve} We have
$$
\sigma (x,y)= \left(\frac{z_2(x-z_1)(y-x)}{(x-z_2)(x+y)}, \frac{z_2(y-x) (x^2+y (z_1-z_2)-z_1z_2)}{(x-z_2)^2(x+y)}\right),
$$
for all $(x,y) \in \mathcal E_-$. 
\end{lemma}
\begin{proof} 
	The line through the point $(x,y)$ and $(z_2,z_2)$ is given by the formula $Y= \lambda (X-z_2)+z_2$ where $\lambda= \frac{y-z_2}{x-z_2}$. By substituting this into the equation for $\mathcal E$, moving all terms to the right-hand side and collecting the coefficient of $X^2$ we obtain 
	$$- \lambda^2c_1^2+c_1^2-c_2 (z_1+z_2),$$
	and this equals $-c_2$ times the sum of the three zeros of the resulting cubic equation for $X$. In other words, after setting $(x^*,-y^*)=(x,y)+(z_2,z_2)$ we have 
	$$-c_2(x^*+z_2+x)= -c_1^2 \lambda^2+c_1^2-c_2 (z_1+z_2).
	$$
	Thus,
	$$x^*=\frac{c_1^2}{c_2}(\lambda^2-1) +z_1-x=\frac{c_1^2}{c_2}\left(\frac{(y-x)(x+y-2z_2)}{(x-z_2)^2}\right) +z_1-x.$$
		Now use the fact that $(x,y) \in \mathcal E$ to find 
		\begin{multline*}x^*= \frac{(x-z_1) x (x+y-2z_2)}{(x-z_2)(x+y)}+z_1-x\\
		= \frac{(x-z_1)}{(x-z_2)(x+y)}\left(x(x+y-2z_2)-(x-z_2)(x+y)\right)
		= \frac{z_2(x-z_1)(y-x)}{(x-z_2)(x+y)}.
		\end{multline*}
		Inserting this back into $y^*= \lambda (x^*-z_2)+z_2$ we find 
		\begin{multline*}
            y^*=
            z_2 \left(\left(\frac{(y-z_2)(x-z_1)(y-x)}{(x-z_2)^2(x+y)}-\frac{y-z_2}{x-z_2}\right)+1\right)
            =z_2 \left(\frac{(y-z_2)(x-z_1)(y-x)}{(x-z_2)^2(x+y)}+\frac{x-y}{x-z_2}\right),
        \end{multline*}
		and further simplification shows
		$$y^*= \frac{z_2(y-x) \left((y-z_2)(x-z_1)-(x+y)(x-z_2)\right)}{(x-z_2)^2(x+y)}=-\frac{z_2(y-x) \left(x^2+y(z_1-z_2)-z_1z_2\right)}{(x-z_2)^2(x+y)}.$$
		By flipping the sign of $y^*$ we thus obtain the statement. 	
\end{proof}

\subsection{Equivalence of the flows} \label{sec:eqflow}

Our main point is that the flows $s$ and $\sigma$ from Definition \ref{def:flow} and Lemma \ref{lem:additionellipticcurve} are equivalent.  We start with the following. 
\begin{proposition}
	The map $\pi: (0,\infty) \times  \mathcal E_- \to S(z_1,z_2,c_1,c_2)$ defined by 
\begin{equation} \label{eq:defpi}
	\pi (u,(x,y) ) = 
	\begin{pmatrix} 
			c_1\left(1-\frac{y}{x} \right) & u(z-x)\\
		\frac{c_2}{u} \left(1-\frac{z_1z_2}{ x z}\right)  			
        & c_1 \left(1+\frac  y x \right)
	\end{pmatrix}
\end{equation}
is  well-defined and a bijection.
\end{proposition}
\begin{proof}
    First, since $x<0$ and $|y|<|x|$ for $(x,y) \in \mathcal E_-$ we see that all entries and coefficients of $\pi(u,(x,u))$ are positive and thus $\pi(u,(x,y)) \in \mathcal S$. To see that $\pi(u,(x,y)) \in  S(z_1,z_2,c_1,c_2)$ we have to check that the defining equations match. To this end, we note that 
$$\Tr \pi (u,(x,y))=2c_1,$$ 
and 
\begin{multline} \label{eq:detPiuxy}
    \det  \pi (u,(x,y))= c_1^2\left(1-\frac{y^2}{x^2} \right)-c_2(z-x)\left(1-\frac{z_1z_2}{ x z}\right)\\
    = \frac{c_1^2\left(x^2-y^2 \right)+c_2 x(x-z_1)(x-z_2) }{x^2}-\frac{c_2 (z-z_1)(z-z_2)}{z}. 
\end{multline}
Hence, 
$$
    \det  \pi (u,(x,y))=-c_2 (z-z_1)(z-z_2)/z
$$
if and only if $(x,y) \in \mathcal E_-$ (note that we already observed that $x<0$). Therefore, $\pi(u,(x,y)) \in \mathcal S(z_1,z_2,c_1,c_2)$.

To establish that $\pi$ is a bijection we construct the inverse map as follows. It is not difficult to see that any matrix from the general space $S$  can be written as in the right-hand side of \eqref{eq:defpi} after choosing $c_1,c_2,z_1,z_2$ as in \eqref{eq:equationsforc1c2} and \eqref{eq:equationsforz1z2} and $u,x,y$ as 
$$
	\begin{dcases}
		u=b_{12},\\
        x=-\frac{a_{12}}{b_{12}},\\
        y= \frac{a_{12}}{b_{12}}  \left(\frac{a_{11}-a_{22}}{a_{11}+a_{22}}\right).
	\end{dcases}
$$
 By the assumptions $a_{ij}>0$ and $b_{ij}>0$ we see that $u,c_1,c_2,z_1z_2>0,$ hence $x<0$ and $|y|<|x|$. We still need to verify that $(x,y)$ lies on the elliptic curve. But this follows from the computation of the determinant \eqref{eq:detPiuxy}. Indeed, since the determinant matches with $\det P(z)$ we must have that $(x,y) \in \mathcal E$. Since we already know that $x<0$ we  find $(x,y) \in \mathcal E_-$, and we have thus proved the statement.
\end{proof}
We now come to the key point of this section.

\begin{theorem} \label{thm:equivalence}
    For any $(u,(x,y)) \in (0,\infty)\times \mathcal E_-$ we have $\pi(u,\sigma(x,y))=s(\pi(u,(x,y)))$.
\end{theorem}
\begin{proof}
    Since $\pi$ is a bijection, there must exist $(u',(x',y')) \in (0,\infty) \times \mathcal E_-$ such that $s(\pi(u,(x,y)))=\pi(u',(x',y'))$. 
    We first compute $s(\pi(u,(x,y)))$. Note that from Proposition \ref{prop:uniquefact} and \eqref{eq:abcd_intermsof_aijbij} we have $\pi(u,(x,y))=Q_-Q_+$ with 
    \begin{equation}\label{eq:abcd}
        \begin{dcases} 
            a=\frac{uc_1(x-y)}{c_2(x-z_2)},\\
            b=-\frac{c_2z_2}{xu},\\
            c=\frac{c_2}{u},\\
            d=\frac{-u x(x-z_2)}{c_1(x-y)}.
        \end{dcases}
    \end{equation}
    We note that since $(x,y)$ is a point on the elliptic curve, we can rewrite $d$ as 
    $$
     d=\frac{uc_1(x+y)}{c_2(x-z_1)}.
    $$
    Now we can compute
    $$
     s(\pi(u,(x,y))=P_+(z)P_-(z)=D^{-1} Q_+(z)Q_-(z) D= 
     \begin{pmatrix}
        ab+\frac{bdz_1}{z_2} &  a^2b+\frac{abdz}{z_2}\\
         c+ \frac{cdz_1}{az} &  ac+cd
     \end{pmatrix}.
    $$
    To find $(u',(x',y'))$ such that $s(\pi(u,(x,y))=\pi(u',(x',y'))$ we argue as follows. From \eqref{eq:defpi} we see that $u'$ is the coefficient of $z$ in the 12-entry. This gives $u'=u, $
    so the parameter $u$ is unchanged under the flow. 

    Then $x'$ is the zero of the 12-entry viewed as a linear function in $z$ and  thus 
    $$ 
        x'=\frac{-z_2a}{d}=\frac{z_2(y-x)(x-z_1)}{(x+y)(x-z_2)}.
    $$
   Next, by looking at the 22-entry of $P_+P_-$ we find  
    $$
        c(a+d)
        =c_1\frac{(x-z_1)(x-y)+(x-z_2)(x+y)}{(x-z_1)(x-z_2)}.
    $$
    By solving for $y'$ from the 22-entry of $\pi(u',(x',y'))$, cf. \eqref{eq:defpi}, we find 
    \begin{multline*}
     y'=\left(\frac{c(a+d)}{c_1}-1\right)x'=\left(\frac{(x-z_1)(x-y)+(x-z_2)(x+y)}{(x-z_1)(x-z_2)}-1\right)\frac{z_2(y-x)(x-z_1)}{(x+y)(x-z_2)}\\
     =\left(\frac{(x-z_1)(x-y)+(x-z_2)(y+z_1)}{(x-z_1)(x-z_2)}\right)\frac{z_2(y-x)(x-z_1)}{(x+y)(x-z_2)}\\
     =\frac{z_2(x^2+y(z_1-z_2)-z_1z_2)(y-x)}{(x+y)(x-z_2)^2}.
    \end{multline*}
   Thus, $(x',y')$ matches with $(z_2,z_2)+(x,y)$ from Lemma \ref{lem:additionellipticcurve} as desired.
\end{proof}

\subsection{Wiener-Hopf factorizations} \label{sec:WH}

Let $P(z)\in \mathcal S$  with $\mathcal S$ as defined in \eqref{eq:defCalS} and $n \in \mathbb N$. In this paragraph we will show how the flows above  can be used to find an explicit Wiener-Hopf factorization 
$$ 
   (P(z))^{n+1}= P_-(z)P_+(z). 
$$
First of all, as also discussed in Section \ref{sec:intro_wh}, with $P_k(z)=s^k(P(z))$ and $P_k(z)=P_{k,-}(z)P_{k,+}(z)$ as in Definition \ref{def:flow} we can take 
$$
 P_-(z)= P_{0,-}(z)P_{1,-}(z)\cdots P_{n,-}(z),
$$
and 
$$
 P_+(z)=P_{n,+}(z)P_{n-1,+}(z)\cdots P_{0,+}(z).
$$
Then, by Theorem \ref{thm:equivalence} we can obtain an explicit representation in terms of the flow on the elliptic curve. To this end, we first define the functions (cf. \eqref{eq:abcd})
$$
  \begin{dcases}
    a(x,y)=\frac{uc_1(x-y)}{c_2(x-z_2)},\\
    b(x,y)=-\dfrac{c_2z_2}{xu}.
  \end{dcases}
$$
 Using the parametrizaton for $P(z)$ as in \eqref{eq:defpi} we then have, by Theorem \ref{thm:equivalence},
\begin{equation}
P_{j,-}(z)=b( \sigma^j(x,y))) 
    \begin{pmatrix} 
     a(\sigma^j(x,y))& 0 \\
     0 & 1
    \end{pmatrix} 
    \begin{pmatrix} 
     1& 1 \\
     \frac{z_1}{z} & 1
    \end{pmatrix}
    \begin{pmatrix} 
     1& 0 \\
     0 & a(\sigma^j(x,y)) 
    \end{pmatrix}.
 \end{equation}
Hence,
\begin{multline}
   P_{0,-}(z)P_{1,-}(z)\cdots P_{n,-}(z)=\prod_{j=0}^n  b(\sigma^j(x,y)))\\
  \times \prod_{j=0}^n 
   \begin{pmatrix} 
    a(\sigma^j(x,y))& 0 \\
    0 & 1
   \end{pmatrix} 
   \begin{pmatrix} 
    1& 1 \\
    \frac{z_1}{z} & 1
   \end{pmatrix}
   \begin{pmatrix} 
  1& 0 \\
    0 & a(\sigma^j(x,y))
   \end{pmatrix}
\end{multline}
For future reference, we note that the constant pre-factor is of no interest to us and will cancel out in the integrand for the double integral formula of Proposition \ref{prop:BD} for the correlation kernel. It is thus the evolution of $a(\sigma^j(x,y))$ that is of importance. 

Next, define the function
$$
    d(x,y)=\frac{-u x(x-z_2)}{c_1(x-y)}.
$$
Then we have
\begin{equation}
P_{j,+}(z)= 
    \begin{pmatrix} 
     1& 0 \\
     0 & \frac{c_2}{z_2 u^2}d(\sigma_j(x,y))
    \end{pmatrix} 
    \begin{pmatrix} 
     1& \frac{z}{z_2} \\
     1 & 1
    \end{pmatrix}
    \begin{pmatrix} 
     1& 0 \\
     0 & d(\sigma^j(x,y))
    \end{pmatrix}.
 \end{equation}
Hence,
\begin{equation}
   P_{n,+}(z)P_{n-1,+}(z)\cdots P_{0,+}(z)=
  \prod_{j=0}^n 
   \begin{pmatrix} 
    1 & \\
    0&  \frac{c_2}{z_2u^2}d(\sigma_j(x,y))
   \end{pmatrix} 
   \begin{pmatrix} 
    1& \frac{z}{z_2} \\
    1 & 1
   \end{pmatrix}
   \begin{pmatrix} 
    1& 0 \\
    0 & d(\sigma_j(x,y))
   \end{pmatrix}.
\end{equation}

 \section{Proofs of the main results} \label{sec:proofs}
 We now return to the model of the biased doubly periodic Aztec diamond from Section \ref{sec:dimer} and prove our main results. 

 \subsection{Proof of Theorem \ref{thm:main_result}} \label{sec:proofflow}
 \begin{proof}[Proof of Theorem \ref{thm:main_result}]
We recall from Proposition \ref{prop:BD} that we are interested in finding a factorization for 
$$
    A(z)= \frac{1}{(1-a^2/z)^N} (P(z))^N,
$$
where 
$$
P(z)=  \begin{pmatrix}
    \alpha & a \alpha z\\
    \frac{\alpha}{a}   & \frac{1}{\alpha}
\end{pmatrix} 
\begin{pmatrix}
    1 &  a\\
    \frac{a}{z} & 1
\end{pmatrix}.
$$
Comparing this with the setting of Section \ref{sec:eqflow} we see that we have the special case 
\begin{equation}\label{eq:general_to_aztec}
 \begin{dcases}
    z_1=a^2, \\
    z_2=1/a^2,\\
    c_1=\frac12 (a^2+1)(\alpha+1/\alpha),\\
	c_2=a^2,\\
    u=a \alpha,
 \end{dcases}
\end{equation}
and thus the elliptic curve can be written as 
$$
y^2-x^2=\frac{4x(x-a^2)(x-1/a^2)}{(a+1/a)^2(\alpha+1/\alpha)^2}.
$$
The flow starts with the initial parameters $(x_0,y_0)=(-1,-\frac{1-\alpha^2}{1+\alpha^2})$. The theorem is a straightforward consequence of the factorization of  Section \ref{sec:WH}. 
\end{proof}

\subsection{Proof of Lemma \ref{lem:spectral_factorization_intro}} \label{sec:prooflemmasp}

\begin{proof}[Proof of Lemma \ref{lem:spectral_factorization_intro}]
    It is readily verified that \eqref{eq:spectral_decomposition_Pz_intro} holds.  An important observation is that 
 $$
      \left(P^{(d)}_-(z)P^{(d)}_+(z)\right)^2=(P(z))^{2d}=\left(P^{(d)}_-(z)\right)^2 \left(P^{(d)}_+(z)\right)^2.
 $$
 This implies that $P(z)^d$, $P^{(d)}_-(z)$ and $P^{(d)}_+(z)$ commute\footnote{We are grateful to Tomas Berggren for reminding us of this fact}, and therefore are simultaneously diagonalizable. Hence, we can write $P_\pm^{(d)}(z)$ as in   \eqref{eq:spectral_decomposition_Pminus_intro} and \eqref{eq:spectral_decomposition_Pplus_intro}. Furthermore, note that we can rewrite \eqref{eq:spectral_decomposition_Pminus_intro} and  \eqref{eq:spectral_decomposition_Pplus_intro} as 
 \begin{equation} \label{eq:munuinE}
    E(z)^{-1}P^{(d)}_-(z)E(z)=  
 \begin{pmatrix}
     \mu_1(z) & 0 \\
  0 & \mu_2(z)
 \end{pmatrix}, \quad E(z)^{-1}P^{(d)}_+(z)E(z)=  
 \begin{pmatrix}
     \nu_1(z) & 0 \\
  0 & \nu_2(z)
 \end{pmatrix},
\end{equation}
with $E(z)$ as in \eqref{eq:general_eigenvectors_intro}. Now the entries of $E(z)$ and $E(z)^{-1}$ are meromorphic functions for $z \in \mathbb C \setminus \left((-\infty,x_1]\cup [x_2,0]\right)$.  From \eqref{eq:munuinE} we then see that $\mu_{1,2}$ and $\nu_{1,2}$ are also meromorphic for $z \in \mathbb C \setminus \left((-\infty,x_1]\cup [x_2,0] \right)$. Now,  on the cuts $(-\infty,x_1]\cup [x_2,0]$ we have 
$$  
    E_+(z)=E_-(z) \begin{pmatrix} 0 & 1 \\ 1 & 0 \end{pmatrix},
$$ 
where $E_\pm (z)= \lim_{\eps \downarrow 0}E(z\pm \eps i)$. This implies that, for $z \in (-\infty,x_1)\cup (x_2,0)$, we have 
$$
\mu_{1,\pm}(z)=\mu_{2,\mp}(z), \qquad \nu_{1,\pm}(z)=\nu_{2,\mp}(z),
$$
where $\mu_{j,\pm}= \lim_{\eps\downarrow 0} \mu_j(z+\eps i)$ and  $\nu_{j,\pm}= \lim_{\eps\downarrow 0} \nu_j(z+\eps i)$. Therefore, we see that
the functions $\mu$ defined by $\mu(z^{(j)})=\mu_j(z)$ and, similarly, $\nu$ defined  $\nu(z^{(j)})=\nu_j(z)$ extend to meromorphic functions on $\mathcal R$. 

Clearly, $\mu$ and $\nu$ must satisfy \eqref{eq:spectralfactorization_intro}.

What remains is the statement on the zeros and poles of $\nu$ and $\mu$. By \eqref{eq:spectral_curve_nu_intro}, any pole of $\nu$  is a pole of  $\Tr P_+^{(d)}(z)$ and/or $\det P_+^{(d)}(z)$. Since $\Tr P_+^{(d)}(z)$ can only possibly have  a pole at  $z=\infty$, and $ \det P_+^{(d)}(z)$ has exactly one pole which is at $z=\infty$ of degree $d$,  we see that  $\nu$ has a pole at the branch point $z=\infty$ of degree $d$ and no other. The zeros of $\nu$ can then be determined from the zeros of $ \det P_+^{(d)}(z)$, and this shows that the only possible locations of the zeros are $z=(a^{-2})^{(1)}$ and $z=(a^{-2})^{(2)}$, where the sum of the orders equals $d$. By \eqref{eq:spectralfactorization_intro} and the fact that $\lambda$ has no zero at $z=(a^{-2})^{(1)}$, it follows that $\nu$ has  a zero at $z=(a^{-2})^{(2)}$ of order $d$. The poles and zeros of $\mu$ can be determined analogously. 
\end{proof}

\subsection{Proof of Theorem \ref{thm:correlationkernel_spectralcurve}} \label{sec:prooftheoremp}
\begin{proof}[Proof of Theorem \ref{thm:correlationkernel_spectralcurve}]
    Note that by \eqref{eq:defF} we can rewrite the spectral decomposition \eqref{eq:spectral_decomposition_Pz_intro} as
    $$
     P(w)=F(w^{(1)}) \lambda(w^{(1)})+ F(w^{(2)}) \lambda(w^{(2)}),
    $$
    and, similarly for $P_+(w)$, 
    $$
    P_+^{(d)}(w)=F(w^{(1)}) \nu(w^{(1)})+ F(w^{(2)}) \nu(w^{(2)}).
    $$
    Combining this with $F(w^{(1)})F(w^{(2)}) =\mathbb O$ (the zero matrix), we see that 
    \begin{multline}\label{eq:powerPinspectral1}
        P(w)^{-m'}P(w)^{d(N-T)}(P_+^{(d)}(w))^{-N}\\
            =F(w^{(1)}) \lambda(w^{(1)})^{d(N-T)-m'} \nu(w^{(1)})^{-N}+ F(w^{(2)}) \lambda(w^{(2)})^{d(N-T)-m'} \nu(w^{(2)})^{-N}\\
            =F(w^{(1)}) \lambda(w^{(1)})^{-m'}  \mu(w^{(1)})^{N-T}\nu(w^{(1)})^{-T}+ F(w^{(2)}) \lambda(w^{(2)})^{-m'}  \mu(w^{(2)})^{N-T}\nu(w^{(2)})^{-T}.
    \end{multline}
    In the same way, 
    \begin{multline} \label{eq:powerPinspectral2}
        (P_-^{(d)}(z))^{-N} P(z)^{dT}P(z)^m\\
        =F(z^{(1)}) \lambda(z^{(1)})^{m}  \mu(z^{(1)})^{T-N}\nu(z^{(1)})^{T}+ F(z^{(2)}) \lambda(z^{(2)})^{m}  \mu(z^{(2)})^{T-N}\nu(z^{(2)})^{T}.
    \end{multline}
    By substituting \eqref{eq:powerPinspectral1} and \eqref{eq:powerPinspectral2} in the double integral of \eqref{eq:correlationkernelmainperiodic} (with adjusted parameters) and inserting
    $$
        (P(z))^{m-m'}=F(z^{(1)}) \lambda(z^{(1)})^{m'-m}+ F(z^{(2)}) \lambda(z^{(2)})^{m'-m}
        $$
        in the single integral one obtains the statement. \end{proof}

   \subsection{Proof of Proposition \ref{prop:four_saddles}} \label{sec:prooffoursaddle}
   \begin{proof}[Proof of Proposition \ref{prop:four_saddles}]
    One can easily see that $\Phi'(z)dz$ has simple poles at $0$ and  $\infty$. On the first sheet $\Phi'(z)$ takes the form 
    \begin{equation} \label{eq:Phi1onfirst}
    \frac{(1-\tau)\mu'_1(z)}{\mu_1(z)} -\frac{\tau \nu_1'(z)}{ \nu_1(z)} +\frac{d(1-\tau+\xi)}{z}  -\frac{d(1-\tau)}{z-a^2},
    \end{equation}
    and we see that we have a simple pole at $(a^{2})^{(1)}$. On the second sheet we can use the relations $\nu_1(z)\nu_2(z)=const\cdot(z-a^2)^d$ and $\mu_1(z)\mu_2(z)=const\cdot(z-a^{{-2}})^d/z^d$,  to deduce that $\Phi'(z)$ takes the form
    \begin{equation}\label{eq:Phi1onsecond}
    -\frac{(1-\tau)\mu'_1(z)}{\mu_1(z)} +\frac{\tau \nu_1'(z)}{ \nu_1(z)} +\frac{d\xi}{z}  -\frac{d\tau }{z-a^{-2}},
    \end{equation}
    and the pole at $(a^{2})^{(2)}$ gets canceled at the cost of a  new simple pole at $(a^{-2})^{(2)}$. Thus, $\Phi'(z)dz$ has four simple poles at said locations and thus also four zeros (since $\mathcal R(z)$ is of genus 1). 

    We now show that there are at least two saddle points in $\mathcal C_1$, which can be done using the same argument as in \cite[proof of Proposition 6.4]{DK}. The point is that one can show that
    \begin{equation} \label{eq:cyclecondition}
      \oint_{\mathcal C_1} \Phi'(z)dz=0.
    \end{equation}
    Indeed, since $\nu_{1}(z)$ and $\mu_{1}(z)$ are real-valued for $z \in (x_1,x_2)$, so is $\Phi'(z^{(1,2)})$ by \eqref{eq:Phi1onfirst} and \eqref{eq:Phi1onsecond}, and thus
    $$
      \Im \oint_{\mathcal C_1} \Phi'(z)dz=0.
    $$
    As for the real part, note that 
    $$
    \oint_{\mathcal C_1} \Phi'(z)dz= \int_{x_1}^{x_2} \Phi'(z^{(1)})dz-\int_{x_1}^{x_2} \Phi'(z^{(2)})dz
    $$ 
    and, since $\Re \Phi$ is single-valued on $\mathcal R$,
    $$
    \Re \int_{x_1}^{x_2} \Phi'(z^{(1)})dz= \Re \Phi(x_2)- \Re \Phi(x_1)= \Re \int_{x_1}^{x_2} \Phi'(z^{(2)})dz.
    $$
  Therefore, also the real part of the left-hand side of \eqref{eq:cyclecondition} vanishes. By combining  this with the fact that $\Phi'(z) dz$ is real-valued and continuous on $\mathcal C_1$, we see that $\Phi'(z) dz$ must change sign at least two times.  This means that there are at least two zeros of $\Phi'(z)dz$. (Note that this argument does not work on $\mathcal C_2$ since $\Phi'(z)dz$ has two poles on $\mathcal C_2$.) 
\end{proof}

\subsection{Asymptotic analysis in the smooth phase} \label{sec:proofsaddle}

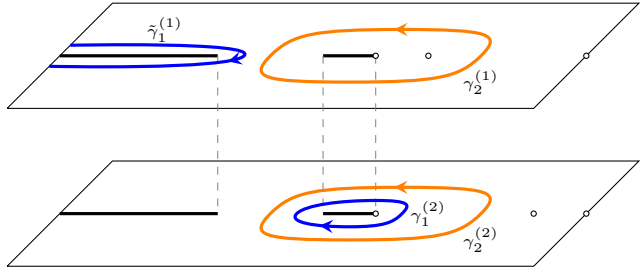
\begin{figure}[t]
    \begin{center}
    \begin{tikzpicture}[scale=.7,path/.append style={
        decoration={
            markings,
            mark=at position 0.5 with {\arrow[xshift=2.5\pgflinewidth,>=stealth]{>}}
        },
        postaction=decorate,
        thick,
      }]
        \draw (0,0)--(2,2)--(12,2)--(10,0)--(0,0);
        \draw (0,-3)--(2,-1)--(12,-1)--(10,-3)--(0,-3);
        \draw[very thick] (1,1)--(4,1);
        \draw[very thick] (6,1)--(7,1);
        \draw[very thick] (6,-2)--(7,-2);
        \draw[very thick] (1,-2)--(4,-2);
        \draw[help lines,dashed] (4,1)--(4,-2);
        \draw[help lines,dashed] (6,1)--(6,-2);
        \draw[help lines,dashed] (7,1)--(7,-2);

        \draw[path,orange, very thick] plot [smooth cycle, tension=1.5] coordinates {(5,1) (6.5,.5) (9,1) (7.5,1.5)};
        \draw[path,very thick, blue] plot [smooth cycle, tension=1.2] coordinates {(5.5,-2) (7,-1.75) (7.5,-2) (6.5,-2.25)};
        \draw[path,very thick, blue] plot [smooth, tension=2.5] coordinates {(1.2,1.2) (4.5,1) (.8,.8)};
        \draw[path,orange, very thick] plot [smooth cycle, tension=1.5] coordinates {(5,-2) (6.5,-2.5) (9,-2) (7.5,-1.5)};
        \filldraw[fill=white] (8,1) circle(.05);
        \filldraw[fill=white](10,-2) circle(.05);
        \filldraw[fill=white](7,1) circle(.05);
        \filldraw[fill=white](7,-2) circle(.05);
        \filldraw[fill=white](11,1) circle(.05);
        \filldraw[fill=white] (11,-2) circle(.05);
        \draw (9,.5) node {\tiny{$\gamma_2^{(1)}$}};
        \draw (9,-2.5) node {\tiny{$\gamma_2^{(2)}$}};
        \draw (8,-2) node {\tiny{$\gamma_1^{(2)}$}};
        \draw (3,1.5) node {\tiny{$\tilde \gamma_1^{(1)}$}};
    \end{tikzpicture}
    \caption{The first deformation of contours. The contours  $\gamma_2^{(1)}$,   $\gamma_2^{(2)}$ and $\gamma_{1}^{(2)}$  remain untouched. The blue contour $\tilde \gamma_{1}^{(1)}$ is a deformation of the contour $\gamma_{1}^{(1)}$. By deforming the contour like this, we pick up a residue at $z=w$. Note also that each  blue contour can be deformed through the cuts and be entirely, or partly, on the second sheet. The orange contour is allowed to pass the cuts provided one does not pass through the origin while deforming.}
    \label{fig:preldeform}
\end{center}
\end{figure}

We will work out the asymptotic analysis in the smooth phase. We prepare the proof of Theorem \ref{thm:gas} by first performing two steps:
\begin{enumerate}
    \item a preliminary deformation of paths.
    \item a qualitative description of the paths of steepest descent and ascent leaving from the saddle points. 
\end{enumerate}
After these steps, the asymptotic analysis follows by standard arguments. 

\subsubsection{A preliminary deformation}
We will need the following lemma on the asymptotic behavior of the integrand in \eqref{eq:correlationkernelmainperiodiceigenvalues} near the poles at $0$ and $\infty$. 
\begin{lemma}
    We have that 
    \begin{multline}\label{eq:poles?}
    \lambda(w)^{-m'}w^{x-m}(w-a^2)^{m'} \mu(w)^{N-T}\nu(w)^{-T}w^{d(T-N)+X}{(w-a^2)^{d(N-T)}}\\=\begin{cases}
        \mathcal O(|w|^{X+x'-dT/2-m'/2}), & w \to \infty,\\
        \mathcal O(|w|^{X+x'-d(N-T)/2 -m'/2}), & w \to 0.\\
    \end{cases}
\end{multline}
\end{lemma}
\begin{proof}
The behavior near $w=\infty$ follows readily after observing 
$$
\begin{cases}
    \lambda(w)=\mathcal O(|w|^{1/2}),\\
    \mu(w)=\mathcal O(1),\\
    \nu(w)=\mathcal O(|w|^{d/2}),
\end{cases}
$$
as $w\to \infty$.  

Similarly, the behavior near $w=0$ follows after observing 
$$
\begin{cases}
    \lambda(w)=\mathcal O(|w|^{-1/2}),\\
    \mu(w)=\mathcal O(|w|^{-d/2}),\\
    \nu(w)=\mathcal O(1),
\end{cases}
$$
as $w\to 0$.  \end{proof}

It is important to observe that we are considering $(\tau,\xi)$ in the parallellogram defined by 
$\tau=0$, $\tau=1$, $\xi=\tau/2$ and $\xi=(\tau-1)/2$. By \eqref{eq:point_zoom_in} this means that for any  $x',m' \in \mathbb Z$ we have that 
\begin{equation}\label{eq:nopoles:)}
    X-d(N-T)/2'+x'-m'/2>0, \quad  X-dT/2+x'-m'/2<0,
\end{equation}
for  $N$ sufficiently large, and thus the left-hand side of \eqref{eq:poles?} has no poles (and no residues) for either $w=0$ or $w=\infty$.\\

We proceed with the first contour deformation. The contours $\gamma_2^{(1)}$, $\gamma_2^{(2)}$ and $\gamma_1^{(2)}$ will be untouched, but the contour $\gamma_1^{(1)}$  is deformed to the contour $\tilde \gamma_1^{(1)}$ that goes around the cut $(-\infty,x_1)$ in clockwise direction, as indicated in Figure \ref{fig:preldeform}.  While deforming we pick up possible residues at the pole at $w=\infty$ and at $w=z$ for $z \in \gamma_2^{(1)}$. As mentioned above, with our choice of parameters there is no pole at $w=\infty$. The pole at $w=z$ has a residue for $z \in \gamma_2^{(1)}$, and this gives us a contribution:
$$
\frac{1}{2 \pi i} 
\int_{\gamma_2^{(1)}} A_e(z)^{-\eps'} F(z)A_o(z)^{\eps} \lambda(z) ^{m-m'}\frac{z^{m-x-m'+x'}}{(z-a^2)^{m-m'}} \frac{dz}{z}.
$$
This means that we can rewrite \eqref{eq:correlationkernelmainperiodiceigenvalues} as 
\begin{multline} \label{eq:correlationkernelmainperiodiceigenvaluesafterdeforming}
    \left[K_{   d N}((2dT+2m+\eps, 2X +2x-j),(2dT +2m'+ \eps',2X +2x'-j'))\right]_{j,j'=0}^1\\
    = -\frac{\mathbbm{1}_{2m'+\eps'<2m+\eps}}{2 \pi i} 
    \int_{\gamma_2^{(2)}} A_e(z)^{-\eps'} F(z)A_o(z)^{\eps} \lambda(z) ^{m-m'}\frac{z^{m-x-m'+x'}}{(z-a^2)^{m-m'}} \frac{dz}{z}\\
    +\frac{\mathbbm{1}_{2m'+\eps'\geq 2m+\eps}}{2 \pi i} 
    \int_{\gamma_2^{(1)}} A_e(z)^{-\eps'} F(z)A_o(z)^{\eps} \lambda(z) ^{m-m'}\frac{z^{m-x-m'+x'}}{(z-a^2)^{m-m'}} \frac{dz}{z}\\
    +\frac{1}{(2 \pi i)^2}
    \oint_{\tilde \gamma_1^{(1)} \cup \gamma_1^{(2)}} \oint_{\gamma_2^{(1)} \cup \gamma_2^{(2)}} A_e(w)^{-\eps'}  F(w) F(z) A_{o}(z)^{\eps} \frac{\lambda(z)^{m} }{\lambda(w)^{m'}}\frac{w^{x'-m'} }{z^{x-m}}\frac{(w-a^2)^{m'}}{(z-a^2)^{m}} \\
    \times 
    \frac{\mu(w)^{N-T}}{ \mu(z)^{N-T} }\frac{\nu(z)^{T}}{\nu(w)^{T}}
     \frac{w^{d(N-T)+X}}{z^{d(N-T)X}}\frac{(z-a^2)^{d(N-T)}}{(w-a^2)^{d(N-T)}}\frac{dw dz}{z(z-w)}.
\end{multline}
This finishes the preliminary deformation.

Before we continue to the steepest descent analysis, we mention that by \eqref{eq:poles?} and \eqref{eq:nopoles:)}, the integrand with respect to $w$ has no pole at $w=0$. This means that we can deform the contour $\gamma_1^{(2)}$ to lie partly or even entirely on the first sheet. The integrand with respect to $z$ does have poles at $z=0$ and $z= \infty$ and thus, the contours $\gamma_1^{(1)}$ and $\gamma_2^{(2)}$ may be deformed over the surface $\mathcal R$ but cannot pass through the origin (without picking up a  residue). 
\subsubsection{Description of the paths of steepest descent/ascent}

By definition, we have four saddle points in the cycle $\mathcal C_1$, and in the interior of the smooth region these are distinct and simple. By viewing $\Re \Phi$ as a function on the cycle $\mathcal C_1$, these saddle points will correspond to the locations of the two local minima and two local maxima of $\Re \Phi$. We will denote the saddles associated to the local minima by $s_1$ and $s_3$, and the local maxima by $s_2$ and $s_4$. We take the indexing such that when traversing the cycle $C_1$ starting from $x_1$ to $x_2$ on $\mathcal R_1$ and then from $x_2$ to $x_1$ on $\mathcal R_2$,  the first saddle point one encounters is $s_1$, then $s_2$ and so on. Note also that  both  local minima are neighbors to both local maxima  (on the cycle $\mathcal C_1)$ and therefore $\Re \Phi(s_{1,3})< \Re \Phi(s_{2,4})$.

We proceed by giving a description of the contours of steepest descent and ascent for  $\Re \Phi$ leaving from these four saddles.  Since each saddle point is simple, there will be two paths of steepest descent and two path of steepest ascent leaving from them.  It is  straightforward that the segment of $\mathcal C_1$ between $s_{2j}$ and $s_{2j\pm 1}$ is a path of steepest descent for $\Re \Phi$ leaving from $s_{2j}$ and a path of steepest ascent leaving from $s_{2j\pm1}$.  What remains, is to identify the paths of steepest descent leaving from $s_1$ and $s_3$ and the paths of steepest ascent from $s_2$ and $s_4$. These paths will continue in the lower and upper half planes of the sheets $\mathcal R_j$ and they are further characterized by the condition that 
$$
\Im\left[\int_{s_j}^z \Phi'(z)dz\right]=0.
$$
It is important to note that, even though $\Phi'(z)$ is single-valued, $\Phi(z)$ is a multi-valued function, and we cannot replace the condition simply with $\Im \Phi(z)=\Im \Phi(s_j)$. Indeed, because of the logarithmic terms the imaginary part $\Im \Phi(z)$ jumps whenever we cross a cut (which we did not specify) for a logarithm. The real part  $\Re \Phi(z)$, however,  is single-valued, and this will be important. We will also need the behavior near the logarithmic singularities of $\Re \Phi(z)$ at $z=0$, $z=(a^{2})^{(1)}$, $z=(a^{-2})^{(2)}$ and at $z=\infty$: from \eqref{eq:defPhi} (see also  \eqref{eq:Phi1onfirst} and \eqref{eq:Phi1onsecond}) 
\begin{equation} \label{eq:logsingularitiesmin}
    \Re \Phi(0)=\Re \Phi(\infty)=-\infty,
\end{equation} 
and 
\begin{equation} \label{eq:logsingularitiesplus}
    \Re \Phi((a^{2})^{(1)})=\Re \Phi((a^{-2})^{(2)})=+ \infty.
\end{equation}

 By analyticity of $\Phi'(z)$ the paths are a finite union of analytic arcs and ultimately  have to end up at some special points. The only options for such special points are other saddle points or the poles of $\Phi'$. It takes only a short argument to exclude  possibility that they will connect to another saddle point. Indeed, since $\Re \Phi(s_{1,3})< \Re \Phi(s_{2,4})$ it is impossible to connect $s_{2j}$ to $s_{2j\pm 1}$ in this way. Moreover, it is obvious that  a path of steepest descent (or ascent) from  the global minimum (or maximum) cannot be connected to any other saddle point, hence $s_1$ cannot be connected  to $s_3$ and $s_2$ cannot be connected  to $s_4$. We conclude so far that the path of steepest descent leaving from $s_{1,3}$ and the paths of steepest ascent from $s_{2,4}$ will have to end up at the four simple poles of $\Phi$. From \eqref{eq:logsingularitiesmin} we further deduce $s_{1,3}$ connect to $\infty$ and $0$, and from \eqref{eq:logsingularitiesmin} we find that $s_{2,4}$ connect to the  simple poles at $(a^2)^{(1)}$ and $(1/a^{2})^{(2)}$. 
 \begin{figure}
    \begin{center}
    \begin{tikzpicture}[scale=.6]

        \filldraw[fill opacity=.2,very thick, orange](6.5,1) .. controls (6.2,0.9) and (4.5,0.5) .. (5,1);
        \filldraw[fill opacity=.2,very thick, orange](6.5,1) .. controls (6.2,1.1) and (5.5,1.5) .. (5,1);
        \filldraw[very thick,orange, fill opacity=.2] plot [smooth cycle, tension=1.2] coordinates {(5,-2) (9.5,-1.5) (10,-2) (8.5,-2.5)};
        \filldraw[very thick, draw = orange,fill =white] (6.5,-2) .. controls (6.8,-2.1) and (9,-2.5) .. (10,-2);
        \filldraw[very thick, draw= orange, fill=white] (6.5,-2) .. controls (6.8,-1.9) and (8,-1.5) .. (10,-2);

        \draw (0,0)--(2,2)--(12,2)--(10,0)--(0,0);
        \draw (0,-3)--(2,-1)--(12,-1)--(10,-3)--(0,-3);
        \draw[very thick] (1,1)--(4,1);
        \draw[very thick] (6,1)--(7,1);
        \draw[very thick] (6,-2)--(7,-2);
        \draw[very thick] (1,-2)--(4,-2);
        \draw[help lines,dashed] (4,1)--(4,-2);
        \draw[help lines,dashed] (6,1)--(6,-2);
        \draw[help lines,dashed] (7,1)--(7,-2);

        \filldraw (4.5,1) circle(.07);
        \filldraw (5,1) circle(.07);
        \filldraw (5.5,-2) circle(.07);
        \filldraw (5,-2) circle(.07);

        \filldraw[fill=white](8,1) circle(.07);
        \filldraw[fill=white](10,-2) circle(.07);
        \filldraw[fill=white](7,1) circle(.07);
        \filldraw[fill=white](7,-2) circle(.07);
        \filldraw[fill=white](11,1) circle(.07);
        \filldraw[fill=white](11,-2) circle(.07);
    \end{tikzpicture}

        \caption{An illustration of the hypothetical case that the two saddle points $s_2$ and $s_4$ connect to the same saddle point $(1/a^2)^{(2)}$. In this case, the four paths together form a contractible curve and enclose the (shaded) region  that contains $s_3$ but not the cycle $\mathcal C_2$. This means that the steepest descent paths leaving $s_3$ will have to cross the paths from $s_2$ or $s_4$, which is not possible.   }
        \label{fig:path_steep_descent_ascent_impossible_case}
    \end{center}
\end{figure}
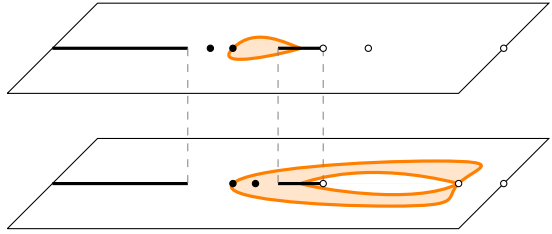

 Observe that none of these paths can cross each other, since by analyticity of $\Phi$ such a crossing  necessarily would be a saddle point (which we already excluded) or a pole. For the same reason, since $\Im \int ^z \Phi'(s)ds$ is constant on the cycles, the paths can never cross the cycles $\mathcal C_1$ and $\mathcal C_2$  as the point where it would cross would necessarily be a saddle point. The only point the paths have in common with the cycles are the saddle point at $\mathcal C_1$ they started at and the pole at $\mathcal C_2$ they end in. Hence, a path that starts at a saddle point on $\mathcal R_1$  and continues in the upper half plane of $\mathcal R_1$ always remain in the union of the upper half plane of $\mathcal R_1$ and the lower half plane of  $\mathcal R_2$ glued together along the cuts $(-\infty,x_1) \cup [x_1,x_2]$.  This   important property shows, in particular,  that steepest ascent/descent paths  do not wind around the poles of $\Phi'(z)$. 

Next we argue that the paths of steepest ascent leaving from $s_2$ and $s_4$ cannot end in the same pole. Indeed, if they would, then all these four paths together would form a closed loop that is contractible and hence cuts the Riemann surface into two parts, as illustrated in Figure \ref{fig:path_steep_descent_ascent_impossible_case}. The cycle $\mathcal C_2$ lies fully in one of the parts. But $s_1$ and $s_3$ are in different parts, and hence there must be  one of them that is in a part that is different from the part that contains the cycle $\mathcal C_2$. The  steepest descent path  that leaves that saddle point has  to cross the closed loop in order to end up at a pole on $\mathcal C_2$, which is not possible, and we arrive at a contradiction. This means that the steepest ascent paths from $s_2$ and $s_4$ have to end up at different poles, one saddle connects to $(a^2)^{(1)}$ and the other to $(1/a^2)^{(2)}$. A similar argument shows that one of the saddle points  $s_1$ and $s_3$ connect, via steepest descent paths, to  $0$ and the other to $\infty$.  

Let us summarize our findings above:
\begin{proposition}
     The steepest descent paths leaving from $s_1$ and $s_3$ and steepest ascent path from $s_2$ and $s_4$ form  simple closed loops  on $\mathcal R$, such that no two loops intersect, each loop intersects each cycle $\mathcal C_1$ and $\mathcal C_2$ exactly once, and it does so at a saddle point  for $\Re \Phi$ in $\mathcal C_1$ and a pole for $\Phi'(z)$ in $\mathcal C_2$. 
\end{proposition}

 See  Figure \ref{fig:path_steep_descent_ascent} and its caption for an illustration.

 \begin{figure}
    \begin{center}
        \begin{subfigure}{.4 \textwidth}
    \begin{tikzpicture}[scale=.5]
        \draw (0,0)--(2,2)--(12,2)--(10,0)--(0,0);
        \draw (0,-3)--(2,-1)--(12,-1)--(10,-3)--(0,-3);
        \draw[very thick] (1,1)--(4,1);
        \draw[very thick] (6,1)--(7,1);
        \draw[very thick] (6,-2)--(7,-2);
        \draw[very thick] (1,-2)--(4,-2);
        \draw[help lines,dashed] (4,1)--(4,-2);
        \draw[help lines,dashed] (6,1)--(6,-2);
        \draw[help lines,dashed] (7,1)--(7,-2);

        \draw[very thick, orange]plot [smooth cycle, tension=1.5] coordinates {(5,1) (7,1.5) (8,1) (6,.5)};
        \draw[very thick, blue] plot [smooth cycle, tension=1.2] coordinates {(5.5,1) (6.5,1.3) (7,1) (6,.7)};
        \draw[very thick, blue] plot [smooth, tension=1.2] coordinates {(3,2) (4.5,1) (1,0)};
        \draw[very thick,orange] plot [smooth, tension=1] coordinates {(5.75,1) (6,1.1) (6.5,1)};
        \draw[very thick,orange] plot [smooth, tension=1] coordinates {(5.75,1) (5.75,0.9) (6.5,1)};

        \draw[very thick,orange]plot [smooth, tension=1] coordinates {(6.5,-2) (9.5,-1.5) (10,-2)};
        \draw[very thick,orange]plot [smooth, tension=1] coordinates {(6.5,-2) (8.5,-2.5) (10,-2)};

        \filldraw (4.5,1) circle(.07);
        \filldraw (5,1) circle(.07);
        \filldraw (5.5,1) circle(.07);
        \filldraw (5.75,1) circle(.07);

        \filldraw[fill=white] (8,1) circle(.07);
        \filldraw[fill=white] (10,-2) circle(.07);
        \filldraw[fill=white] (7,1) circle(.07);
        \filldraw[fill=white] (7,-2) circle(.07);
        \filldraw[fill=white] (11,1) circle(.07);
        \filldraw[fill=white] (11,-2) circle(.07);
    \end{tikzpicture} 
    \caption{}
\end{subfigure}
\hfill
\begin{subfigure}{.4 \textwidth}
    \begin{tikzpicture}[scale=.5]
        \draw (0,0)--(2,2)--(12,2)--(10,0)--(0,0);
        \draw (0,-3)--(2,-1)--(12,-1)--(10,-3)--(0,-3);
        \draw[very thick] (1,1)--(3.5,1);
        \draw[very thick] (6,1)--(7,1);
        \draw[very thick] (6,-2)--(7,-2);
        \draw[very thick] (1,-2)--(3.5,-2);
        \draw[help lines,dashed] (3.5,1)--(3.5,-2);
        \draw[help lines,dashed] (6,1)--(6,-2);
        \draw[help lines,dashed] (7,1)--(7,-2);

        \draw[very thick, orange]plot [smooth cycle, tension=1.5] coordinates {(5,1) (7,1.5) (8,1) (6,.5)};
        \draw[very thick, blue] plot [smooth cycle, tension=1.2] coordinates {(5.5,1) (6.5,1.3) (7,1) (6,.7)};
        \draw[very thick, blue] plot [smooth, tension=1.2] coordinates {(3.5,2) (4.5,1) (1.5,0)};
        \draw[very thick,orange]plot [smooth, tension=1] coordinates {(4,1) (3.75,1.5) (2.5,1)};
        \draw[very thick,orange]plot [smooth, tension=1] coordinates {(4,1) (2.75,0.5) (2.5,1)};

        \draw[very thick,orange]plot [smooth, tension=1] coordinates {(2.5,-2) (4,-1.5) (10,-2)};
        \draw[very thick,orange]plot [smooth, tension=1] coordinates {(2.5,-2) (3,-2.5) (10,-2)};

        \filldraw (4.5,1) circle(.07);
        \filldraw (5,1) circle(.07);
        \filldraw (5.5,1) circle(.07);
        \filldraw (4,1) circle(.07);

        \filldraw[fill=white] (8,1) circle(.07);
        \filldraw[fill=white] (10,-2) circle(.07);
        \filldraw[fill=white] (7,1) circle(.07);
        \filldraw[fill=white] (7,-2) circle(.07);
        \filldraw[fill=white] (11,1) circle(.07);
        \filldraw[fill=white] (11,-2) circle(.07);
    \end{tikzpicture} 
    \caption{}
\end{subfigure}

    \begin{subfigure}{.4 \textwidth}
    \begin{tikzpicture}[scale=.5]
        \draw (0,0)--(2,2)--(12,2)--(10,0)--(0,0);
        \draw (0,-3)--(2,-1)--(12,-1)--(10,-3)--(0,-3);
        \draw[very thick] (1,1)--(4,1);
        \draw[very thick] (6,1)--(7,1);
        \draw[very thick] (6,-2)--(7,-2);
        \draw[very thick] (1,-2)--(4,-2);
        \draw[help lines,dashed] (4,1)--(4,-2);
        \draw[help lines,dashed] (6,1)--(6,-2);
        \draw[help lines,dashed] (7,1)--(7,-2);

        \draw[very thick, orange]plot [smooth cycle, tension=1.5] coordinates {(5,1) (7,1.5) (8,1) (6,.5)};
        \draw[very thick, blue] plot [smooth cycle, tension=1.2] coordinates {(5.5,1) (6.5,1.3) (7,1) (6,0.7)};
        \draw[very thick, blue] plot [smooth, tension=1.2] coordinates {(3.5,2) (4.5,1) (1.5,0)};
        \draw[very thick,orange]plot [smooth cycle, tension=1.2] coordinates {(5,-2) (8,-1.5) (10,-2) (7,-2.5)};

        \filldraw (4.5,1) circle(.07);
        \filldraw (5,1) circle(.07);
        \filldraw (5.5,1) circle(.07);
        \filldraw (5,-2) circle(.07);

        \filldraw[fill=white] (8,1) circle(.07);
        \filldraw[fill=white] (10,-2) circle(.07);
        \filldraw[fill=white] (7,1) circle(.07);
        \filldraw[fill=white] (7,-2) circle(.07);
        \filldraw[fill=white] (11,1) circle(.07);
        \filldraw[fill=white] (11,-2) circle(.07);
    \end{tikzpicture} \caption{}
\end{subfigure}\hfill 
\begin{subfigure}{.4 \textwidth} 
    \begin{tikzpicture}[scale=.5]
        \draw (0,0)--(2,2)--(12,2)--(10,0)--(0,0);
        \draw (0,-3)--(2,-1)--(12,-1)--(10,-3)--(0,-3);
        \draw[very thick] (1,1)--(4,1);
        \draw[very thick] (6,1)--(7,1);
        \draw[very thick] (6,-2)--(7,-2);
        \draw[very thick] (1,-2)--(4,-2);
        \draw[help lines,dashed] (4,1)--(4,-2);
        \draw[help lines,dashed] (6,1)--(6,-2);
        \draw[help lines,dashed] (7,1)--(7,-2);
        \draw[very thick, orange]plot [smooth cycle, tension=1.5] coordinates {(4.5,1) (7,1.5) (8,1) (6,.5)};
        \draw[very thick, blue] plot [smooth cycle, tension=1.2] coordinates {(5,1) (6.5,1.3) (7,1) (6,.7)};
        \draw[very thick, blue] plot [smooth, tension=1.2] coordinates {(3,-1) (4.5,-2) (1,-3)};
        \draw[very thick,orange]plot [smooth, tension=1] coordinates {(5.5,1) (6,1.1) (6.5,1)};
        \draw[very thick,orange]plot [smooth, tension=1] coordinates {(5.5,1) (5.75,0.9) (6.5,1)};

        \draw[very thick,orange]plot [smooth, tension=1] coordinates {(6.5,-2) (9.5,-1.5) (10,-2)};
        \draw[very thick,orange]plot [smooth, tension=1] coordinates {(6.5,-2) (8.5,-2.5) (10,-2)};

        \filldraw (4.5,-2) circle(.07);
        \filldraw (4.5,1) circle(.07);
        \filldraw (5,1) circle(.07);
        \filldraw (5.5,1) circle(.07);

        \filldraw[fill=white] (8,1) circle(.07);
        \filldraw[fill=white] (10,-2) circle(.07);
        \filldraw[fill=white] (7,1) circle(.07);
        \filldraw[fill=white] (7,-2) circle(.07);
        \filldraw[fill=white] (11,1) circle(.07);
        \filldraw[fill=white] (11,-2) circle(.07);
    \end{tikzpicture} 
    \caption{}
\end{subfigure}

    \begin{subfigure}{.4 \textwidth}
    \begin{tikzpicture}[scale=.5]
        \draw (0,0)--(2,2)--(12,2)--(10,0)--(0,0);
        \draw (0,-3)--(2,-1)--(12,-1)--(10,-3)--(0,-3);
        \draw[very thick] (1,1)--(3.5,1);
        \draw[very thick] (6,1)--(7,1);
        \draw[very thick] (6,-2)--(7,-2);
        \draw[very thick] (1,-2)--(3.5,-2);
        \draw[help lines,dashed] (3.5,1)--(3.5,-2);
        \draw[help lines,dashed] (6,1)--(6,-2);
        \draw[help lines,dashed] (7,1)--(7,-2);

        \draw[very thick, orange]plot [smooth cycle, tension=1.5] coordinates {(5,1) (7,1.5) (8,1) (6,.5)};
        \draw[very thick, blue] plot [smooth cycle, tension=1.2] coordinates {(5.5,-2) (6.5,-1.7) (7,-2) (6,-2.3)};
        \draw[very thick, blue] plot [smooth, tension=1.2] coordinates {(3.5,2) (4.5,1) (1.5,0)};
        \draw[very thick,orange]plot [smooth, tension=1] coordinates {(4,1) (3.75,1.5) (2.5,1)};
        \draw[very thick,orange]plot [smooth, tension=1] coordinates {(4,1) (2.75,0.5) (2.5,1)};

        \draw[very thick,orange]plot [smooth, tension=1] coordinates {(2.5,-2) (4,-1.25) (10,-2)};
        \draw[very thick,orange]plot [smooth, tension=1] coordinates {(2.5,-2) (3,-2.75) (10,-2)};

        \filldraw (4.5,1) circle(.07);
        \filldraw (5,1) circle(.07);
        \filldraw (5.5,-2) circle(.07);
        \filldraw (4,1) circle(.07);

        \filldraw[fill=white] (8,1) circle(.07);
        \filldraw[fill=white] (10,-2) circle(.07);
        \filldraw[fill=white] (7,1) circle(.07);
        \filldraw[fill=white] (7,-2) circle(.07);
        \filldraw[fill=white] (11,1) circle(.07);
        \filldraw[fill=white] (11,-2) circle(.07);
    \end{tikzpicture}
    \caption{}
\end{subfigure} \hfill   \begin{subfigure}{.4 \textwidth}
    \begin{tikzpicture}[scale=.5]
        \draw (0,0)--(2,2)--(12,2)--(10,0)--(0,0);
        \draw (0,-3)--(2,-1)--(12,-1)--(10,-3)--(0,-3);
        \draw[very thick] (1,1)--(3.5,1);
        \draw[very thick] (6,1)--(7,1);
        \draw[very thick] (6,-2)--(7,-2);
        \draw[very thick] (1,-2)--(3.5,-2);
        \draw[help lines,dashed] (3.5,1)--(3.5,-2);
        \draw[help lines,dashed] (6,1)--(6,-2);
        \draw[help lines,dashed] (7,1)--(7,-2);

        \draw[very thick, blue] plot [smooth, tension=1.2] coordinates {(3.5,2) (5,1) (1.5,0)};
        \draw[very thick,orange]plot [smooth, tension=1] coordinates {(4.5,1) (4,1.5) (2.5,1)};
        \draw[very thick,orange]plot [smooth, tension=1] coordinates {(4.5,1) (3.25,0.5) (2.5,1)};

        \draw[very thick,blue] plot [smooth, tension=1] coordinates {(4,1) (3.5,1.25) (3,1)};
        \draw[very thick,blue] plot [smooth, tension=1] coordinates {(4,1) (3.25,0.75) (3,1)};

        \draw[very thick,orange]plot [smooth, tension=1] coordinates {(2.5,-2) (4,-1.25) (10,-2)};
        \draw[very thick,orange]plot [smooth, tension=1] coordinates {(2.5,-2) (3,-2.75) (10,-2)};

        \draw[very thick,orange]plot [smooth, tension=1] coordinates {(4.5,-2) (5.75,-1.7) (6.5,-2)};
        \draw[very thick,orange]plot [smooth, tension=1] coordinates {(4.5,-2) (5.25,-2.3) (6.5,-2)};

        \draw[very thick,orange]plot [smooth, tension=1] coordinates {(6.5,1) (8,1.5) (8,1)};
        \draw[very thick,orange]plot [smooth, tension=1] coordinates {(6.5,1) (7,0.5) (8,1)};
        
        \draw[very thick, blue] plot [smooth cycle, tension=1] coordinates {(3,-2) (4.25,-1.5) (7,-2) (3.25,-2.5)};

        \filldraw (4.5,1) circle(.07);
        \filldraw (5,1) circle(.07);
        \filldraw (4.5,-2) circle(.07);
        \filldraw (4,1) circle(.07);

        \filldraw[fill=white] (8,1) circle(.07);
        \filldraw[fill=white] (10,-2) circle(.07);
        \filldraw[fill=white] (7,1) circle(.07);
        \filldraw[fill=white] (7,-2) circle(.07);
        \filldraw[fill=white] (11,1) circle(.07);
        \filldraw[fill=white] (11,-2) circle(.07);
    \end{tikzpicture}
    \caption{}
\end{subfigure} 

\begin{subfigure}{.4 \textwidth}
    \begin{tikzpicture}[scale=.5]
        \draw (0,0)--(2,2)--(12,2)--(10,0)--(0,0);
        \draw (0,-3)--(2,-1)--(12,-1)--(10,-3)--(0,-3);
        \draw[very thick] (1,1)--(4,1);
        \draw[very thick] (6,1)--(7,1);
        \draw[very thick] (6,-2)--(7,-2);
        \draw[very thick] (1,-2)--(4,-2);
        \draw[help lines,dashed] (4,1)--(4,-2);
        \draw[help lines,dashed] (6,1)--(6,-2);
        \draw[help lines,dashed] (7,1)--(7,-2);

        \draw[very thick, orange]plot [smooth cycle, tension=1.5] coordinates {(5,1) (7,1.5) (8,1) (6,.5)};
        \draw[very thick, blue] plot [smooth cycle, tension=1.2] coordinates {(5.5,-2) (6.5,-1.7) (7,-2) (6,-2.3)};
        \draw[very thick, blue] plot [smooth, tension=1.2] coordinates {(3,2) (4.5,1) (1,0)};
        \draw[very thick,orange]plot [smooth cycle, tension=1.2] coordinates {(5,-2) (8,-1.25) (10,-2) (7,-2.75)};

        \filldraw (4.5,1) circle(.07);
        \filldraw (5,1) circle(.07);
        \filldraw (5.5,-2) circle(.07);
        \filldraw (5,-2) circle(.07);

        \filldraw[fill=white] (8,1) circle(.07);
        \filldraw[fill=white] (10,-2) circle(.07);
        \filldraw[fill=white] (7,1) circle(.07);
        \filldraw[fill=white] (7,-2) circle(.07);
        \filldraw[fill=white] (11,1) circle(.07);
        \filldraw[fill=white] (11,-2) circle(.07);
    \end{tikzpicture} 
    \caption{}
\end{subfigure}

        \caption{The seven  pictures illustrate  the possible locations (schematically) of the paths of steepest descent and ascent leaving from the  four saddle points on the cycle $\mathcal C_1$ in the smooth region. In (a) and (b) we have all four saddle points on the first sheet, in pictures (c)--(f) we have three saddle points on the first sheet and  in picture (g) we have one point on the first sheet.  It is also possible that all four saddle points are on the second sheet, and in that case the picture is similar to that of (a) and (b) with the two sheets switched (but keeping the poles $a^{\pm 2}$ in place and slightly adjusting the contours accordingly).   Similarly, for the case of three saddle point on the second sheet. 
        All pictures can be reconstructed started from the picture in (a) by continuous deformations. For example, (b) can be obtained by moving the right most saddle point (and the orange contour) in (a) over the cycle $\mathcal C_1$, first passing  the branch point $x_1$ to the second sheet and then passing the branch point $x_2$ back to the first sheet to become the left most saddle point at (b). The pictures (c) and (d) can be obtained from (a) by moving the right most and the left most points respectively to the second sheet, etc.
        We did not check whether all configurations indeed occur and perhaps some cases can be excluded, but our arguments hold for any of the above configurations.}
        \label{fig:path_steep_descent_ascent}
    \end{center}
\end{figure}
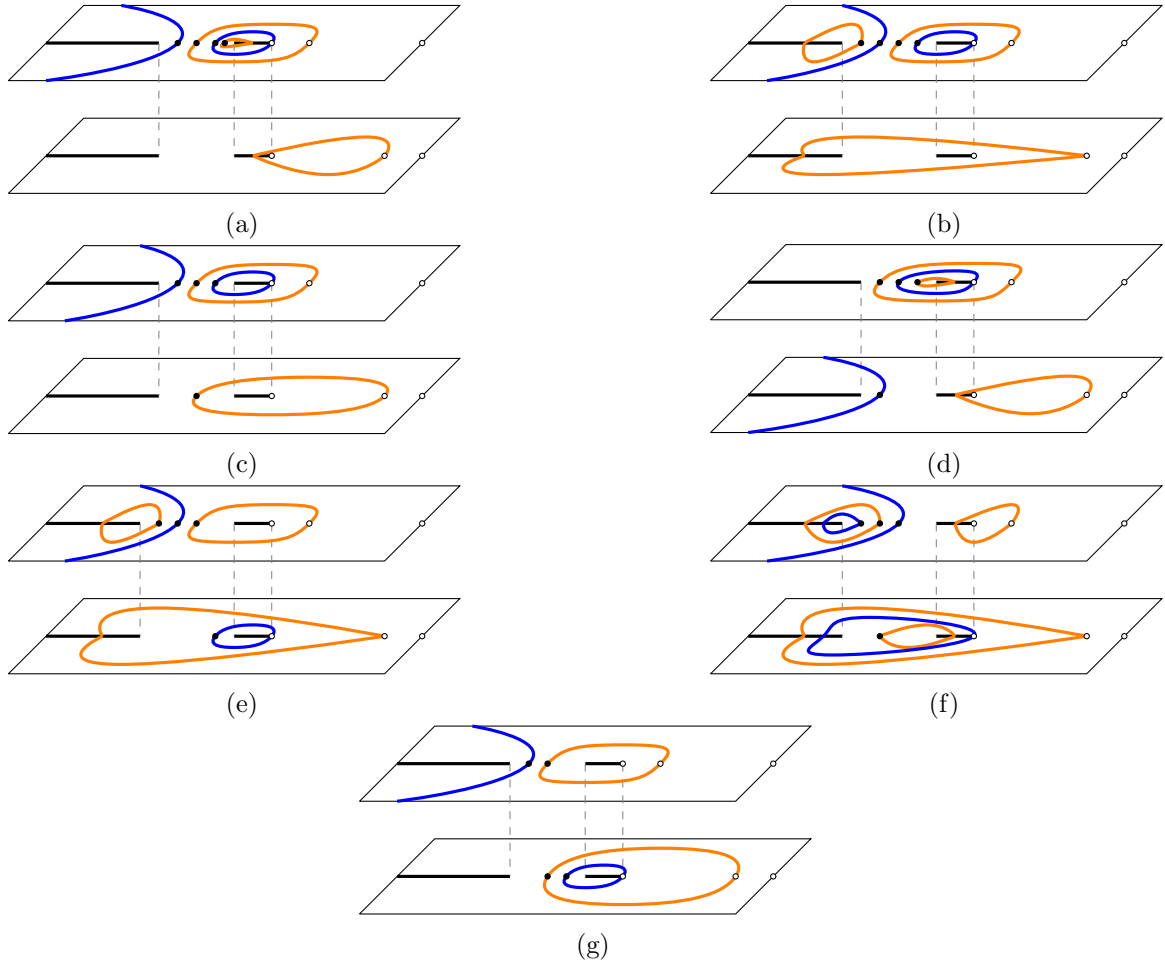

\subsubsection{Proof of Theorem \ref{thm:gas}} 

Now we are ready for the 
\begin{proof}[Proof of Theorem \ref{thm:gas}]
The starting point is the representation of the kernel after the preliminary deformation as given in \eqref{eq:correlationkernelmainperiodiceigenvaluesafterdeforming}. To prove the result, all that is needed is to show that the double integral tends to zero as $N \to \infty$. This is rather straightforward after one has realized that the contours of the preliminary deformation strongly resemble the paths of steepest descent and ascent for the saddle point $s_j$. Indeed, the two contours  $\tilde \gamma_1^{(1)}$ and $\gamma_1^{(2)}$ can be deformed to go through the saddle points $s_1$ and $s_3$ and follow the paths of steepest descent, and the contours $\gamma_2^{(1)}$ and $\gamma_2^{(2)}$ can be deformed to the path of steepest ascent ending in $z=(a^2)^{(1)}$ and $z=(1/a^2)^{(2)}$ respectively. During this deformation, no additional residues are being picked  up, and standard saddle point arguments show that there exists $c>0$ such that 
$$
\oint_{\tilde \gamma_1^{(1)} \cup \gamma_1^{(2)}} \oint_{\gamma_2^{(1)}\cup \gamma_2^{(2)}}= \mathcal O(\exp(-N c)), 
$$
as $N \to \infty$. This finishes the proof. \end{proof}

\subsection{Proof of Proposition \ref{prop:phiprime}} \label{sec:proofphi} 
Before we come to the proof of Proposition \ref{prop:phiprime} we need a few lemmas. We use the notation $\lfloor x \rfloor$ for the largest integer smaller than $x$.
\begin{lemma}
    There exists polynomials $p,\tilde p$ with real coefficients and of degree at most $\lfloor \tfrac{d}{2} \rfloor$, and polynomials $q,\tilde q$ with real coefficients, of degree  at most $\lfloor \tfrac{d-1}{2} \rfloor$  and $q(0)=\tilde q(0)=0$, such that 
    \begin{align}
        \nu(z)&=p(z)+q(z)(R(z))^{1/2}, \label{eq:formnu}\\
        \mu(z)&=\tilde p(1/z)+\tilde q (1/z) (R(z))^{1/2}, \label{eq:formmu}
    \end{align}
    where $R(z)$ is as in \eqref{eq:defR} and the square root $(R(z))^{1/2}$ is such that $(R(z))^{1/2}$ for $z>0$ on $\mathcal R_1$.
\end{lemma}
\begin{proof}
    From \eqref{eq:lambda_curve} and \eqref{eq:munuinE} we then find that 
    $$
        \nu(z)=p(z)+q(z)(R(z))^{1/2},
    $$    
    for some rational functions $p(z)$ and $q(z)$ with real coefficients. It remains to show that $p$ and $q$  are in fact polynomials in $z$ of said degree. 
    
    By computing the trace of $P_+^{(d)}(z)$ we have
    $$
      \Tr P_+^{(d)}(z)= \nu_1(z)+\nu_2(z)=2 p(z),
    $$
    and thus $p(z)$ is a polynomial. The degree of $\Tr P_+^{(d)}(z)$ can also be estimated from above. Indeed, for any  matrices $A_{j,1},A_{j,2}$ for $j=1, \ldots,d$ of the same dimensions such that $$A_{j,2}A_{j+1,2}=\mathbb O,\qquad j=1,\ldots, d-1,$$ we have that 
    $$
      \Tr \prod_{j=1}^d(A_{j,1}+z A_{j,2})
    $$
    is a polynomial of degree at most $\lfloor \tfrac d2\rfloor$. In the case of $P_+^{(d)}$, we have $A_{j,2}= c_j\begin{pmatrix} 0 & 1 \\  0 & 0 \end{pmatrix}$  for some constant $c_j$, and this shows that $ p(z)$ has degree at most $\lfloor \tfrac d2\rfloor$ as stated. 

    Finally, let us consider  $q(z)$.  We have
    $$
    \det P_+^{(d)}(z)= \nu_1(z)\nu_2(z)=p(z)^2+R(z)(q(z))^2.
    $$
    Since the left-hand is a polynomial of degree $d$, and $p(z)^2$ is a polynomial of degree at most $d$, $R(z)(q_2(z))^2$ is a polynomial of degree $d$. Hence,  the rational function $q$ must be a polynomial of degree at most $\lfloor \tfrac{d-1}{2}\rfloor$. Moreover, since $R(z)$ has a simple pole at $z=0$, the polynomial $q(z)$ must have a zero at $z=0$.

    The statement for $\mu$ follows in the same way. 
\end{proof}
\begin{lemma} \label{lem:maxmod}
    We have $|\lambda_1(z)|> |\lambda_2(z)|$,  $|\mu_1(z)|> |\mu_2(z)|$, and $|\nu_1(z)|> |\nu_2(z)|$  for $z\in\mathbb C\setminus\left((-\infty,x_1)\cup (x_2,0]\right)$.
\end{lemma}
\begin{proof}
    The proof is the same for all three cases, so we only prove that $|\mu_1(z)|>|\mu_2(z)|$. 
    To this end, we note that $\mu_2(z)/\mu_1(z)$ is analytic on $\mathbb C\setminus\left((-\infty,x_1]\cup [x_2,0]\right)$. It has a zero at $z=a^2$ and a possible pole at $z=0$. However, from \eqref{eq:formmu} it follows that the singularity at $z=0$ is removable. Moreover, $\mu_2(z)/\mu_1(z) \to 1$ as $z \to \infty$. From \eqref{eq:formmu} it also follows that $|\mu_2(z)/\mu_1(z)|=1$ for $z\in (-\infty,x_1) \cup (x_2,0)$. By the maximum modulus principle we must have either $|\mu_2(z)/\mu_1(z)|>1$ or $|\mu_2(z)/\mu_1(z)|<1$, for $z\in\mathbb C\setminus\left((-\infty,x_1)\cup (x_2,0]\right)$. Since $\mu_2(a^2)=0$, we conclude that  $|\mu_2(z)/\mu_1(z)|<1$.
\end{proof}
We also need the behavior of $\mu$ near the branch point at $\infty$. 
\begin{lemma}
   With 
    \begin{equation} \label{eq:irrelevant_constant}
    \Pi=\prod_{j=0}^{d-1} a(\sigma^j(x,y)) b(\sigma^j(x,y))
    \end{equation}
    we have
    \begin{equation}\label{eq:mu_asymptotic}
    \mu_{1}(z)=\Pi
    \left(
        1+ {\frac{a}{z^{1/2}}} 
        \left( 
            \sum_{j=0}^{d-1} a(\sigma^j(x,y)) \sum_{k=0}^{d-1} \frac{1}{a(\sigma^k(x,y))}
        \right)^{1/2}
        \right)+\mathcal O(z^{-1}), 
    \end{equation}
    and
    $$
    \mu_{2}(z)=\Pi
    \left(
        1-{\frac{a}{z^{1/2}}} 
        \left( 
            \sum_{j=0}^{d-1} a(\sigma^j(x,y)) \sum_{k=0}^{d-1} \frac{1}{a(\sigma^k(x,y))}
        \right)^{1/2}
        \right)+\mathcal O(z^{-1}), 
    $$
    as $z\to \infty$ along the positive real axis, and the square root is taken such that  $z^{1/2}>0$. 
\end{lemma}
\begin{proof}
    A simple computation gives
\begin{multline*}
 P_-^{(d)}(z)=\Pi 
 \left( 
     \begin{pmatrix} 
        1 & \sum_{j=0}^{d-1} a(\sigma^j(x,y)) \\ 
        0 & 1 
    \end{pmatrix} 
\right.\\
\left.+ \frac{a^2}{z} \sum_{k=0}^{d-1} 
    \begin{pmatrix} 
        1 & \sum_{j=0}^{k-1} a(\sigma^j(x,y)) \\ 
        0 & 1 
    \end{pmatrix} 
    \begin{pmatrix} 
        0 & 0\\ 
        a(\sigma_k(x,y))^{-1} & 0 
    \end{pmatrix}
    \begin{pmatrix} 
        1 & \sum_{j=k+1}^{d-1} a(\sigma^j(x,y)) \\ 
        0 & 1 
    \end{pmatrix}+ \mathcal O(z^{-2})
    \right),
\end{multline*}
as $z \to \infty$. Hence,
\begin{multline*}
\Tr P_-^{(d)}(z)= \Pi \left(
    2+ \frac{a^2}{z} \sum_{k=0}^{d-1}\sum_{j=0, j\neq k}^{d-1} \frac{a(\sigma^j(x,y))}{a(\sigma^k(x,y))} 
    +\mathcal O(z^{-2})
    \right)\\
    = \Pi \left(
    2+ \frac{a^2}{z} \left( \sum_{j=0}^{d-1} a(\sigma^j(x,y)) \sum_{k=0}^{d-1} \frac{1}{a(\sigma^k(x,y))} -d\right)
    +\mathcal O(z^{-2})
    \right),
\end{multline*}
as $z \to \infty$. Since $\det P_-^{(d)}(z)=\Pi^2 (1-a^2/z)^d$, we find 
$$
    \mu_{1,2}=\Pi
        \left(
            1\pm {\frac{a}{z^{1/2}}} 
            \left( 
                \sum_{j=0}^{d-1} a(\sigma^j(x,y)) \sum_{k=0}^{d-1} \frac{1}{a(\sigma^k(x,y))}
            \right)^{1/2}
            \right)+\mathcal O(z^{-1}), 
$$
as $z \to \infty$. It remains to determine whether $\mu_1$ or $\mu_2$ comes with the plus sign. 
Since $\mu_1(z)>\mu_2(z)$, we see that $\mu_1$ comes with the plus sign and $\mu_2$ with the minus sign.
\end{proof}

Now we are ready for the 
\begin{proof}[Proof of Proposition \ref{prop:phiprime}] 
    By \eqref{eq:formnu} we have 
    $$
        \nu_1(z)=p(z)+q(z) \sqrt{R(z)}, \qquad \nu_2(z)=p(z)-q(z) \sqrt{R(z)},
    $$
    with $ R(z)=a^2(z-x_1)(z-x_2)/z$ and the square root is taken so that $\sqrt{R(z)}>0$ for $z>0$. 
    Here $p(z)$ is a polynomial of degree at most  $\lfloor d/2\rfloor$ and $q$ is a polynomial of degree $\lfloor(d-1)/2\rfloor$ with a zero at $z=0$. 

     Observe that $\nu_1'(z)\nu_2(z)$ can be written as 
     \begin{multline}\label{eq:nu1primenu1}
     \nu_1'(z)\nu_2(z)=\left(p'(z)+q'(z) \sqrt{R(z)}+\frac{q(z)R'(z)}{2\sqrt{R(z)}}\right)(p(z)-q(z) \sqrt{R(z)})\\
     = \frac{r_1(z)+r_2(z) \sqrt{R(z)}}{z \sqrt{ R(z)}},
     \end{multline}
     where 
     $$r_1(z)=2zq'(z)p(z)R(z)+zq(z)p(z) R'(z)-2p'(z)q(z) z R(z),$$ 
     and
     $$r_2(z)=2z p'(z)p(z)-2zq'(z)q(z)R(z)+zq(z)^2 R'(z).$$ 
     Since $q(0)=0$ and $R'(z)$ has double pole at $z=0$, $r_1$ and $r_2$ are polynomials and $r_2(0)=0$. The degree of $r_1$ and $r_2$ is at most $d$. 
      By replacing $\sqrt{R(z)}$ by $-\sqrt{R(z)}$ in the derivation above we also find
    $$
    \nu_2'(z)\nu_1(z)=\frac{-r_1(z)+r_2(z) \sqrt{R(z)}}{z \sqrt{ R(z)}}.
    $$
    Therefore, we can write 
    $$
     2r_1(z)=\left(\nu_1'(z)\nu_2(z)-\nu_2'(z)\nu_1(z)\right)z\sqrt{R(z)},
    $$
    and
\begin{equation}\label{eq:relr2}
2r_2(z)=z\left(\nu_1'(z)\nu_2(z)+\nu_2'(z)\nu_1(z)\right)
\end{equation}
Since $\nu_2(z)$ has a zero of order $d$ at $z=a^{-2}$, this means that both $r_1$ and $r_2$ have a zero of order $d-1$ at $z=a^{-2}$. This implies that $\nu_1'(z)\nu_2(z)$ can be written as 
     $$
       \nu_1'(z)\nu_2(z)= d\frac{(z-a^{-2})^{d-1}\left(\gamma_1+\gamma_2z+  \gamma_3 z\sqrt {R(z)}\right)}{z \sqrt {R(z)}}
     $$
     and thus 
    \begin{equation} \label{eq:logder_nu_proof}
     \frac{\nu_1'(z)}{\nu_1(z)}=\frac{\nu_1'(z)\nu_2(z)}{{\nu_1(z)} {\nu_2(z)}}= d\frac{\gamma_1+\gamma_2z+  \gamma_3 z\sqrt {R(z)}}{(z-a^{-2})z \sqrt {R(z)}},
    \end{equation}
    where $\gamma_j\in \mathbb R$, for $j=1,2,3$,  are some real constants.

   By a similar reasoning,  one can show that
        \begin{equation}\label{eq:logder_mu_proof}
            \frac{\mu_1'(z)}{\mu_1(z)}=da^2\frac{z\tilde \gamma_1+\tilde \gamma_2+ \tilde  \gamma_3 \sqrt {R(z)}}{(z-a^2)z \sqrt {R(z)}},
        \end{equation}
    for some real parameters $\tilde \gamma_j$, for $j=1,2,3$.

The next step is to compute the values of the constants $\gamma_j, \tilde \gamma_j$ for $j=1,2,3$. 
   To this end,  add \eqref{eq:logder_nu_proof} and \eqref{eq:logder_mu_proof} to obtain 
   \begin{multline}\label{eq:llpin1}
    d\frac{\lambda_1'(z)}{\lambda_1(z)}=\frac{\mu_1'(z)}{\mu_1(z)}+\frac{\nu_1'(z)}{\nu_1(z)}
 =d\frac{(\gamma_1+ \gamma_2 z)(z-a^2)+a^2(\tilde \gamma_1 z+ \tilde \gamma_2 )(z-a^{-2})}{(z-a^2)(z-a^{-2})z \sqrt{R(z)}}\\
 +d\frac{\left(z\gamma_3(z-a^2)+a^2 \tilde \gamma_3(z-a^{-2}\right))\sqrt{R(z)}}{(z-a^2)(z-a^{-2})z \sqrt{R(z)}}.
   \end{multline}
   On the other hand, we easily compute from \eqref{eq:spectral_curve_introduction}--\eqref{eq:lambda_curve} that
   \begin{multline} \label{eq:llprin2}
    \frac{\lambda_1'(z)}{\lambda_1(z)}= \frac{\lambda_1'(z) \lambda_2(z)}{\lambda_1(z)\lambda_2(z)}=-\frac{z^2 R'(z)\left(\frac12 (1+a^2)(\alpha+1/\alpha)-\frac12 \sqrt{R(z)}\right)}{4 a^2 z(z-a^2)(z-a^{-2}) \sqrt{R(z)}}\\
    =-\frac{(z^2-1)\left( (1+a^2)(\alpha+1/\alpha)-\sqrt{R(z)}\right)}{ 2 z(z-a^2)(z-a^{-2}) \sqrt{R(z)}}
   \end{multline}
   where we used 
   $z^2 R'(z)=4a^2(z^2-1)$ in the last step. Comparing \eqref{eq:llpin1} and \eqref{eq:llprin2} leads to the following two equations:
   \begin{equation}\label{eq:gammacond1}
   (\gamma_1+ \gamma_2 z)(z-a^2)+a^2(\tilde \gamma_1 z+ \tilde \gamma_2 )(z-a^{-2})=-\frac{1}{2}(z^2-1) (1+a^2)(\alpha+1/\alpha),
   \end{equation}
and 
\begin{equation} \label{eq:gammacond2}
z\gamma_3(z-a^2)+a^2 \tilde \gamma_3(z-a^{-2})=\frac12(z^2-1).
\end{equation}
From \eqref{eq:gammacond2} we find 
\begin{equation}\label{eq:g1}
    \gamma_3= \tilde \gamma_3=\frac 12,
\end{equation}
and from \eqref{eq:gammacond1} we find 
\begin{equation}\label{eq:g2}
    \tilde \gamma_1=\gamma_2,\quad \tilde \gamma_2=\gamma_2
\end{equation}
and 
\begin{equation}\label{eq:g3}
    a^2\gamma_1+\gamma_2=- \frac12 (1+a^2)(\alpha+1/\alpha).
\end{equation}
Thus far, we have derived the first two identities in \eqref{eq:constants_gamma}. 
   
   The value of $\gamma_1$ can be computed by comparing the asymptotic expansion for the logarithmic derivative for $\mu_1$   from \eqref{eq:logder_mu_proof} and \eqref{eq:mu_asymptotic}, and  comparing the results. Indeed, from \eqref{eq:mu_asymptotic} we find 
   $$
   \frac{\mu_1'(z)}{\mu_1(z)}= -{\frac{a}{2z^{3/2}}} 
   \left( 
       \sum_{j=0}^{d-1} a(\sigma^j(x,y)) \sum_{k=0}^{d-1} \frac{1}{a(\sigma^k(x,y))}
   \right)^{1/2}+ \mathcal O(1/z^{2}),
   $$
   as $z\to \infty$, and from \eqref{eq:logder_mu_proof} we find, using $\gamma_1=\tilde \gamma_1$ and \eqref{eq:defR}, that 
  $$
   \frac{\mu_1'(z)}{\mu_1(z)}=\frac{da\gamma_1}{ z^{3/2}}+ \mathcal O(1/z^{2}),
   $$
   as $z\to \infty$. Therefore
   \begin{equation}\label{eq:g4}
   \gamma_1=-\frac12 \left( 
       \frac1d \sum_{j=0}^{d-1} a(\sigma^j(x,y))\frac 1d \sum_{k=0}^{d-1} \frac{1}{a(\sigma^k(x,y))}
   \right)^{1/2},
   \end{equation}
   which is the third identity in \eqref{eq:constants_gamma}.

   Finally, by substituting \eqref{eq:logder_nu_proof} and \eqref{eq:logder_mu_proof} using \eqref{eq:g1}, \eqref{eq:g2}, \eqref{eq:g3} and \eqref{eq:g4} into $\Phi'(z)$, and using analytic continuation to $\mathcal R$,  we obtain \eqref{eq:phiingam}. 
\end{proof}
\subsection{Proof of Lemma \ref{lem:relationsgamma}} \label{sec:proofgamma}
\begin{proof}[Proof of Lemma \ref{lem:relationsgamma}]
    The cycle condition \eqref{eq:cyclecondition} implies that (using \eqref{eq:phiingam}) 
    \begin{equation}\label{eq:cyclecondition_real}
    (1-\tau )a^2\int_{x_1}^{x_2} \frac{ x \gamma_1+\gamma_2}{(x-a^{2})x\sqrt{R(x)}} dx
    -\tau \int_{x_1}^{x_2} \frac{ \gamma_1+x\gamma_2}{(x-a^{-2}) x\sqrt{R(x)}} dx=0.
\end{equation}
    By a change of variable $x\mapsto 1/x$ we find (using \eqref{eq:defR})
    $$
    \int_{x_1}^{x_2} a^2
        \frac{x \gamma_1+\gamma_2}{(x-a^{2}) x\sqrt{R(x)}} dx
        =-\int_{x_1}^{x_2} \frac{ \gamma_1+x\gamma_2}{(x-a^{-2})x\sqrt{R(x)}} dx,
    $$
    and after substituting this into the first integral, \eqref{eq:cyclecondition_real} reduces to 
    $$
     \int_{x_1}^{x_2} \frac{ \gamma_1+x\gamma_2}{(x-a^{-2})x\sqrt{R(x)}} dx=0,
     $$
     after which the statement easily follows.
\end{proof}

\appendix

\section{Example:  torsion point of order six} \label{sec:order_six}

Let us now assume
\begin{equation}\label{eq:relationssix}
a^2=\frac{\alpha}{1+\alpha+\alpha^2}.
\end{equation}
Then, as discussed in Section \ref{sec:examplestorsion},  $(a^{-2},a^{-2})$ is a torsion point of order six. Here we will compute the spectral curves for $\mu$ and $\nu$, and  derive a degree eight equation for the boundary of the rough disordered region. We will also show, numerically, how the steepest descent/ascent path can be chosen when $(\tau,\xi)$ are in the center of the diamond. 
\subsection{The flow on the matrices}
 The linear flow on the elliptic curve is given already in \eqref{eq:flowsix}. The flow on the matrices  and their decomposition can then be traced giving:
\begin{multline}
    P^{(0)}_-(z)=\alpha
    \begin{pmatrix}
       1 & a \alpha^2\\
       \frac{a}{z \alpha^2} & 1
    \end{pmatrix}
    \mapsto
    P^{(1)}_-(z)=
    \begin{pmatrix}
       1 & {a \alpha^3}\\
       \frac{a }{z \alpha^3} & 1
    \end{pmatrix}
    \mapsto
    P^{(2)}_-(z)=\frac{1}{\alpha}
    \begin{pmatrix}
       1 & {a \alpha^2}\\
       \frac{a }{z \alpha^2} & 1
    \end{pmatrix}\\
    \mapsto
    P^{(3)}_-(z)=\frac{1}{\alpha}
    \begin{pmatrix}
       1 & a\\
       \frac{a}{z} & 1
    \end{pmatrix}
    \mapsto
    P^{(4)}_-(z)=
    \begin{pmatrix}
       1 & \frac{a}{\alpha}\\
       \frac{a \alpha}{z} & 1
    \end{pmatrix}
    \mapsto
    P^{(5)}_-(z)=\alpha
    \begin{pmatrix}
       1 & a\\
       \frac{a}{z} & 1
    \end{pmatrix}
    \end{multline}
    and 
    \begin{multline}
       P^{(0)}_+(z)=
       \begin{pmatrix}
          1 & az\\
          \frac{a}{\alpha^2 } & \frac{1}{\alpha^2}
       \end{pmatrix}
      \mapsto
       P^{(1)}_+(z)=
       \begin{pmatrix}
          1 & {a  \alpha z }\\
          \frac{a}{\alpha  } & 1
       \end{pmatrix}
      \mapsto
       P^{(2)}_+(z)=
       \begin{pmatrix}
          1 & {a\alpha^2 z}\\
          {a} & \alpha^2
       \end{pmatrix}\\
       \mapsto  P^{(3)}_+(z)=
       \begin{pmatrix}
          1 & {a\alpha^2 z}\\
          {a} & \alpha^2
       \end{pmatrix}
      \mapsto
       P^{(4)}_+(z)=
       \begin{pmatrix}
          1 & a\alpha z\\
          \frac{a}{\alpha} & 1
       \end{pmatrix}
      \mapsto
       P^{(5)}_+(z)=
       \begin{pmatrix}
          1 & az\\
          \frac{a}{\alpha^2 } & \frac{1}{\alpha^2}
       \end{pmatrix}.
    \end{multline}
From here we can compute $P_\pm ^{(d)}$ as in \eqref{eq:pplus} and \eqref{eq:pmin} and the spectral curves \eqref{eq:spectral_curve_nu_intro} and \eqref{eq:spectral_curve_mu_intro}.

The discriminant $R$ as in \eqref{eq:defR} can be written as 
$$
  R(z)=(a^{-2}-1)^2(a^2+1)^2+4a^2(z+1/z-a^2-a^{-2}).
$$
Straightforward computations give $\det P_+^{(d)}(z)= (a^2 z-1)^6$ and
$$
   \Tr P_+^{(d)}(z)=2+\frac{(1-6a^4+3a^8)z}{a^6}+2(2-3a^4)z^2+2a^6 z^3.
$$
The discriminant then becomes
$$
   \Tr P_+^{(d)}(z)-4 \det P_+^{(d)}(z)=a^{-12} z^2(1-3a^4+2a^6z)^2 (1-6a^4-3a^8+4a^6(z+1/z)),
$$
so that 
$$
   \nu(z)=1+\frac{(1-6a^4+3a^8)z}{2a^6}+(2-3a^4)z^2+a^6 z^3\pm \frac{z(1-3a^4+2a^6z)}{2a^4}(R(z))^{1/2}.
$$
Similarly, $\det P_-^{(d)}(z)=(z-a^2)^6/z^6$ and 
$$
\Tr P_-^{(d)}(z)=2 + (2 a^6z^{-3} + 4 z^{-2} - 6 a^4 z^{-2} + ((1 - 6 a^4 + 3 a^8) z^{-1})/a^6).
$$
The discriminant then becomes 
$$
\Tr P_-^{(d)}(z)-4 \det P_-^{(d)}(z)=a^{-12} z^{-2}(1-3a^4+2a^6z^{-1})^2 (1-6a^4-3a^8+4a^6(z+1/z)),
$$
so that
$$
   \mu(z)=1+\frac{(1-6a^4+3a^8)z^{-1}}{2a^6}+(2-3a^4)z^{-2}+a^6 z^3\pm \frac{z^{-1}(1-3a^4+2a^6z^{-1})}{2a^4}(R(z))^{1/2}.
$$
Note that $\mu(z)=\nu(1/z)$.

Now that we have $\mu$ and $\nu$ we can compute the saddle point equation $\Phi'(z)dz=0$. We start with computing the logarithmic derivatives of $\mu$ and $\nu$:
\begin{equation}
    \frac{\mu'(z)}{\mu(z)}=\frac{-2+(-1/a^2+3a^2)z+3 a^2 (R(z))^{1/2}}{(z-a^2)z(R(z))^{1/2}}
\end{equation}
and
\begin{equation}
    \frac{\nu'(z)}{\nu(z)}=\frac{(-1/a^2+3a^2)-2z+3 a^2 z (R(z))^{1/2}}{(a^2z-1)z(R(z))^{1/2}}.
\end{equation}
From the above expressions we can read off the values for $\gamma_j$:
\begin{equation}
\label{eq:valuesgammasix} \gamma_1=-\frac{1}{6a^{4}}+\frac{1}{2}, \quad  \gamma_2=-\frac{1}{3a^2}, \qquad \gamma_3=\frac{1}{2}.
\end{equation}
This, together with \eqref{eq:saddlepointequations2}, allows us to compute the four saddle points as a function of the parameters $a$, $\tau$ and $\xi$. The expressions are rather long, and we omit them here. Instead, we will provide some numerical results in the next subsection.

\subsection{Contours of steepest descent/ascent}

We will plot the contours of steepest descent/ascent for $\Re \Phi$, with $\Phi$ as in \eqref{eq:defPhi}, for the special values 
$$ 
        a^2=\frac 13-\frac{1}{100}, \quad \tau=\frac 12, \qquad \xi=0.
$$
This is the midpoint of the Aztec diamond, where we indeed expect a smooth disordered region to appear.  Indeed, with this choice of parameters, we find four saddle points on the cycle $\mathcal C_1$.  Two of them are on the first sheet:
$$
    s_1=-1.97156, \qquad s_2= -0.833032,
$$
and the other two on the second sheet:
$$
     s_3= -0.507212, \qquad s_4=-1.20043.
$$
The branch points are at 
$$
  x_1=-2.01885, \qquad x_2=-0.495331.
$$
Observe that $s_1$ is close to $x_1$ and $s_3$ is close to $x_2$, which we found to be typical for any choice of parameters. This has the unfortunate consequence that in numerical illustrations the saddle points $s_1$ and $s_3$ are hard to distinguish from the branch points $x_1$ and $x_2$, respectively.

The contours of steepest descent/ascent leaving from the saddle points  locally coincide with the level lines of $\Im \Phi(z)$. The problem is that $ \Phi(z)$ has logarithmic terms making  $\Im \Phi(z)$ multi-valued and plotting the level sets $\Im \Phi(z)=\Im \Phi(s_j)$ does  not give the correct result. For this reason, we compute the vectorfield given by $(\Re \Phi', - \Im \Phi')$ and compute the streamlines using the function  \texttt{Streamplot} in Mathematica. The results are given in Figure \ref{fig:saddlesketch}.  

 From Figure \ref{fig:saddlesketch} one can see that the paths of ascent/descent are indeed  as illustrated schematically in Figure \ref{fig:path_steep_descent_ascent}(g). On the first sheet, the paths of steepest descent leaving from $s_1$ end up at $\infty$ and the paths of steepest ascent leaving from $s_2$ end up at $a^{2}$. On the second sheet, the paths of steepest descent leaving from $s_3$ end up at $0$ and the paths of steepest ascent leaving from $s_4$ end up at $a^{-2}$. The statement for $s_3$ is not immediately obvious from  Figure \ref{fig:saddlesketch}(b) and this is why we zoom in around $s_3$ as in Figure \ref{fig:saddlesketch}(c).

\begin{figure}
    \begin{center}
    \begin{subfigure}{1 \textwidth}
        \centering{
        \begin{overpic}[scale=.8]{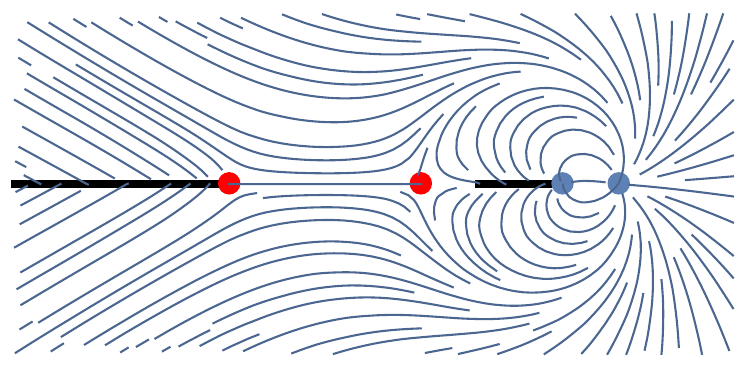}\end{overpic}\\
        \caption{Level lines for $\Im \Phi$. The paths of steepest descent  for $\Re \Phi$ from  $s_1$ connect to infinity on the first sheet. The paths of steepest ascent from $s_2$ end up in $a^2$.}
        \label{}}
    \end{subfigure}
    \vspace*{1cm}

    \begin{subfigure}{.45 \textwidth}
        \centering{
        \begin{overpic}[scale=.5]{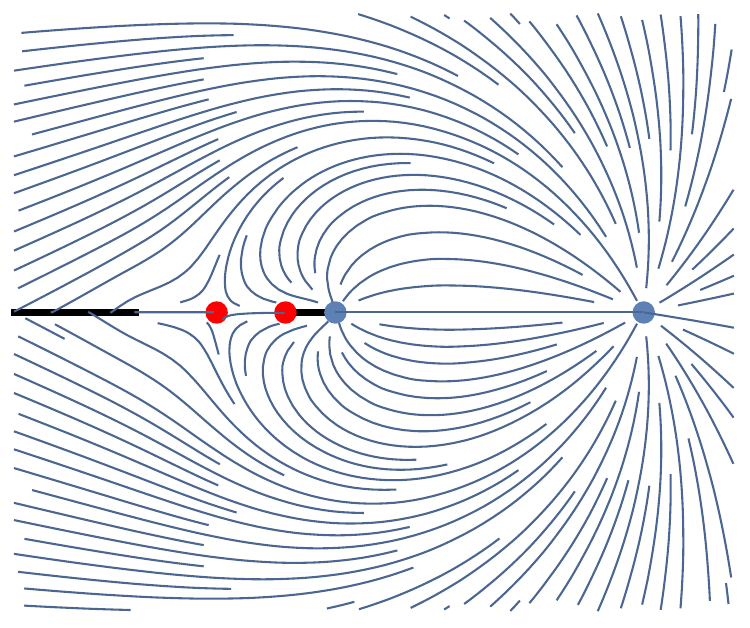} \end{overpic}
        \caption{Level lines for $\Im \Phi$. The paths of steepest ascent from $s_2$ end up in $a^{-2}$. To see the paths of steepest descent from $s_3$ we need to zoom in.}}
    \end{subfigure}\quad 
    \begin{subfigure}{.45 \textwidth}
        \begin{center}
        \begin{overpic}[scale=.6]{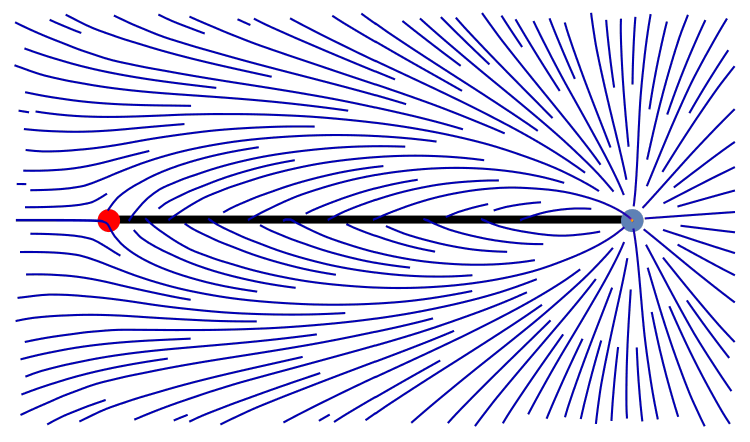} \end{overpic}
        \caption{Zooming in on the segment $(x_2,0)$ shows that the path of steepest descent for $\Re \Phi$ starting from the saddle point $s_3$ end in the origin.}
        \end{center}
    \end{subfigure}
    
    \caption{}
    \label{fig:saddlesketch}
\end{center}
\end{figure}
\subsection{Boundary of the rough disordered region}

The last result for the case of order six that we will present here, is an explicit expression for the boundary of the rough disordered region.  We follow the procedure indicated in Section \ref{sec:boundary} with $a$ and $\alpha$ related by \eqref{eq:relationssix} and values for $\gamma_j$, $j=1,2,3$ as in \eqref{eq:valuesgammasix}.  We start with \eqref{eq:saddlepointequations2}, square both sides and remove the additional factors $(z-a^2)(z-a^{-2})/z$. Then we obtain an equation that is polynomial in $z$ and of order four. The values of $(\tau,\xi)$ where the discriminant vanishes lead to a double saddle point and this will be the boundary of the liquid region. 

The discriminant has degree twelve in $\tau$ and $\xi$, and for general parameters $a$ (and $\alpha$ related by \eqref{eq:relationssix}) the expression is rather long. In order to obtain a shorter expression  it will be convenient to perform the following change of variables
$$
(\tau,\xi)=((q+v+1)/2,q/2).
$$
These coordinates change the parallellogram in the left panel  into the tilted square in the right panel of  Figure \ref{fig:regions}.

The discriminant has two factors.  The first factor, of degree four in $v$ and $q$,  reads
$$
(-1 + 9 a^4 + 9 q^2 - 9 a^4 q^2 - 9 a^4 v^2 + 9 a^8 v^2)^2.
$$
The zero set of this factor is a hyperbola that lies outside the Aztec diamond, and hence this factor does not contribute to the boundary for the rough region.  The second factor, of degree eight in $v$ and $q$, is the factor that defines the boundary for the rough disordered region:
\small{
\begin{multline*}
    0=16 - 336 a^4 + 1440 a^8 + 7776 a^{12} - 34992 a^{16} - 104976 a^{20} - 
 288 q^2 + 6336 a^4 q^2 - 45504 a^8 q^2 + 124416 a^{12} q^2 - \\
 209952 a^{16} q^2 + 419904 a^{20} q^2 + 1296 q^4 - 32400 a^4 q^4 + 
 242352 a^8 q^4 - 587088 a^{12} q^4 + 839808 a^{16} q^4 - \\
 629856 a^{20} q^4 + 23328 a^4 q^6 - 303264 a^8 q^6 + 769824 a^{12} q^6 - 
 909792 a^{16} q^6 + 419904 a^{20} q^6 + 104976 a^8 q^8 - \\
 314928 a^{12} q^8 + 314928 a^{16} q^8 - 104976 a^{20} q^8 - 72 v^2 + 
 1152 a^4 v^2 - 1224 a^8 v^2 - 43200 a^{12} v^2 + 75816 a^{16} v^2 + \\
 419904 a^{20} v^2 - 157464 a^{24} v^2 + 1296 q^2 v^2 - 
 20088 a^4 q^2 v^2 + 119880 a^8 q^2 v^2 - 527472 a^{12} q^2 v^2 + \\
 1283040 a^{16} q^2 v^2 - 997272 a^{20} q^2 v^2 + 472392 a^{24} q^2 v^2 - 
 5832 q^4 v^2 + 81648 a^4 q^4 v^2 - 367416 a^8 q^4 v^2 + \\
 863136 a^{12} q^4 v^2 - 833976 a^{16} q^4 v^2 + 734832 a^{20} q^4 v^2 - 
 472392 a^{24} q^4 v^2 + 52488 a^4 q^6 v^2 - 472392 a^8 q^6 v^2 + \\
 944784 a^{12} q^6 v^2 - 524880 a^{16} q^6 v^2 - 157464 a^{20} q^6 v^2 + 
 157464 a^{24} q^6 v^2 + 81 v^4 - 1215 a^4 v^4 - 3483 a^8 v^4 + \\
 79461 a^{12} v^4 - 2349 a^{16} v^4 - 750141 a^{20} v^4 + 570807 a^{24} v^4 - 59049 a^{28} v^4 - 1458 q^2 v^4 +\\ 21870 a^4 q^2 v^4 - 
 158922 a^8 q^2 v^4 + 867510 a^{12} q^2 v^4 - 1963926 a^{16} q^2 v^4 + 
 1418634 a^{20} q^2 v^4 - 301806 a^{24} q^2 v^4 \\+ 118098 a^{28} q^2 v^4 + 6561 q^4 v^4 - 98415 a^4 q^4 v^4 + 452709 a^8 q^4 v^4 - 
 387099 a^{12} q^4 v^4 - 610173 a^{16} q^4 v^4 \\
 + 964467 a^{20} q^4 v^4 - 
 269001 a^{24} q^4 v^4 - 59049 a^{28} q^4 v^4 + 5832 a^8 v^6 -
 52488 a^{12} v^6 - 128304 a^{16} v^6  \\
 + 734832 a^{20} v^6 - 
 717336 a^{24} v^6 + 157464 a^{28} v^6 + 52488 a^8 q^2 v^6 - 
 472392 a^{12} q^2 v^6 + 944784 a^{16} q^2 v^6 \\- 524880 a^{20} q^2 v^6 - 
 157464 a^{24} q^2 v^6 + 157464 a^{28} q^2 v^6 + 104976 a^{16} v^8 - 
 314928 a^{20} v^8 + 314928 a^{24} v^8 - 104976 a^{28} v^8.
\end{multline*}
}

\normalsize
For $\alpha=1$ (and thus $a^2=1/3$) this can be reduced to 
$$
0=(3 q^2 + v^2)^3 (-3 + 12 q^2 + 4 v^2).
$$
The first factor is only zero for $(q,v)=(0,0)$, and  what is left is the boundary for the smooth disordered region. The second factor is an ellipse. 

 Finally, for $\alpha\downarrow 0$ (and hence $a \downarrow 0$ simultaneously) the curve reduces to 
 $$ 
    0=(1 - 9 q^2)^2 (4 - 9 v^2)^2,
 $$
 which gives a rectangular shape. In this case, there is no rough disordered region, but only a frozen region and a smooth disordered region.
 
 \begin{figure}[t]
    \begin{center}
    \includegraphics[scale=.4,angle=90]{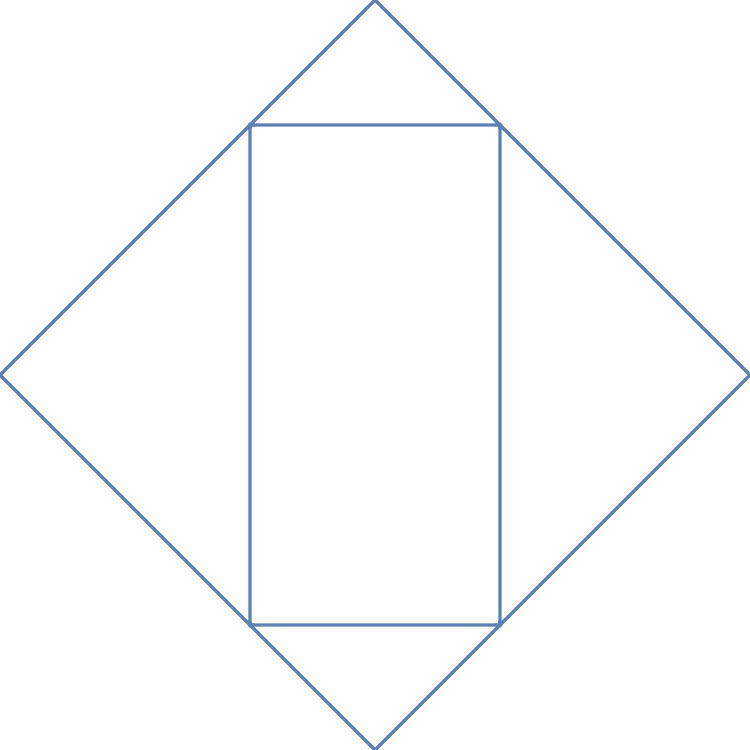} \qquad  \includegraphics[scale=.4,angle=90]{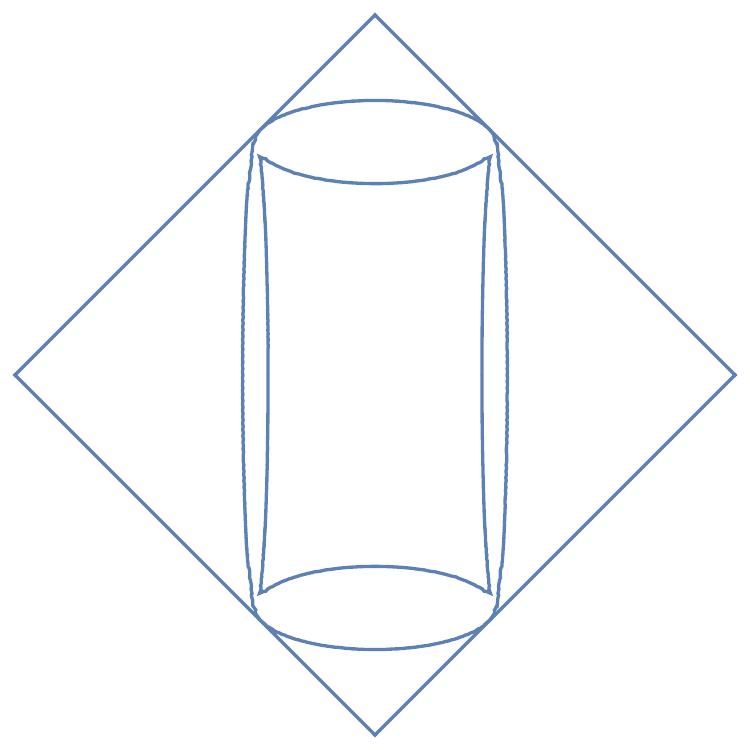}
    \vspace*{1cm}

    \includegraphics[scale=.4,angle=90]{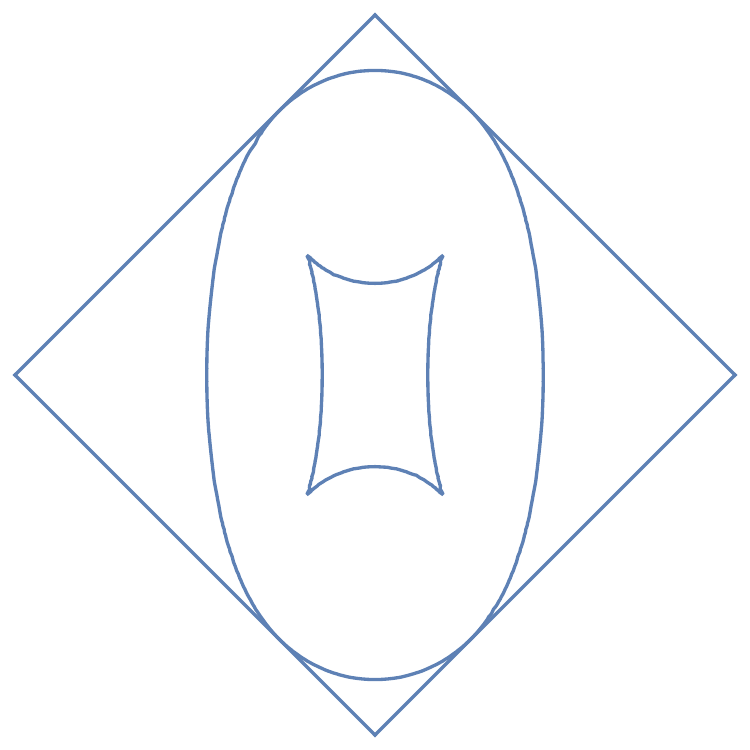} \qquad  \includegraphics[scale=.4,angle=90]{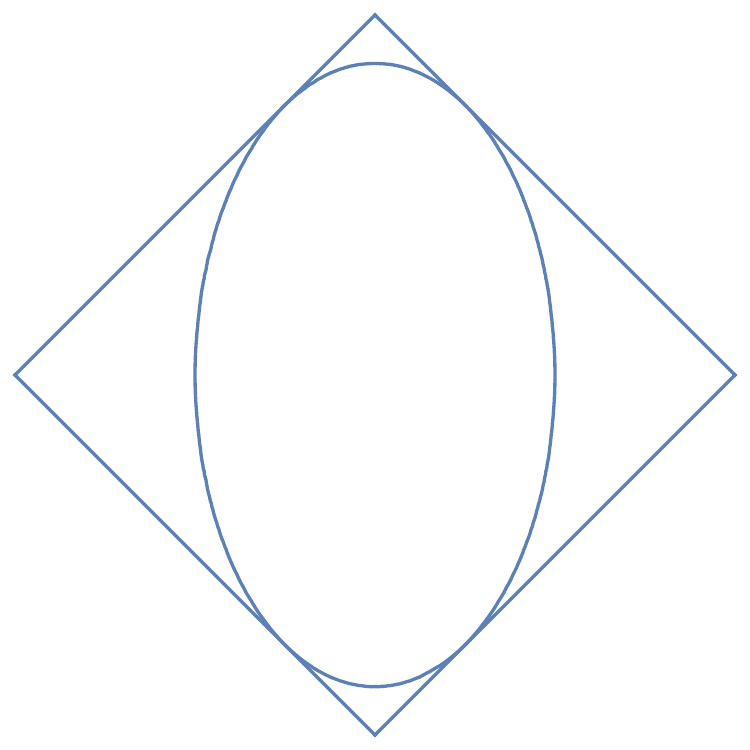}
    \end{center}
    \caption{The boundary of the rough disordered regions  in the $(v,q)$ plane  for the values $a=0$, $a=0.4$, $a=.55$ and $a=\frac 13 \sqrt{3}$. For $a=0$, the rough disordered region has disappeared. For $a=\frac13 \sqrt{3}$ the smooth disordered region has disappeared. }
    \label{fig:examplescurvesix}
 \end{figure}

 In Figure \ref{fig:examplescurvesix} we have plotted the boundary of the rough disordered region for several particular values of $a$. 

 \section{Computation of torsion points} \label{sec:divisionpolynomials}

In Section \ref{sec:examplestorsion} we gave a few examples of particular choices for the parameters $\alpha$ and $a$ such that $(1/a^2,1/a^2)$ is a torsion point. Here we will indicate how one can find such examples by recalling the notion of division polynomials. This is a standard construction for finding the torsion subgroups of the elliptic curve. We will base our discussion on \cite[p. 105]{Silv}.

First, let us introduce new variables
$$ Y=\tfrac y 2 {(a+1/a)(\alpha+1/\alpha)}, \qquad X=x,
$$
 and rewrite \eqref{eq:elliptic_aztec} as
$$
Y^2=X^3+\left((a+1/a)(\alpha+1/\alpha)^2/2-a^2-1/a^2\right)X^2+X.
$$
In the new variables, we ask for what choices of $\alpha$ and $a$ we have that $(1/a^2,\frac{(a+1/a)(\alpha+1/\alpha)}{2a^2})$ is a torsion point. With the same notation as in  \cite[p. 105]{Silv} we define
$$
\begin{cases}
    a_2=-a^2-1/a^2+\frac14 (a+1/a)^2(\alpha+1/\alpha)^2,\\
    a_4=1,\\
    b_2=4a_2,\\
    b_4=2,\\
    b_8=-1,
\end{cases}
$$
and the remaining parameters $a_1=a_3=a_6=b_6=0$. Then we define 
$$
\begin{cases}
    \psi_1=1,\\
    \psi_2=2Y,\\
    \psi_3=3X^4+b_2X^3+3b_4X^2+b_8,\\
    \psi_4=\psi_2\left(2X^6+b_2 X^5 +5 b_4 X^4+10b_8X^2+b_2b_8 X+b_4 b_8\right),
\end{cases}
$$
and $\psi_k$ with $k \geq 5$ recursively using:
$$
\begin{cases}
    \psi_{2m+1}= \psi_{m+2}\psi_m^3-\psi_{m-1}\psi_{m+1}^3,& m\geq 2,\\
    \psi_2\psi_{2m} =\psi_{m-1}^2 \psi_m \psi_{m+2}-\psi_{m-1}\psi_{m} \psi_{m+1}^2, & m \geq 3.
\end{cases}
$$ 
The torsion subgroup of order $m$ consists of all zeros of $\psi_m$, which are called division polynomials. 

Note that we are not looking for the entire subgroup, but for situations where 
$$
    \left(1/a^2,\frac{(a+1/a)(\alpha+1/\alpha)}{2a^2}\right)
$$
is a torsion point. After substituting this point into  $\psi_m$ we find a rational function in $a$ and $\alpha$. This rational function will have several zeros, but we are only interested in the zeros that satisfy $0<\alpha<1$ and $0<a \leq 1$. For instance, for $m=3$ we find the following equation 
$$
-\frac{(1+a^2)^2 (-1-\alpha+a^2 \alpha-\alpha^2) (1-\alpha+a^2 \alpha+\alpha^2)}{a^8 \alpha^2}=0.
$$
This equation has no solutions such that  $0<\alpha<1$ and $0<a \leq 1$, and thus a third order torsion point cannot occur.

With the help of a computer code, we found the following equations that give proper solutions such that we have a torsion point of order $m=4, \ldots,8$:
\begin{eqnarray*}
    m=4:& a=1,\\
    m=5:& 0=-a^4 + \alpha - a^2 \alpha + \alpha^2 - 2 a^2 \alpha^2 - 
    2 a^4 \alpha^2 + \alpha^3 + a^2 \alpha^3 - 3 a^4 \alpha^3 \\ & + 
    a^6 \alpha^3 + \alpha^4  2 a^2 \alpha^4 - 
    2 a^4 \alpha^4 + \alpha^5 - a^2 \alpha^5 - a^4 \alpha,\\
    m=6:& 0=(1+\alpha+\alpha^2) a^2-\alpha\\
    m=7:& 0=a^4 + a^4 \alpha - a^6 \alpha + a^8 \alpha - a^{10} \alpha + 
    5 a^4 \alpha^2 + 2 a^6 \alpha^2 - a^8 \alpha^2 - \alpha^3\\ &  + 
    3 a^2 \alpha^3 - a^4 \alpha^3 + 5 a^6 \alpha^3 - 
    4 a^8 \alpha^3 - 2 a^{10} \alpha^3 - \alpha^4 + \\ & 
    4 a^2 \alpha^4 + 5 a^4 \alpha^4 + 12 a^6 \alpha^4 - 
    5 a^8 \alpha^4 - \alpha^5 - a^2 \alpha^5 + 12 a^4 \alpha^5\\ & - 
    7 a^8 \alpha^5 - 3 a^{10} \alpha^5 - \alpha^6 + 
    2 a^2 \alpha^6 + 17 a^4 \alpha^6 + 7 a^8 \alpha^6 \\ &  - 
    6 a^{10} \alpha^6 + a^{12} \alpha^6 - \alpha^7 - a^2 \alpha^7 + 
    12 a^4 \alpha^7 - 7 a^8 \alpha^7 - 
    3 a^{10} \alpha^7 - \alpha^8 + 4 a^2 \alpha^8 \\ & + 
    5 a^4 \alpha^8 + 12 a^6 \alpha^8 - 
    5 a^8 \alpha^8 - \alpha^9 + 3 a^2 \alpha^9 - a^4 \alpha^9 + 
    5 a^6 \alpha^9 - 4 a^8 \alpha^9 - 2 a^{10} \alpha^9 \\ &+ 
    5 a^4 \alpha^{10} + 2 a^6 \alpha^{10} - a^8 \alpha^{10} + 
    a^4 \alpha^{11} - a^6 \alpha^{11} + a^8 \alpha^{11} - 
    a^{10} \alpha^{11} + a^4 \alpha^{12},\\
    m=8 :& 0=a - \alpha + a^2 \alpha + a \alpha^2.
\end{eqnarray*}
Each term on the right-hand side is a factor of $\psi_m$.
\subsection*{Acknowledgements}

The authors are grateful to B. Poonen for directing them to \cite{Silv}.  Figure \ref{fig:sample} was plotted using a code that was kindly provided to us by S. Chhita. We thank T. Berggren and M. Bertola for helpful discussions.

A. Borodin was partially supported by the NSF grant DMS-1853981, and the Simons Investigator program. M. Duits was partially supported by the Swedish Research Council (VR),   grant no 2016-05450 and grant no. 2021-06015, and the European Research Council (ERC), Grant Agreement No. 101002013. 

\subsection*{Data availablity statement}
Data sharing not applicable to this article as no datasets were generated or analysed during the current study.

\end{document}